\documentclass[11pt]{article}
\usepackage{algorithm}
\usepackage{algorithmic}
\usepackage[algo2e]{algorithm2e}
\RestyleAlgo{ruled}
\usepackage{amsthm,amsmath,amssymb}
\usepackage{natbib}
\usepackage{multirow}
\usepackage[pdftex]{graphicx}
\usepackage{subfigure}
\usepackage{makecell}
\usepackage{booktabs}
\usepackage{array}
\usepackage{fullpage}
\usepackage{url}
\usepackage{bm}
\usepackage{smile}
\usepackage{mathtools}
\usepackage{wrapfig}
\usepackage{lipsum}
\usepackage{mathrsfs}
\usepackage{dsfont}
\usepackage{titling}
\usepackage{epstopdf}
\usepackage{bbm}
\usepackage{relsize}
\usepackage{smile}
\usepackage{natbib}
\usepackage{multirow}
\usepackage{color}
\usepackage{mathtools}
\usepackage{caption}

\usepackage{natbib}
\usepackage{multirow}

\usepackage[usenames,dvipsnames,svgnames,table]{xcolor}
\usepackage[colorlinks=true,
            linkcolor=blue,
            urlcolor=blue,
            citecolor=blue]{hyperref}

\providecommand{\norm}[1]{\|#1\|}

\newcommand{\sign}[1]{\text{sign}(#1)}

\newcommand\independent{\protect\mathpalette{\protect\independenT}{\perp}}
\def\independenT#1#2{\mathrel{\rlap{$#1#2$}\mkern2mu{#1#2}}}
\usepackage{comment}
\usepackage{tikz}
\tikzset{
	declare function={
		normcdf(\x,\m,\s)=1/(1 + exp(-0.07056*((\x-\m)/\s)^3 - 1.5976*(\x-\m)/\s));
		zeroone(\x)= (\x<=-2) * (\x*\x + 6*\x + 8)   +
		and(\x>-2, \x<=1) * (2 - \x - \x*\x)     +
		and(\x>1,  \x<=2) * (6 - 8*\x + 2*\x*\x) +
		(\x>2) * (-10 + 6*\x - \x*\x);
	}
}

\usepackage{mathtools}

\DeclarePairedDelimiter{\floor}{\lfloor}{\rfloor}
\usepackage{enumitem}

	\makeatletter
	\newcommand{\mypm}{\mathbin{\mathpalette\@mypm\relax}}
	\newcommand{\@mypm}[2]{\ooalign{%
			\raisebox{.1\height}{$#1+$}\cr
			\smash{\raisebox{-.6\height}{$#1-$}}\cr}}
	\makeatother

	\def\beq{\begin{equation}\begin{aligned}[b]}
	
	\def\eeq{\end{aligned}\end{equation}}

\def\boxit#1{\vbox{\hrule\hbox{\vrule\kern6pt  \vbox{\kern6pt#1\kern6pt}\kern6pt\vrule}\hrule}}

\begin{document}

\title{Nonregular and Minimax Estimation of Individualized Thresholds in High Dimension with Binary Responses}

\author{Huijie Feng\thanks{Department of Statistics and Data Science, Cornell University, Ithaca, NY 14850, USA; e-mail: \texttt{hf279@cornell.edu}.}~~~~~Yang Ning\thanks{Department of Statistics and Data Science, Cornell University, Ithaca, NY 14850, USA; e-mail: \texttt{yn265@cornell.edu}.}~~~~~Jiwei Zhao\thanks{Department of Biostatistics, State University of New York at Buffalo, Buffalo, NY 14214, USA; e-mail: \texttt{zhaoj@buffalo.edu}.}
}
\date{\today}

\maketitle

\vspace{-0.3in}

\begin{abstract}
Given a large number of covariates $\bZ$, we consider the estimation of a high-dimensional parameter $\btheta$ in an individualized linear threshold $\btheta^T\bZ$ for a continuous variable $X$, which minimizes the disagreement between $\sign{X-\btheta^T\bZ}$ and a binary response $Y$. While the problem can be formulated into the M-estimation framework, minimizing the corresponding empirical risk function is computationally intractable due to discontinuity of the sign function. Moreover, estimating $\btheta$ even in the fixed-dimensional setting is known as a nonregular problem leading to nonstandard asymptotic theory. To tackle the computational and theoretical challenges in the estimation of the high-dimensional parameter $\btheta$, we propose an empirical risk minimization approach based on a regularized smoothed loss function. The Fisher consistency of the proposed method is guaranteed as the bandwidth of the surrogate smoothed loss function is shrunk to 0 and the resulting empirical risk function is generally non-convex. The statistical and computational trade-off of the algorithm is investigated. Statistically, we show that the finite sample error bound for estimating $\btheta$ in $\ell_2$ norm is $(s\log d/n)^{\beta/(2\beta+1)}$, where $d$ is the dimension of $\btheta$, $s$ is the sparsity level, $n$ is the sample size and $\beta$ is the smoothness of the conditional density of $X$ given the response $Y$ and the covariates $\bZ$. The convergence rate is nonstandard and slower than that in the classical Lasso problems. Furthermore, we prove that the resulting estimator is minimax rate optimal up to a logarithmic factor. The Lepski's method is developed to achieve the adaption to the unknown sparsity $s$ and smoothness $\beta$. Computationally, an efficient path-following algorithm is proposed to compute the solution path. We show that this algorithm achieves geometric rate of convergence for computing the whole path. Finally, we evaluate the finite sample performance of the proposed estimator in simulation studies and a real data analysis from the ChAMP (Chondral Lesions And Meniscus Procedures) Trial.


\end{abstract}

\section{Introduction}



In this paper, we consider the problem of estimating a threshold for a continuous variable $X$ in order to predict a binary response $Y$, which arises in a wide range of applications such as clinical diagnostics, signal processing, personalized medicine and econometrics. To account for the heterogeneity of populations, it is often preferred to construct an individualized threshold based on a large number of covariates $\bZ\in\RR^d$ often easily collected in many modern applications. To maintain interpretablity of the threshold and overcome the curse of dimensionality, practitioners usually assume that the threshold can be well approximated by some prespecified function of $\bZ$ up to some unknown parameters. In this paper, we focus on the estimation of the linear threshold $\btheta^T\bZ$, where $\btheta$ is the unknown parameter of interest.

Formally, let $\{(x_i,y_i,\bz_i)\}^n_{i = 1}$ be $n$ i.i.d. copies of $(X,Y,\bZ)$ that follows an unknown distribution $P = P(X,Y,\bZ)$. Our goal is to estimate the optimal $\btheta \in \RR^d$ that minimizes the disagreement between $Y$ and $\sign{X - \btheta^T\bZ}$. For this purpose, a natural formulation is
\beq\label{MCID_linear}
\btheta^{*} = \argmin_{\btheta} \Big\{w(-1)\PP_P(X\geq \btheta^T\bZ,Y =-1) + w(1)\PP_P(X<  \btheta^T\bZ,Y =1)\Big\},
\eeq
where $w(\cdot)$ is a prespecified weight function and $\btheta^* \in \RR^d$ is a high-dimensional parameter which can be also viewed as a functional of $P$. We add the subscript $P$ in (\ref{MCID_linear}) to indicate the probability measure under the distribution $P$. Hereafter, we will omit the subscript $P$ for simplicity. Throughout the paper, we assume $\btheta^*$ exists and is unique so that the estimation problem is well posed. In the following, we present several motivating examples for the problem (\ref{MCID_linear}).

\begin{example}[Covariate-adjusted Youden index]\label{exp_2}
	Youden's J statistic \citep{youden1950index} is one of the most important tools to evaluate the performance of a dichotomous test based on receiver operating characteristic (ROC) curve. In the context of clinical diagnostics, let $X$ denote a continuous biomarker and $Y$ denote the disease status ($Y=1$ if a patient is diseased and $Y=-1$ if a patient is healthy). To determine the status of a patient, a clinically meaningful diagnostic procedure is based on whether the biomarker is greater than a given threshold $c$, where the value of threshold will largely affect the diagnostic accuracy. Recently, \cite{xu2014model} among others showed that the diagnostic accuracy can be greatly improved by using individualized thresholds based on patient-specific features $\bZ$. From a practical perspective, understanding which and how covariates $\bZ$ affect the diagnosis itself is of great interest. Specifically, let $c_{\btheta}(\bZ) = \btheta^T\bZ$ be the covariate-adjusted threshold under the linearity assumption. In this scenario, the Youden index is defined as
	\[
	  J = \max_{\btheta} \{ \text{sen}(c_{\btheta}(\bZ)) + \text{spe}(c_{\btheta}(\bZ)) -1 \},
	\]
	where $\text{sen}(c_{\btheta}(\bZ))=\PP(X\geq c_{\btheta}(\bZ)|Y =1)$ and $\text{spe}(c_{\btheta}(\bZ)) = \PP(X < c_{\btheta}(\bZ)|Y =-1)$ are the sensitivity and specificity of the test, respectively. Thus, estimating the optimal threshold $c^*(\bz)=\btheta^{*T}\bz$ that maximizes the Youden index is equivalent to estimating $\btheta^*$ defined in (\ref{MCID_linear}) with weight function $w(y) = 1/\PP(Y = y)$.
	
\end{example}

\begin{example}[One-bit compressed sensing]\label{exp_1}
	In classical compressed sensing, the objective is to recover the signal $\btheta^* \in \RR^d$ from $n < d$ linear measurements $\tilde{y}_i = \bz_i^T\btheta^*$, where $\bz_i$ is a measurement vector. However, in practice the measurements are always  discretized prior to further digital processing. Under an extreme quantization scenario (one-bit compressed sensing), the measurements are quantized into a single bit via the sign of $x_i - \bz_i^T\btheta^*$, where $x_i$ is a measurement assumed with a positive sign \citep{boufounos20081}.  Under this formulation, the goal of one-bit compressed sensing is to recover the signal $\btheta^*$ from the measurements
	\[y_i = \sign{x_i - \bz_i^T\btheta^*},\]
	which can be viewed as the ``noiseless" version of the problem  (\ref{MCID_linear}), because in this case $y_i$ is fully determined by $(x_i, \bz_i)$.
\end{example}

\begin{example}[Personalized medicine]
Many clinical researches find that patient response to the same treatment can be highly heterogeneous. Due to this observation, how to design the best treatment for each individual draws tremendous attention in recent years and is known as a fundamental problem in personalized medicine. In statistics community, a variety of statistical methods have been proposed to estimate optimal individualized treatment rules (ITR) based on the information from a large number of clinical variables. In this context, let $R$ denote a continuous outcome variable coded so that a higher value suggests a better condition, $\bW\in\RR^{d+1}$ denote a vector of baseline subject features, and $A\in \{-1, 1\}$ denote the assigned treatment. An ITR, $D(\cdot)$, is a map from $\RR^{d+1}$ into $\{-1, 1\}$ so that a patient presenting with $\bW = \bw$ is recommended to receive treatment $D(\bw)$. \cite{zhao2012estimating} showed that the optimal ITR $D^*$ can be defined as the minimizer of
$$
\EE\Big[\frac{R}{\pi(A|\bW)}I(A\neq D(\bW))\Big],
$$
where $\pi(A|\bW)=\PP(A=1|\bW)$ is known as the propensity score. In practice, simple ITR such as linear rules are usually more desirable for convenient interpretation \citep{qiu2018estimation}. In particular, if some domain knowledge can be used to determine the sign of any baseline variable in $D^*$ (say $W_1$ has a positive sign), the linear ITR can be written as $D(\bW)=\sign{W_1 - \btheta^T \bW_{-1}}$, where $\bW_{-1}$ denotes the baseline variable excluding $W_1$. Thus, estimating the optimal ITR reduces to the problem (\ref{MCID_linear}) with the weight $w(a)=R/\pi(a|\bW)$ for $a\in\{-1,1\}$.

\end{example}

\begin{example}[Linear binary response model and maximum score estimator]\label{exp_4}
	Consider the  model
	\begin{equation}\label{eqbinary}
	Y = \sign{\bW^T\bbeta^* + \epsilon},~\bW,\bbeta^*\in \RR^{d+1},	
	\end{equation}
	where the noise $\epsilon$ is not necessarily independent of $\bW$ but comes with a weaker assumption that $\text{Median}(\epsilon|\bW) = 0$. Write $\bbeta^* = (\beta^*_0,\dotso,\beta^*_d)^T$. In this case, $\bbeta^*$ is only identifiable up to an arbitrary scale factor, and one way to remove the ambiguity is by setting $|\beta^*_0| = 1$. Suppose $\beta_0 = 1$, and let $W = (X,-\bZ^T)^T$ and $\bbeta^* = (1,\btheta^{*T})^T$. Motivated by the fact that
	\begin{equation}\label{eqbinary2}
	\btheta^{*} = \argmin_{\btheta}\Big\{\PP(X\geq \btheta^T\bZ, Y =-1) + \PP(X<  \btheta^T\bZ, Y =1)\Big\},
	\end{equation}
	\cite{manski1975maximum,manski1985semiparametric} proposed the maximum score estimator for $\bbeta^*$ by maximizing an empirical score function.  The equation (\ref{eqbinary2}) can be treated as an unweighted case of (\ref{MCID_linear}).
\end{example}

Motivated by the above examples, in this paper we consider the problem of estimating $\btheta^*$ defined in (\ref{MCID_linear}) under a high-dimensional ``model-free'' regime, where the dimension $d$ is allowed to be much larger than $n$, and we do not impose any parametric modeling assumption on the joint distribution $P$ of $Y$, $X$ and $\bZ$. For instance, we do not assume the linear binary response model (\ref{eqbinary}) in Example \ref{exp_4} holds. Despite its practical importance and generality, estimating $\btheta^*$ has several difficulties from both computational and statistical perspectives.
Computationally,
a natural idea to estimate $\btheta^*$ is via empirical risk minimization. However, the empirical counterpart of (\ref{MCID_linear}) is computationally intractable. To see this, we can rewrite (\ref{MCID_linear}) as
\beq\label{MCID_linear2}
\btheta^{*} = \argmin_{\btheta} R(\btheta),~~\textrm{where}~~R(\btheta)=\EE \bigg[w(Y)L_{01}\bigg( Y(X - \btheta^T\bZ)  \bigg)  \bigg],
\eeq
and $L_{01}(u)=\frac12(1-\sign{u})$ is the 0-1 loss. The corresponding emipirical risk function is generally NP-hard to minimize.
Statistically,  under the ``model-free'' regime, $\btheta^*$ is defined implicitly as a functional of the underlying joint distribution $P$, i.e., $\chi(P)$. Although this formulation enjoys more generality and potentially allows for misspecification, as a price to pay, the classical plug-in estimator say $\chi(\hat P)$ in the semiparametric literature \citep{bickel1993efficient} is infeasible, because the initial estimator $\hat P$ of $P$ may have slow rate of convergence and is even inconsistent in high-dimensional setting. Moreover, as shown in \cite{kim1990cube}, the estimator that minimizes the empirical counterpart of (\ref{MCID_linear2}) (assuming it is computable) has a nonstandard rate of convergence even in fixed dimension. Given the above reasons, constructing a computationally efficient estimator of $\btheta^*$ and analyzing its statistical properties in high dimension is a challenging problem.


We wish to highlight that the problem is closely related to but very different from the standard classification task, in which the ultimate goal is to accurately predict $Y$ based on $X$ and $\bZ$. In fact, most existing classification methods would fail to recover $\btheta^*$. To explain this,
under the standard classification setting, many popular methods seek to minimize the empirical loss w.r.t a convex upper bound of the 0-1 loss, in order to circumvent the computational problem.
For instance, Adaboost minimizes the exponential loss $(\exp(-x))$, and Support Vector Machine (SVM) minimizes the hinge loss $(\max\{0,1-x\})$ or its variants.
It is well known that even with such surrogate loss functions, the minimizer of the population risk function agrees with the Bayes rule, implying Fisher consistency in classification. However, if the Bayes rule is nonlinear and we are only interested in the best linear classifier, the estimand of SVM or Adaboost does not generally coincide with $\btheta^*$. Consequently, these methods are not generally Fisher consistent for the purpose of estimating $\btheta^*$. The following toy example demonstrates such inconsistency.

{\bf A toy example.}
	Suppose $X\sim N(0,1)$, $Z$ is a random variable taking values in $\{0.5,5\}$ with equal probability and $Y = \sign{X- Z}$. It can be shown that under this simple noiseless setting, the parameter of interest $\btheta^*$ in (\ref{MCID_linear}) equals to $1$ for any arbitrary nonnegative weights, while the minimizers of the population risk w.r.t hinge loss and exponential loss do not coincide with $\btheta^*$. Specifically, direct calculation gives that $\nabla R_{\text{hinge}}(\btheta)|_{\btheta = 1}\approx -0.035 < 0$ and $\nabla R_{\text{exp}}(\btheta) |_{\btheta = 1}\approx -0.059 <0$, where the risk functions are as follows
	\beq
	&R_{\text{hinge}}(\btheta) = \EE[\max \{1 - Y(X - \btheta Z), 0\}],\\
	&R_{\text{exp}}(\btheta) = \EE[\exp(- Y(X - \btheta Z)].
	\eeq
	By the first order optimality condition, $\btheta=1$ is not a minimizer of $R_{\text{hinge}}(\btheta)$ and $R_{\text{exp}}(\btheta)$.
	Figure \ref{demo} depicts the shape of the population risk functions w.r.t. 0-1 loss (also (\ref{MCID_linear2}) with weight function $w(y) \equiv 1$), hinge loss and exponential loss. This confirms that minimizing the empirical risk w.r.t. hinge loss or exponential loss will not consistently estimate $\btheta^*$ in this example.
	\begin{figure}
	    \centering
	    \includegraphics[width=80mm,height = 60mm]{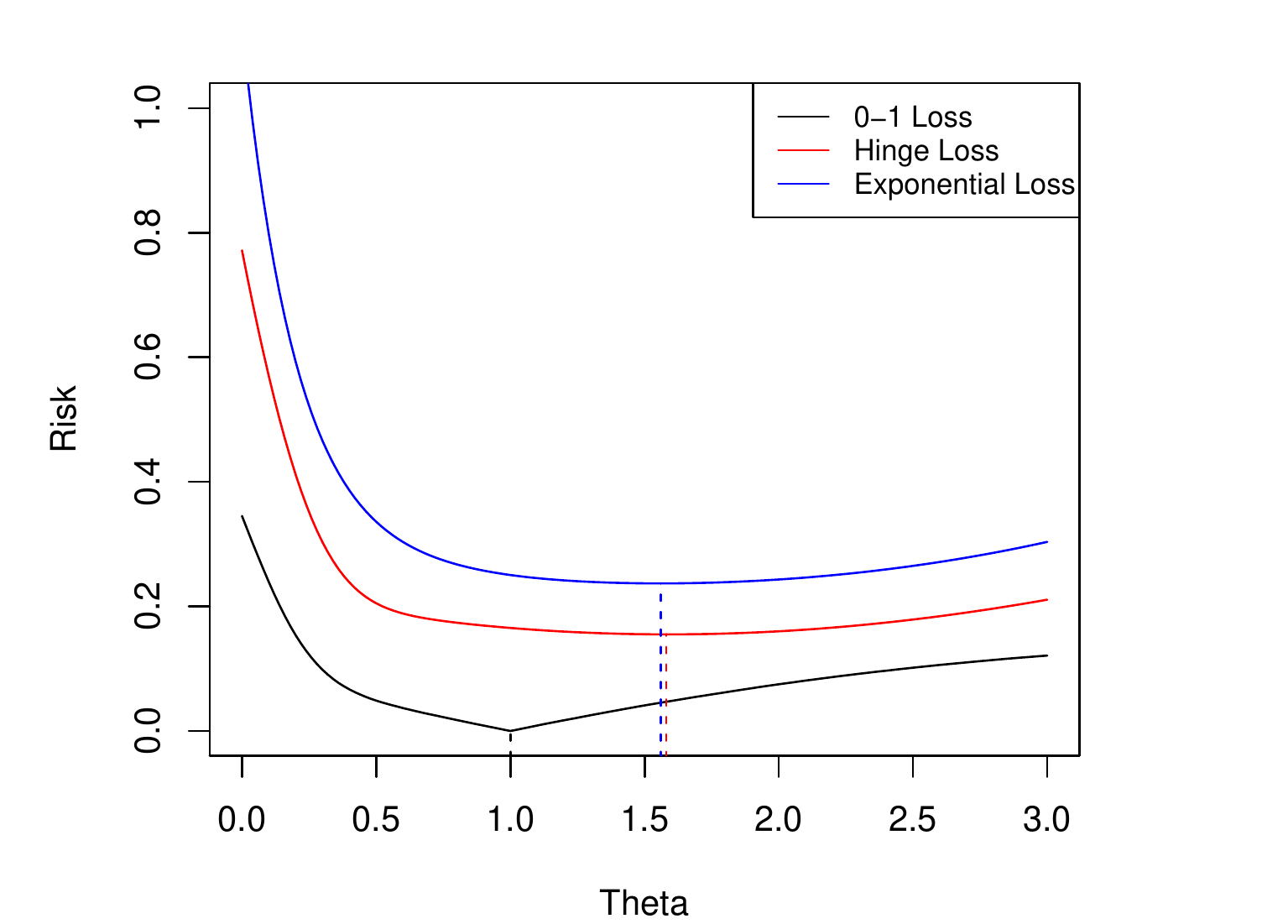}
	    \caption{Population risk functions w.r.t. 0-1 loss, hinge loss and exponential loss. The dashed lines depict the location of the minimizers.
		}   \label{demo}
	\end{figure}

\subsection{Main contributions}
The first contribution of this work is to propose a penalized M-estimation framework for estimating $\btheta^*$ in (\ref{MCID_linear}) or equivalently (\ref{MCID_linear2}). The estimator is defined as the minimizer of a $\ell_1$ penalized empirical risk function corresponding to a smoothed surrogate loss $L_{\delta,K}(\cdot)$, where $K$ is a kernel function and $\delta$ is the associated bandwidth parameter. By shrinking the bandwidth $\delta$ towards $0$, the loss function $L_{\delta,K}(\cdot)$ converges (pointwise) to the original 0-1 loss, and thus achieves the (asymptotic) Fisher consistency for estimating $\btheta^*$. In addition, $L_{\delta,K}(\cdot)$ is differentiable and its gradient has a simple closed form, so that we can leverage the smoothness of the surrogate loss to design  gradient-based algorithms for computation. When the dimension $d$ is fixed and much smaller than $n$, some earlier works proposed to use the smoothed loss function for inference purpose. For instance, \cite{horowitz1992smoothed} studied the inferential properties of the smoothed maximum score estimator under the linear binary response model. More recently, \cite{qiu2018estimation} studied the estimation and inference on the optimal linear ITR based on the ramp loss. To the best of our knowledge, such penalized smoothed loss function based methods have not been used in any related high-dimensional problems. Moreover, both computational and statistical guarantees in this setting are largely unknown. The goal of this work is to bridge this gap. In the following, we will explain our contributions in computational and statistical theory.


The second contribution is to propose a path-following algorithm for computing the solution path and analyze the rate of convergence of the algorithm. While the surrogate loss $L_{\delta,K}(\cdot)$ is smooth, it is generally nonconvex and may have multiple local solutions. To overcome the issues arising from nonconvexity, we leverage the homotopy path-following framework \citep{efron2004least,park2007l1,xiao2013proximal,wang2014optimal} and seek for an approximate local solution of the objective function. This algorithm solves a sequence of minimization problems with decreasing regularization parameters until the one that achieves desired statistical performance. Within each stage that corresponds to a fixed regularization parameter, we utilize a warm start from the previous stage which is within the fast local convergence region for the current stage, and apply the proximal-gradient method \citep{nesterov2013gradient} to construct a sequence of sparse approximations approaching the exact local solution of this stage. We show that the algorithm achieves a global geometric rate of convergence, namely, to compute the whole regularization path and approximate an exact local solution up to $\epsilon_{tgt}$ optimization precision, it takes no more than $O(\log(1/\epsilon_{tgt}))$ number of total proximal-gradient iterations. In addition, we find an interesting phenomenon on the computational and statistical trade-off. While the proposed estimator has slower statistical rate than the classical Lasso estimator in linear regression (as explained later), our estimator requires less computational cost in terms of the total number of path-following stages. Unlike the previous analysis of path-following algorithms, our theory requires a more refined analysis which separately controls the approximation bias due to the use of the smoothed surrogate loss and the variance term, from which the effects on the statistical rate of the estimator are further balanced.



The third contribution is to establish nonasymptotic statistical guarantees for the  approximate solution of the path-following algorithm as well as the exact local solution. We show that if the conditional density of $X$ given $Y, \bZ$ satisfies some  H\"older class type condition with smoothness parameter $\beta$ and the true parameter  $\btheta^*$ is $s$-sparse, then the $\ell_2$ estimation error of the proposed estimator achieves the rate (in Theorem \ref{theorem_main2})
$$\big(s\log d/n\big)^{\beta/(2\beta + 1)}.$$
The convergence rate is nonstandard and slower than the rate $(s\log d/n)^{1/2}$ appearing in most of the high-dimensional estimation problems, such as linear regression with Lasso or nonconvex penalty \citep{buhlmann2011statistics}. This result also implies that the estimator has faster rate if the conditional density of $X$ given $Y, \bZ$ is more smooth. Furthermore, in Theorem \ref{theorem_minimax}, we prove that the proposed estimator is minimax rate optimal up to a logarithmic factor. In order to attain the optimal rate, the bandwidth parameter $\delta$ in the smoothed loss function and the regularization parameter $\lambda$ need to satisfy  $\delta \asymp \big(s(\log d)/n\big)^{1/(2\beta + 1)}$ and $\lambda \asymp (\frac{\log d}{n\delta})^{1/2}$, where both depend on unknown  smoothness $\beta$ and sparsity $s$. Finally, we develop adaptive estimation procedures for $\btheta^*$ if $\beta$ or $s$ is unknown by applying Lepski's method \citep{lepskii1991problem,lepskii1993asymptotically}.



\subsection{Related works}
To the best of our knowledge, the parameter estimation problem under a classification loss such as  (\ref{MCID_linear2}) is largely unexplored when the number of covariates diverges to infinity. In a recent work, \cite{zhang2016variable} investigated the estimation and variable selection properties of a high-dimensional linear classifier in a nonconvex penalized SVM framework. Different from our estimand $\btheta^*$ defined in (\ref{MCID_linear2}), they defined the target parameter of interest as the minimizer of the population hinge loss. As a result, our methodological development and theoretical analysis is different from theirs and is indeed much more challenging. For instance, to handle the discontinuity of the 0-1 loss, we propose the approximation via the smoothed surrogate loss with a shrinking bandwidth. Since the hinge loss function itself is continuous and convex, they can directly estimate the parameter via penalized empirical risk minimization. From the theoretical perspective, the proposed estimator has a nonstandard statistical rate which is not improvable by the minimax lower bound, whereas their problem leads to a standard estimator with the $\ell_2$ rate $(s/n)^{1/2}$ up to some logarithmic factor.

We note that during the preparation of this manuscript, we became aware of a recent work of \cite{ban2019}, who studied the maximum score estimator under growing dimension. Our work differs from theirs in many aspects. First, they proposed to minimize an empirical risk function based on the 0-1 loss subject to the $\ell_0$ constraint, say $\|\btheta\|_0\leq s$ assuming $s$ is known. For practical implementation, they proposed to first apply penalized Adaboost to perform variable selection and then solve the smoothed maximum score on the selected covariates.
Instead, we propose an $\ell_1$ penalized estimation framework corresponding to a smoothed surrogate loss via a path-following type algorithm. Both computational and statistical guarantees of the algorithms are established. Second, from a technical perspective, their work was mainly based on a soft margin condition commonly used in classification problems \citep{mammen1999smooth} and the rate of convergence of their estimator depends on the parameter in the margin condition, whereas our analysis relies on the H\"older class smoothness assumption and naturally the smoothness parameter $\beta$ affects the convergence rate. Thus, the model space considered in this paper is intrinsically different from their work, and the theoretical results on the rate of convergence and minimax lower bound also differ naturally from theirs. We refer to Section \ref{theory} for a more detailed comparison.




\subsection{Organization}
The rest of this paper is organized as follows. In section \ref{methodology} we introduce the smoothed surrogate loss function, the penalized estimator as well as the path-following algorithm. In section \ref{theory}, we analyze the theoretical properties of the path-following algorithm along with the upper bound for the estimation error of $\btheta^*$, followed by the minimax lower bound.
Section \ref{adaptive} introduces the procedures for adaptive estimation. The proofs of the main results are shown in section \ref{proofs}. Numerical experiments and a real data application are in section \ref{simulation} and \ref{real_data}, respectively.

\subsection{Notation}
For $\bv = (v_1,\dotso,v_d)^T \in \RR^d$, we use $\bv_S$ to denote the subvector of $\bv$ with entries indexed by the set $S$. We write $\mathds{1}()$ as the indicator function. For $q = [1,\infty)$, $\norm{\bv}_q = \big(
\sum^d_{i = 1}|v_i|^q
\big)^{1/q}$ and $\norm{\bv}_0 = \sum^d_{i = 1}\mathds{1}(v_i \neq 0)$. For any $a,b \in \RR$, we write $a \lor b = \max \{a,b\}$. For any positive sequences $\{a_1,a_2,\dotso \}$ and $\{b_1,b_2, \dotso \}$, we write $a_n \lesssim b_n$ or $a_n = O(b_n)$ if there exists a constant $c$ such that $a_n \leq cb_n~\forall~n$, and $a_n \asymp b_n$ if $a_n \lesssim b_n$ and $b_n \lesssim a_n$.

\section{Methodology}\label{methodology}
\subsection{Smoothed surrogate loss and penalized estimator}\label{method_loss}
Recall that our goal is to estimate $\btheta^*$ in (\ref{MCID_linear}) or equivalently (\ref{MCID_linear2}). 
To tackle the discontinuity of the 0-1 loss function, we consider a family of smoothed surrogate loss functions
\beq\label{smooth_loss}
L_{\delta,K}(u) =  \int^{\infty}_{u/\delta}K(t)dt,
\eeq
where $K(t)$ is a symmetric kernel function satisfying $\int K(t)dt = 1$ and $\delta > 0$ is a bandwidth parameter.
Given the surrogate loss, we define the surrogate risk function $R_{\delta,K}$ as
 \beq\label{eq_Rdelta}
 R_{\delta,K}(\btheta) = \EE \bigg[
 w(Y)L_{\delta,K}(Y(X - \btheta^T\bZ))
 \bigg],
 \eeq
 and the corresponding empirical risk as
 \beq
 R_{\delta,K}^n(\btheta) = \frac 1n\sum_{i = 1}^{n}w(y_i)L_{\delta,K} \bigg(y_i(x_i - \btheta^T\bz_i)\bigg).
 \eeq
There are two main reasons to introduce the smoothed surrogate loss $L_{\delta,K}(u)$ and the associated population and empirical risk functions. First, from the optimization perspective, unlike the empirical version of (\ref{MCID_linear2}), the smoothed empirical risk function $R_{\delta,K}^n(\btheta)$ is differentiable so that we can leverage the smoothness of $R_{\delta,K}^n(\btheta)$ to design gradient-based algorithms for computation. The proposed procedure can be interpreted as a solution when we do not have direct access to the gradient of an objective function, in order to apply first order methods.

Second, from the statistical perspective, the smoothed surrogate loss $L_{\delta,K}(u)$ can be viewed as an approximation of the original 0-1 loss, i.e., $L_{\delta,K}(u)$ converges to $\mathds{1}(u<0)$ for any fixed $u\neq 0$ as the bandwidth parameter $\delta \searrow 0$. To better understand the role of the kernel function $K(\cdot)$,
we observe that
\beq\label{eq_grad_original}
\nabla R(\btheta) = \sum_{y = \pm 1}w(y)\int_{\bZ} \bz y f(\btheta^T\bz|\bz,y) f(\bz,y) d\bz,
\eeq
where $R(\btheta)$ is defined in (\ref{MCID_linear2}), $f(x|\bz,y)$ is the conditional density of $X$ given $\bZ, Y$ and $f(\bz,y)$ is the joint density of $\bZ$, $Y$. Similarly,
\beq \label{eq_grad_smooth}
\nabla R_{\delta,K}(\btheta) &= \sum_{y = \pm 1}w(y) \int_{\bZ} \bz y \Big[\int_X\frac 1\delta   K\bigg(\frac{y(x - \btheta^T\bz)}{\delta}\bigg) f(x|\bz,y)dx \Big]f(\bz,y) d\bz.
\eeq
Thus, at the population level, the conditional density function $f(\btheta^T\bz|\bz,y)$ in (\ref{eq_grad_original}) is substituted by its kernel approximation
$$
\int_X\frac 1\delta   K(\frac{y(x - \btheta^T\bz)}{\delta}) f(x|\bz,y)dx=\int_X\frac 1\delta   K(\frac{x - \btheta^T\bz}{\delta}) f(x|\bz,y)dx,
$$
for some symmetric kernel, when applying the smoothed surrogate loss. As $\delta \searrow 0$, the gradient of the smoothed risk function approximates the gradient of the risk function based on the original 0-1 loss. A formal theoretical analysis in section \ref{sec_theory} implies that the estimation of $\btheta^*$ based on the smoothed surrogate loss is asymptotically Fisher consistent.


\begin{remark}
We note that in machine learning literature many examples in the class of surrogate loss  in (\ref{smooth_loss}) have been studied. For instance, when $K(t) =\frac12\mathds{1}(|x| \leq 1)$ is the  rectangular kernel, (\ref{smooth_loss}) becomes the ramp loss or its variants, such as $\psi$-loss \citep{shen2003psi} or truncated hinge loss \citep{wu2007robust}.  However, most of these works focus on the performance of their classifiers in terms of excess risk, rather than the estimation of $\btheta^*$. A more critical difference is that we incorporate the bandwidth parameter in the kernel function and shrink it towards 0 for the (asymptotic) Fisher consistency in estimation.
\end{remark}

When the parameter $\btheta^*$ is of high dimension (i.e., $d\gg n$), a common practice in statistics is to assume $\btheta^*$ is $s$-sparse, namely $\norm{\btheta^*}_0 \leq s$ for some $s < n$. To estimate $\btheta^*$, we propose to minimize the following $\ell_1$ penalized smoothed empirical risk function,
\beq \label{eq_penalized_obj}
\widehat{\btheta}_{{\delta,K,\lambda}} = \argmin_{\btheta} \big\{R^n_{\delta,K}(\btheta)  + \lambda \norm{\btheta}_1\big\},
\eeq
where $\lambda > 0$ is a regularization parameter.
To ease notation, in the rest of the paper we omit the dependence of the risk function on $K$ and the dependence of the estimator on $K$ and $\delta$ by writing $R^n_{\delta,K}(\btheta) = R^n_{\delta}(\btheta) $ and $\widehat{\btheta}_{{\delta,K,\lambda}} = \widehat{\btheta}_{\lambda}$
unless confusion may occur.



\subsection{Path-following algorithm}\label{method_algorithm}
Although the smoothed surrogate loss function remedies the discontinuity of the 0-1 loss, it is generally nonconvex, and thus computing the global solution to (\ref{eq_penalized_obj}) is still intractable.  In addition, even if the goal is to compute some local solution(s) satisfying the first order optimality condition, it is of both practical and theoretical interest to account for the optimization and statistical accuracy in order to design the algorithm and analyze its rate of convergence. To address these issues, we seek to utilize the homotopy path-following framework \citep{efron2004least,park2007l1,xiao2013proximal,wang2014optimal} to compute  approximate local solutions to (\ref{eq_penalized_obj}) corresponding to a sequence of decreasing regularization parameters $\lambda$, until the target regularization parameter is reached.

To be specific, we firstly choose a sequence of $\lambda_0 > \lambda_1 > \dotso >\lambda_{N} = \lambda_{tgt}$, where
\[
\lambda_{t} = \phi^t\lambda_0, t = 0,1,\dotso
\]
for some constant $\phi\in(0,1)$ and $\lambda_{tgt}$ is the target regularization parameter to be specified later. In practice,  we can initialize $\lambda_0 = \norm{\nabla R_\delta^n(\mathbf{0})}_\infty$. By checking the optimality condition, it can be shown that $\tilde{\btheta}_0 = \mathbf{0}$ is an exact local solution to (\ref{eq_penalized_obj}) and will be the initial value in the algorithm. We note that it is possible to choose alternative initial values for $\lambda_0$ provided $\lambda_0$ is large enough such that the solution to (\ref{eq_penalized_obj}) is sparse. Here, $N$ denotes the total number of the path-following stages and we set $N = \frac{\log(\lambda_{tgt}/\lambda_0)}{\log \phi}$. Without loss of generality, we assume $N$ is an integer. At each stage $t=1,...,N$, the goal is to approximately compute the exact local solution $\widehat{\btheta}_t$ corresponding to $\lambda_t$,
\[
\widehat{\btheta}_{t} = \argmin_{\btheta} R^n_{\delta}(\btheta)  + \lambda_t \norm{\btheta}_1.
\]
To this end, we apply the proximal-gradient method \citep{nesterov2013gradient} to iteratively approximate $\widehat{\btheta}_t$ by minimizing a sequence of quadratic approximations of $R^n_{\delta}(\btheta)$ over a convex constraint set $\Omega$:
\beq
\btheta^{k+1}_t &=\argmin_{\btheta \in \Omega}\bigg \{
R^n_\delta(\btheta^k_t ) + \langle \nabla R^n_\delta(\btheta^k_t ),(\btheta - \btheta^k_t) \rangle + \frac{1}{2\eta}\norm{\btheta - \btheta^k_t}_2^2 + \lambda_t \norm{\btheta}_1
\bigg \}
\\&= \argmin_{\btheta\in \Omega}\bigg\{ \frac1 {2\eta}
\bignorm{\btheta - \btheta^k_t + \eta \nabla R^n_\delta(\btheta^k_t ) }_2^2
+ \lambda_t \norm{\btheta}_1
\bigg  \}\\
&:=\cS_{\lambda_t \eta}\big(\btheta^k,\Omega \big),
\eeq
where $\eta$ is the step size to be specified later. Since $R^n_{\delta}(\btheta)$ is nonconvex, there may exist unwanted local solutions to (\ref{eq_penalized_obj}). To further regularize the sequence of estimators, we force $\btheta^{k+1}_t$ to stay in the set $\Omega$ in every iteration, on which we assume the empirical risk function $R^n_{\delta}(\btheta)$ is well behaved, namely, the restricted strong convexity and restricted smoothness (see Assumption \ref{RSC}) hold over $\Omega$. 
The proximal-gradient algorithm is described in Algorithm \ref{proxi_algo}. In the algorithm, we use the stopping criteria $w_\lambda(\btheta)$ defined as
\beq\label{equ_subopti}
\omega_\lambda(\btheta) = \min_{\bxi \in \partial \norm{\btheta}_1}\max_{\btheta' \in \Omega} \bigg \{ \frac{(\btheta - \btheta')^T}{\norm{\btheta - \btheta'}_1}(\nabla R_\delta^n(\btheta) + \lambda \bxi) \bigg \}.\eeq
To interpret this, notice that when $\btheta$ lies in the interior of $\Omega$, $\omega_\lambda(\btheta)$ reduces to
\[
\omega_\lambda(\btheta) = \min_{\bxi \in \partial \norm{\btheta}_1}\norm{\nabla R_\delta^n(\btheta) + \lambda \bxi}_\infty.
\]
The first-order optimality condition for (\ref{eq_penalized_obj}) implies that if $\hat{\btheta}$ is an exact solution then there exists
$\bxi \in \partial \norm{\hat{\btheta}}_1$ such that $\nabla R_\delta^n(\hat{\btheta}) + \lambda \bxi = 0$.
Therefore, $\omega_\lambda(\btheta)$ can serve as a measure of sub-optimality for the approximate solution.

At stage $t$, the proximal-gradient algorithm returns an approximate solution  $\tilde\btheta_t$ with precision $\epsilon_t =  \nu \lambda_t, \nu \in (0,1)$ corresponding to $\lambda_t$ (see Algorithm \ref{pf_algo} for the definition of $\tilde\btheta_t$). This ensures that $\tilde{\btheta}_t$ achieves an adequate precision, such that it belongs to a fast local convergence region for the next stage corresponding to a smaller regularization parameter $\lambda_{t+1}$. Then we use $\tilde{\btheta}_t$ as a warm start for stage $t+1$ and repeat this process. At the final stage $N$, we would compute the approximate solution $\tilde{\btheta}_N = \tilde{\btheta}_{tgt}$ corresponding to $\lambda_{tgt}$ using a high precision $\epsilon_{tgt}$. The detail of the path-following algorithm is described in Algorithm \ref{pf_algo}.

By using this path-following algorithm, one can efficiently compute the entire solution path which is often desired in many real applications (see the data analysis in section \ref{real_data}). In addition, the path-following algorithm may offer extra computational convenience when the estimators need to be computed at many values of the tuning parameter such as in cross-validation (see section \ref{simulation}).

\begin{algorithm}[ht]
	\SetKwInOut{Input}{input}
	\SetKwInOut{Output}{output}
	\SetKwInOut{Par}{parameter}
	\SetKwInOut{Init}{initialize}
	\Input{$\lambda_{tgt} > 0, \nu > 0, \phi > 0, \epsilon_{tgt} > 0,\Omega$}
	\Par{$\eta > 0$}
	\Init{$\tilde{\btheta}_0 \leftarrow \mathbf{0}$, $\lambda_0 \leftarrow \norm{\nabla R_\delta^n(\mathbf{0})}_\infty$, $N \leftarrow \frac{\log(\lambda_{tgt}/\lambda_0)}{\log \phi} $}
	\For{$t = 1,\dotso,N-1$}{
		$\lambda_t \leftarrow \phi^t\lambda_0$\;
		$\epsilon_t \leftarrow \nu \lambda_t$\;
		$\tilde{\btheta}_t \leftarrow$Proximal-Gradient($\lambda_t,\epsilon_t,\tilde{\btheta}_{t-1},\Omega$)\;
	}
	$\tilde{\btheta}_N	\leftarrow$ Proximal-Gradient($\lambda_{tgt},\epsilon_{tgt},\tilde{\btheta}_{N-1},\Omega$)\
	\caption{ $\btheta \leftarrow$Path-Following($\lambda_0,\lambda_{tgt},\nu,N,\epsilon_{tgt}, \Omega$)}\label{pf_algo}
	\Return {$\tilde{\btheta}_{tgt}=\tilde{\btheta}_N$	}
\end{algorithm}

\begin{algorithm}[ht]
	\SetKwInOut{Input}{input}
	\SetKwInOut{Output}{output}
	\SetKwInOut{Par}{parameter}
	\SetKwInOut{Init}{initialize}
	\Input{$\lambda > 0, \epsilon > 0, \btheta^0 \in \RR^d,\Omega $}
	\Par{$\eta > 0$}
	\Init{ $k \leftarrow 0$}
	\While{$w_\lambda(\btheta^k) > \epsilon$}{
		$k \leftarrow k+1$\;
		$\btheta^{k+1} \leftarrow \cS_{\lambda \eta}\big(\btheta^k,\Omega\big)$\;
	}
	\Return{$\btheta^{k+1}$}
	\caption{ $\btheta \leftarrow$Proximal-Gradient($\lambda,\epsilon,\btheta^0,\Omega$)}\label{proxi_algo}
\end{algorithm}


\section{Theoretical properties}\label{theory}\label{sec_theory}
In this section, we establish theoretical results for the statistical performance of the approximate local estimator $\tilde{\btheta}_{tgt}$ (and the exact local estimator $\widehat{\btheta}_{tgt}$), as well as iteration-complexity of the path-following algorithm. In section \ref{sec31}, we list the required assumptions and analyze the optimization algorithm, in which the upper bounds on the statistical rate of convergence for $\tilde{\btheta}_{tgt}$ (and $\widehat{\btheta}_{tgt}$) are embedded. In section \ref{sec32}, we show that the estimator $\tilde{\btheta}_{tgt}$ (and $\widehat{\btheta}_{tgt}$) are minimax rate optimal up to a logarithmic factor. Section \ref{banergee} contains a detailed comparison with some related works.
For simplicity of exposition, we assume that the weight function $w(y)$ is known. Following the proposal in \cite{xu2014model}, we choose $w(y) = 1/\PP(Y = y)$ in the rest of the theoretical development.

\subsection{Assumptions and upper bound}\label{sec31}
We start from the conditions on the underlying unknown distribution  $P = P(X,Y,\bZ)$.

\begin{assumption}\label{assum_propensity}
	There exists a constant $c>0$ such that $c \leq \PP(Y = 1) \leq 1-c$. 
\end{assumption}

\begin{assumption}\label{assum_bounded}
	For all $j = 1,\dotso,d$ we have $|Z_j| \leq M_n$, where $M_n$ is allowed to increase with $n$ such that $M_n \leq C\sqrt{\frac{n\delta}{\log d}}$ for some constant $C$.
	In addition, for all $j= 1,\dotso,d$ and $y\in \{-1,1\}$, we assume $\EE [Z_j^2|Y = y] \leq M_2$ for some constant $M_2 > 0$.
\end{assumption}
Assumption \ref{assum_propensity} ensures that the weight function $w(y)$ is bounded away from infinity. Assumption \ref{assum_bounded} is a technical condition on the distribution of the covariates $\bZ$. For binary covariate $Z_j\in \{0,1\}$, $|Z_j| \leq M_n$ holds with $M_n=1$. Similarly, if the covariates $\bZ$ are entrywise sub-Gaussian with a bounded sub-Gaussian norm, $\max_{1\leq j\leq d}|Z_j| \leq M_n$ holds with high probability with $M_n=C(\log d)^{1/2}$ for some constant $C$. In this case, Assumption \ref{assum_bounded} requires $\log d\leq C(n\delta)^{1/2}$, which is a mild condition provided $\delta$ does not converge to 0 too fast.

Recall that the smoothed surrogate loss depends on the kernel function $K(\cdot)$. For completeness, we firstly specify regularity conditions for $K(\cdot)$.

\begin{definition}\label{def_kernel1}
	We say function $K(t)$ is a proper kernel if it satisfies
	\begin{enumerate}[label=(\Alph*)]
		\item $K(t) = K(-t)\;\forall\; t\in \RR,$
		\item $|K(t)| \leq K_{\max} < \infty \; \forall\; t\in \RR,$
		\item $\int K(t)dt = 1,$
		\item $\int K^2(t)dt < \infty.$
	\end{enumerate}
	Moreover, we say a kernel is a proper kernel of order $l\geq 1$ if it additionally satisfies
	$$\int t^jK(t)dt = 0, \;\forall\;j = 1,\dotso,l.$$
\end{definition}

%

The following definition and assumption are concerned with the smoothness of the conditional density of $X$ given $\bZ$ and $Y$.

\begin{definition}\label{def_holder_new}
 Let $l = \floor{\beta}$ be the greatest integer strictly less than $\beta>0$. We say $P \in \cP(\beta,L, s')$ 
	if the conditional density $f(x|y,\bz)$ of $X|Y,\bZ$ is $l$ times differentiable w.r.t $x$ for any $y,\bz$, and satisfies
	\beq\label{eq_holder_new}
	 \bigg |  \int \bv^T\bz \big[
	f^{(l)}(\Delta + \btheta^{*T}\bz|y,\bz) - f^{(l)}( \btheta^{*T}\bz|y,\bz)
	\big]f(\bz|y) d\bz \bigg |  \leq L\norm{\bv}_2|\Delta|^{\beta - l},
	\eeq
	for any $y \in \{-1,1\}$, $\bz\in\RR^d$, $\Delta \in \RR$ and $\bv$ in the sparse set $\{\bv\in \RR^d: \norm{\bv}_0 \leq s'\}$.
\end{definition}

\begin{assumption}\label{assum_density}
	Recall that $\|\btheta^*\|_0\leq s$. We assume $P \in \cP(\beta,L,s')$, where $\beta, L$ are constants, and $s' =Cs$ for some sufficiently large constant $C$.  In addition,
	\[\sup_{x\in\RR, y\in \{-1,1 \},\bz\in \RR^d} f(x|y,\bz) < p_{\max} < \infty
	\]
	for some constant $ p_{\max}  > 0$.
\end{assumption}

\begin{remark}
	 Assumption \ref{assum_density} assumes that the conditional density $f(x|y,\bz)$ satisfies a weaker version of H\"older class condition. Adapted from \citet{tsybakovintroduction}, we say that the conditional density $f(x|y,\bz)$ belongs to the uniform H\"older class $\Sigma(\beta,L')$, if $f(x|y,\bz)$ is $ \floor{\beta}$ times differentiable and
	\beq\label{def_holder}
	\big |f^{(\floor{\beta})}(x_1|y,\bz) -  f^{(\floor{\beta})}(x_2|y,\bz)   \big | \leq L' |x_1 - x_2|^{\beta - \floor{\beta}},
	\eeq
for some constant $L'$ uniformly in $\bz,y$.
	In Assumption \ref{assum_density}, instead of assuming $f(x|y,\bz)\in \Sigma(\beta,L')$, we only require a moment condition on the smoothness of $f(x|y,\bz)$. To see this, we relax (\ref{def_holder}) to the following non-uniform H\"older class condition
	\beq\label{def_holder2}
	\big |f^{(\floor{\beta})}(x_1|y,\bz) -  f^{(\floor{\beta})}(x_2|y,\bz)   \big | \leq L(y,\bz) |x_1 - x_2|^{\beta - \floor{\beta}},
	\eeq
for any $x_1, x_2, y, \bz$, where we allow $L(y,\bz)$ to depend on $y,\bz$ and possibly diverge for some $y,\bz$. By (\ref{def_holder2}) and Cauchy-Schwarz inequality, the left hand side of (\ref{eq_holder_new}) is bounded above by
$$
\big[\EE \{(\bv^T\bZ)^2 L^2(y,\bZ)|Y=y\}\big]^{1/2}|\Delta|^{\beta-l}\leq  L\norm{\bv}_2|\Delta|^{\beta - l},
$$
provided the maximum sparse eigenvalue of $\EE(L^2(y,\bZ)\bZ\bZ^T|Y=y)$ is bounded, i.e.,
	\beq\label{maxeigen}
	\sup_{\|\bv\|_0\leq s'}\frac{\bv^T\EE(L^2(y,\bZ)\bZ\bZ^T|Y=y)\bv}{\|\bv\|^2_2}\leq L^2.
	\eeq 	
Thus, the non-uniform H\"older class condition (\ref{def_holder2}) and the sparse eigenvalue condition (\ref{maxeigen}) together imply (\ref{eq_holder_new}) in Assumption \ref{assum_density}. To ease notation, in the rest of the paper we write $\cP(\beta,L) = \cP(\beta,L,s')$.

\end{remark}

In the sequel, we analyze the properties of $\nabla R^n_\delta(\btheta^*)$, which is the key influencing factor for the statistical and computational performance of the proposed estimator. In the existing high-dimensional M-estimation framework such as  \cite{negahban2012unified,wang2013calibrating,loh2013regularized,wang2014optimal}, it is typically assumed that the true parameter $\btheta^*$ is a minimizer of the population risk function. However, in our case, there exists an additional approximation bias due to the smoothed surrogate loss. To see this, recall that $R(\btheta)$ is defined in  (\ref{MCID_linear2}) and $\nabla R(\btheta^*)=0$ by the first order condition. Then, decomposing  $\nabla R^n_\delta(\btheta^*)$ yields
\beq\label{eq_biasvariance}
\nabla R^n_\delta(\btheta^*) = \big( \underbrace{\nabla R^n_\delta(\btheta^*) - \nabla R_\delta(\btheta^*)}_{E_1}  \big) + \big( \underbrace{\nabla R_\delta(\btheta^*) - \nabla R(\btheta^*)}_{E_2}  \big),
\eeq
where $E_1$ can be interpreted as the variation of the smoothed empirical risk function relative to its population version, and $E_2$ is the approximation bias induced by the smoothed surrogate loss. In the aforementioned high-dimensional M-estimation literature, typically we get $E_2=0$. Consequently, in our theoretical analysis, instead of directly controlling $\norm{\nabla R^n_\delta(\btheta^*)}_\infty$ as in the existing literature, it is necessary to deal with $E_1$ and $E_2$ separately and balance their effects on the statistical rate of convergence of the proposed estimator in order to derive sharp upper bounds. In what follows, we control the magnitude of $E_1$ and $E_2$ in Propositions \ref{prop_var} and \ref{prop_bias}, respectively.

\begin{proposition}\label{prop_var}
    Under Assumptions \ref{assum_propensity}-\ref{assum_density}, if $K(\cdot)$ is a proper kernel then with probability greater than $1 - 2d^{-1}$, we have
	\[\norm{\nabla R_\delta^n(\btheta^*) - \nabla R_\delta(\btheta^*)}_\infty  \leq C_1\sqrt{\frac{\log d}{n\delta}},\]
	where $C_1$ is a constant independent of $n$ and $d$.
\end{proposition}
\begin{proof}
	See Section \ref{proof_prop_var} for the  detailed proof.
\end{proof}

\begin{proposition}\label{prop_bias}
	Under Assumptions \ref{assum_propensity}-\ref{assum_density}, by choosing a proper kernel $K(\cdot)$ of order $l = \floor{\beta}$ satisfying $\int |t^{\beta}||K(t)|dt <  \infty$, we have for any  $\bv \in \RR^d$ with $\norm{\bv}_0 \leq s'$,
	\[
	\big | \bv^T(\nabla R_\delta(\btheta^*) - \nabla R(\btheta^*))\big | \leq C_2\delta^\beta \norm{\bv}_2,
	\]
	where $s'$ is defined in Assumption \ref{assum_density} and $C_2$ is a constant independent of $n$ and $d$.
\end{proposition}
\begin{proof}
	See Section \ref{proof_prop_bias} for the  detailed proof.
\end{proof}

The two propositions imply that the standard deviation and the bias are of order $O_p(\sqrt{\frac{\log d}{n\delta}})$ and $O(\delta^{\beta})$, respectively, which are similar to the results in kernel density estimation problems \citep{tsybakovintroduction}. Here, an extra $\log d$ factor appears in the standard deviation term to account for the high-dimensionality. While the smoothed population risk $R_\delta(\btheta^*)$ becomes a better approximation of $R(\btheta^*)$ by decreasing $\delta$ to $0$ in a very fast rate, it incurs high variability in the smoothed empirical risk $ R_\delta^n(\btheta^*)$ which makes the estimator less stable. Therefore, one would expect that the optimal estimator is attained via an appropriate $\delta$ that balances the bias and variance.
We will return to this problem after Theorem \ref{theorem_main2}.

In what follows, we make assumptions on the curvature of the empirical risk function.

\begin{assumption}\label{RSC}
	There exists a set $\Omega = \{\bv:\norm{\bv}_2 \leq R\}$ for some $R$ possibly increases with $n,d$, such that $\btheta^* \in \Omega$ and
	the following restricted strong convexity (RSC) and restricted smoothness (RSM) conditions hold over sparse vectors in $\Omega$, that is, for any $\btheta, \btheta' \in \Omega$ and $\norm{\btheta'}_0,\norm{ \btheta}_0 \leq s + \tilde{s}$,
	\beq\label{RSC_condi}
	R_\delta^n(\btheta') \geq R_\delta^n(\btheta) + \nabla R_\delta^n(\btheta)^T(\btheta' - \btheta) + \frac{1}{2} \rho_{-}\norm{\btheta' - \btheta}^2_2,
	\eeq
	and
	\beq\label{RS_condi}
	R_\delta^n(\btheta') \leq R_\delta^n(\btheta) + \nabla R_\delta^n(\btheta)^T(\btheta' - \btheta) + \frac{1}{2} \rho_{+}\norm{\btheta' - \btheta}^2_2,
	\eeq
	where $0 < \rho_{-}\leq \rho_{+} < +\infty$ are two constants and $\tilde s =Cs$ with the constant $C$ specified in Lemma \ref{lemma_sparse}.
	

\end{assumption}
Assumption \ref{RSC} is the key to the fast statistical and computational rate of the path-following algorithm. Similar conditions, including the restricted isometry property (RIP) and sparse eigenvalue condition have been discussed extensively in \citet{candes2005decoding,NIPS2010_3984,negahban2012unified,loh2013regularized,wang2013calibrating}.
In this assumption, we require that the empirical risk is $\rho_{-}$-strongly convex and $\rho_{+}$-smooth when restricted to sparse vectors  in the set $\Omega$.
In general, when $d$ is fixed and $n$ goes to infinity, the empirical risk $R_\delta^n(\btheta)$ will be strongly convex and smooth locally in $\Omega$ if the corresponding population risk $R_\delta(\btheta)$ is strongly convex and smooth in the same region under mild conditions. If the kernel is chosen such that $R_\delta(\btheta)$ is twice differentiable, it is equivalent to saying that the minimum and maximum sparse eigenvalues of the population Hessian $ \nabla^2R_\delta(\btheta)$ are bounded away from zero and infinity for $\btheta \in \Omega$. Under the high-dimensional regime, however, the empirical Hessian $\nabla^2 R_\delta^n(\btheta)$, if exists, is generally singular. To fix this issue, Assumption \ref{RSC} only requires the convexity and smoothness in certain directions represented by sparse vectors. The path-following framework exploits this assumption by constructing a sequence of sparse intermediate estimators approaching the final estimator, and thus achieves fast convergence under Assumption \ref{RSC}.
 \begin{remark}
Due to model-free formulation in (\ref{MCID_linear2}), the Assumption \ref{RSC} can be verified in a case by case manner under specific distributional assumptions. As an illustration, in Section \ref{app_RSC} in the Appendix, we show that under some regularity conditions Assumption \ref{RSC} holds for the conditional mean model with high probability.
 \end{remark}

Now we are ready to characterize the iteration-complexity of the path-following algorithm as well as the statistical performance of the estimator $\tilde{\btheta}_{tgt}$, which is the approximate local solution to (\ref{eq_penalized_obj}) from the proximal-gradient algorithm with the tuning parameter $\lambda_{tgt}$.  

\begin{theorem}\label{theorem_main2}
		Under Assumptions \ref{assum_propensity}-\ref{RSC},  with a proper kernel $K(\cdot)$ of order $l = \floor{\beta}$ satisfying $\int |t^{\beta}||K(t)|dt <  \infty$, by choosing $\nu = 0.25$, $\phi= 0.9$, $\eta \leq \frac{1}{\rho_{+}}$,
		$\lambda_{tgt} = 8C_1\sqrt{\frac{\log d}{n\delta}}$ where $C_1$ is defined in Proposition \ref{prop_var}
		and $\delta = c\big(\frac{s \log d}{n}\big)^{1/(2\beta+1)}$ for some constant $c > 0$,  then with probability greater than $1 - 2d^{-1}$ the following results hold:
	\begin{enumerate}
		\item\label{main_result1} The final approximate local solution $\tilde{\btheta}_{tgt}$ from the path-following algorithm satisfies
		\[
		\begin{aligned}
		&\norm{\tilde{\btheta}_{tgt} - \btheta^*}_2 \lesssim \bigg(\frac{s \log d}{n}\bigg)^{\beta/(2\beta+1)},\\
		&\norm{\tilde{\btheta}_{tgt} - \btheta^*}_1 \lesssim \sqrt{s}\bigg(\frac{s \log d}{n}\bigg)^{\beta/(2\beta+1)}.
		\end{aligned}
		\]
		
		\item\label{main_result2}  It takes no more than
		\[
		1 + I'\log(4I\sqrt{s})
		\]
		iterations to compute each of the first $N-1$ stage of Algorithm \ref{pf_algo}, and
		\[
		\max\big\{ 1 + (I'\log(I \sqrt{s}\lambda_{tgt}/\epsilon_{tgt}), 0 \big\}
		\]
		iterations to compute final stage $N$, where $I =\sqrt{2(\frac 1\eta + \rho_{+})^2\eta \bar{C}_2/\rho_{-}}$ and $I' = \frac 2 {\log\big(\frac{4}{4 - \eta\rho_{-}}\big)}$.
	\end{enumerate}
\end{theorem}
\begin{proof}
	See Section \ref{proof_theo_main2} for the  detailed proof.
\end{proof}

The results in Theorem \ref{theorem_main2} are two folds. The first part of this theorem is about the statistical rates of convergence for the final estimator $\tilde{\btheta}_{tgt}$. Not surprisingly, due to the bias and variance decomposition in (\ref{eq_biasvariance}), the rates of convergence in the $\ell_2$ and $\ell_1$ estimation errors are nonstandard and slower than the rates $(s\log d/n)^{1/2}$ and $s(\log d/n)^{1/2}$ appearing in most of the high-dimensional estimation problems, such as linear regression with Lasso or nonconvex penalty \citep{buhlmann2011statistics}. In addition, our result implies that the rates of convergence of $\tilde{\btheta}_{tgt}$ are faster as the conditional density of $X$ given $\bZ, Y$ becomes smoother (i.e., with a larger $\beta$). The rates can be close to the standard rates $(s\log d/n)^{1/2}$ and $s(\log d/n)^{1/2}$ if the condition density is smooth enough. This is an important advantage of the proposed estimator, see section \ref{banergee} for the comparison with related estimators.

To further decipher the statistical rate, a closer look at the proof reveals that the resulting rate is a consequence of balancing the following two terms
\[ \delta^{\beta}~~~\text{and}~~~  \sqrt{s}\lambda_{tgt} \asymp  \sqrt{\frac{s \log d}{n\delta}},\]
where the first term corresponds to the approximation bias in Proposition \ref{prop_bias} and the second term comes from the estimation error of $\tilde{\btheta}_{tgt}$ towards the minimizer of $R_\delta(\btheta)$ in (\ref{eq_Rdelta}) (rather than $\btheta^*$). Indeed, directly controlling $\norm{\nabla R^n_\delta(\btheta^*)}_\infty$ via Propositions \ref{prop_var} and \ref{prop_bias} as in the previous M-estimation literature \citep{negahban2012unified,wang2013calibrating,loh2013regularized,wang2014optimal} leads to suboptimal rates $s^{1/2}(\log d/n)^{\beta/(2\beta+1)}$ and $s(\log d/n)^{\beta/(2\beta+1)}$ for $\ell_2$ and $\ell_1$ estimation errors. To attain the optimal rate, we conduct a more refined analysis for the path-following algorithm by separately bounding the approximation bias and the estimation error. This makes the current analysis different from the existing ones.

The second part of Theorem \ref{theorem_main2}
is concerned with the iteration complexity of the path-following algorithm. Recall that for the first $N-1$ stages, we only need to solve the local solution given by Algorithm \ref{proxi_algo} to a relatively low precision. Thus by taking $\epsilon_{tgt}$ sufficiently small, Theorem \ref{main_result2} indicates that computing the whole regularization path
takes no more than $O(\log (1/\epsilon_{tgt}))$ number of proximal-gradient steps, which is known to be optimal among all first-order optimization methods.

From the statistical perspective, by taking $\epsilon_{tgt}=c\lambda_{tgt}$ for some very small constant $c$, the computational complexity of the algorithm is dominated by $O(N\log (\sqrt{s}))$, due to the iterations in the first $N-1$ stages. Recall that $N = \frac{\log(\lambda_{tgt}/\lambda_0)}{\log \phi}$, where $\phi$ is an absolute constant, say $\phi = 0.9$. Thus, up to some additive constant, the total number of path-following stages is
$$
N=\frac{\beta}{(2\beta+1)\log (0.9^{-1})}\log (\frac{n}{s\log d}).
$$
Interestingly, we find that, as $\beta$ decreases, the required total number of stages in the path-following algorithm decreases. In particular, in our case $N$ is smaller than that in generalized linear models (GLMs) with Lasso penalty, where the latter requires $\frac{1}{2\log (0.9^{-1})}\log (\frac{n}{s\log d})$ path-following stages using the same constant $\phi=0.9$ \citep{wang2014optimal}. The comparison with GLMs reveals a new phenomenon on the statistical and computational trade-off, i.e., computing our estimator which has slower statistical rates requires less computational cost in terms of the number of path-following stages. The intuition is that there is no benefit to further iterating the path-following stages once our estimator has already reached the desired statistical rate.

\begin{remark}
Note that to achieve the rates in Theorem \ref{theorem_main2}, the  bandwidth parameter $\delta$ needs to be set in the order of $\big(\frac{s \log d}{n}\big)^{1/(2\beta+1)}$, which depends on the unknown sparsity level $s$ and smoothness $\beta$. In section \ref{adaptive}, we will discuss how to perform adaptive estimation if either $s$ or $\beta$ is unknown.
\end{remark}

The next corollary characterizes the rate of convergence of $\hat{\btheta}_{tgt}$, the exact local solution to the minimization problem (\ref{eq_penalized_obj}) with the target regularization parameter $\lambda_{tgt}$.

 \begin{corollary}[Statistical rate of convergence]\label{corollary_main1}
	Under the same conditions in Theorem \ref{theorem_main2},  with probability greater than $1 - 2d^{-1}$, the exact location solutions $\widehat{\btheta}_{tgt}$ satisfies
	\[
	\begin{aligned}
	&\norm{\hat{\btheta}_{tgt} - \btheta^*}_2 \lesssim \bigg(\frac{s \log d}{n}\bigg)^{\beta/(2\beta+1)}~~~\text{and}
	&\norm{\hat{\btheta}_{tgt} - \btheta^*}_1 \lesssim \sqrt{s}\bigg(\frac{s \log d}{n}\bigg)^{\beta/(2\beta+1)}.
	\end{aligned}
	\]
\end{corollary}
\begin{proof}
	See Section \ref{proof_coro} for the  detailed proof.
\end{proof}

\subsection{Minimax rate of convergence}\label{sec32}

In this section, we will prove that our estimator is minimax rate optimal up to a logarithmic factor of $s$. Formally, we define the following parameter space
\[ \Theta_{\beta,s} =   \{(\btheta^*,P) : \btheta^* \in B_0(s) , P \in \cP(\beta,L), \btheta^* = \argmin_{\btheta}R_{P}(\btheta),
 \lambda_{\min}(\nabla^2R_P(\btheta^*)) \geq \rho
\},\]
where $B_0(s) = \{x\in \RR^d: \norm{x}_0 \leq s\}$, $\cP(\beta,L)$ is defined in Definition \ref{def_holder_new}, $R_{P}(\btheta)$ is the risk function evaluated at the joint distribution  $P$ in (\ref{MCID_linear2}), $\lambda_{\min}(\cdot)$ denotes the minimum eigenvalue and $\rho > 0$ is some positive constant.
Here for parameter set $\Theta_{\beta,s}$, we require that the true minimizer $\btheta^*$ is $s$-sparse and the probability measure $P$ satisfies the $\beta$-smoothness condition in Assumption \ref{assum_density}.
In addition, we require that the minimum eigenvalue of the Hessian of the risk function is lower bounded, which is the population version of the RSC condition in Assumption \ref{RSC}.
We define the minimax risk for estimating $\btheta^*$ in $\Theta_{\beta,s}$ as
\[
\inf_{\hat{\btheta}} \sup_{(\btheta,P) \in \Theta_{\beta,s}}\EE_{P}\norm{\hat{\btheta} - \btheta}_q,
\]
for $q=1,2$, where the expectation is taken under $P$.

In order to find a lower bound for the minimax risk, we follow the reduction scheme in \citet{tsybakovintroduction} to reduce minimax estimation to a multiple hypothesis testing problem. We then construct a set of hypotheses by carefully varying the conditional density $f(X|\bZ, Y)$ in $\cP(\beta,L)$ rather than directly varying $\btheta^*$ in $B_0(s)$ as in the lower bound for linear regression \citep{raskutti2011minimax,bellec2018slope}. Since our goal is to estimate $\btheta^*$ instead of the conditional density $f(X|\bZ, Y)$, we need to show that, under the constructed $f(X|\bZ, Y)$ in all hypotheses, the minimizers of the risk function $R_{P}(\btheta)$ exist, are unique, $s$-sparse and well separated in $\ell_q$ norm in order to get a sharp lower bound.
The main result in this section is as follows.

\begin{theorem}\label{theorem_minimax}
	Let $d \geq 2$ and the smoothness parameter $\beta > 1$.
	If $n \geq C\log (d/s)s ^{(4\beta - 1)/(2\beta-2)}$ for some constant $C$ large enough, we have
	\[
	\inf_{\hat{\btheta}} \sup_{(\btheta,P) \in \Theta_{\beta,s}}\PP_{P}\bigg[
	\norm{\hat{\btheta} - \btheta}_q \geq cs^{\frac{1}{q}-\frac{1}{2}}\bigg(\frac{s\log (d/s)}{n}\bigg)^{\beta/(2\beta+1)}
	\bigg] > c',
	\]
	for $q=1,2$, where $c,c'$ are constants independent of $n,d$ and $s$.
\end{theorem}
\begin{proof}
	See Section \ref{proof_minimax} for the  detailed proof.
\end{proof}

This theorem implies that for $\beta > 1$ our estimator is nearly minimax rate optimal in terms of $\ell_1$ and $\ell_2$ estimation errors. We note that our result does not cover the case $\beta\leq 1$, because we need $\beta>1$ for the differentiablity of the constructed conditional density in the proof.

\subsection{Comparison with related works}\label{banergee}

Under the fixed $d$ and diverging $n$ setting, the linear binary response model in (\ref{eqbinary}) has been extensively studied. For instance, \cite{kim1990cube} derived the cubic root rate for the maximum score estimator and \cite{horowitz1992smoothed} showed that the smoothed maximum score estimator has the rate $n^{-\frac{h}{2h+1}}$ under the assumption that the conditional density of $\epsilon$ given $\bW$ lies in the uniform H\"older class $\Sigma(h,L')$ for some constant $L'$.  Within the context of this binary response model, our paper is related to theirs in the sense that our estimator can be treated as an $\ell_1$ penalized smoothed maximum score estimator with normalization $\theta_1 = 1$. In the fixed dimensional setting, the proof of Theorem \ref{theorem_main2} implies that the rate of our estimator reduces to $(\log n/n)^{1/3}$ for $\beta=1$ and $(\log n/n)^{\frac{\beta}{2\beta+1}}$ for general $\beta>0$, which is comparable with the aforementioned results up to a $\log n$ term due to the regularization.

Very recently, \citet{ban2019} studied Manski's maximum score estimator in high-dimensional setting, which is closely related to this paper. In the notation of Example \ref{exp_4},
they proposed to maximize the empirical score function subject to an $\ell_0$ constraint
\[
\hat{\btheta}_{MS} = \argmax_{\substack{\btheta: \norm{\btheta}_2 = 1,\norm{\btheta}_0 \leq s_0}}\frac 1n \sum_{i = 1}^n y_i\sign{\btheta^T\bw_i},
\]
where $s_0$ is a pre-specified upper bound on the true sparsity. With simple algebra we can show that the above is equivalent to minimizing the empirical risk w.r.t. 0-1 loss function. They showed that when $p \gg n$, under some regularity conditions such as the soft margin condition \citep{mammen1999smooth}: $\PP(|\eta(\bW) - \frac12| \leq t) \leq Ct^\alpha~~\forall~0\leq t \leq t^*\leq 1/2,$ where $\eta(\bw) = \PP(Y = 1|\bW = \bw)$, and $\alpha \geq 1$ is some smoothness parameter, their estimator achieves the $\ell_2$ rate $\big(
\frac{s_0\log(d/s_0)\log n}{n}
\big)^{\frac{\alpha}{\alpha + 2}}.$

Compared to \citet{ban2019}, our theoretic results differ in many aspects. First, instead of based on the soft margin condition, our result requires the H\"older class smoothness condition for the conditional density.
Consequently, the characterization of the convergence rate of our estimator is intrinsically different from theirs. To be specific, let us consider the binary response model. Under mild conditions on the noise $\epsilon$ in (\ref{eqbinary}), the soft margin condition holds for $\alpha=1$, which leads to the cubic root rate by \citet{ban2019}.
On the contrary,
in this case if the conditional density $f(X|\bZ,Y)$ is indeed very smooth, Theorem \ref{theorem_main2} implies that the rate of our estimator can be much faster and even close to root-$n$; we refer to Section \ref{sup_ban} in the supplementary material for a concrete example.
On the other hand, based on the soft margin condition, their estimator may achieve a rate faster than the canonical root-$n$ rate if $\alpha>2$, whereas our estimator is always slower than the root-$n$ rate.


Second, for the lower bound, we show that the $\ell_2$ minimax rate is $\big(\frac{s\log(d/s)}{n}\big)^{\beta/(2\beta+1)}$ for $P\in\cP(\beta,L)$ with $\beta > 1$, whereas they showed that the minimax rate is $\big(\frac{s\log(d/s)}{n}\big)^{1/3}$ under the soft margin condition with $\alpha = 1$. Due to different characterization of parameter spaces in these two works, the way of constructing the alternative hypotheses in the proof of the lower bound is completely different. Since our lower bound result focuses on the more smooth case $\beta > 1$ (with faster rate), these two results indeed complement to each other.


Third, from the computational perspective, computing their estimator $\hat{\btheta}_{MS}$ is NP-hard due to discontinuity of the maximum score function and nonconvexity of the constraint set $\|\btheta\|_0\leq s_0$. In terms of implementation, they proposed to first apply penalized Adaboost to perform variable selection and then solve the smoothed maximum score on the selected covariates. Unlike their approach, the analysis of statistical performance of our estimator is embedded in the analysis of the path-following algorithm. As shown in Theorem \ref{theorem_main2}, our estimator has both statistical and computational guarantees.

\section{Adaptive estimation}\label{adaptive}

To achieve the optimal convergence rate in Section \ref{sec_theory}, it requires the knowledge of the smoothness parameter $\beta$ as well as the sparsity $s$ to choose appropriate $\delta \asymp
\big(\frac{s \log d  }{n}\big)^{1/(2\beta+1)}
$ and $\lambda_{tgt} \asymp\sqrt{\frac{\log d}{n\delta}}$ in the penalized smoothed empirical risk function.  However, in practice these two parameters are generally unknown.
Adaptive estimation for unknown smoothness in nonparametric kernel estimation and
 unknown sparsity in high-dimensional statistics has been studied by  \cite{lepskii1991problem,lepskii1993asymptotically,gine2010confidence,cai2014adaptive,su2016slope,bellec2018slope}, among others.
This section aims to develop adaptive estimation procedures for $\btheta^*$. 
Specifically, we show that when either $\beta$ or $s$ is unknown, by applying the Lepski's method \citep{lepskii1991problem,lepskii1993asymptotically,birge2001alternative}, the $\ell_2$ estimation error of the proposed method achieves the optimal rate
$\big(\frac{s \log d  }{n}\big)^{\beta/(2\beta+1)}$ up to a logarithmic factor.

\subsection{Adaptation for unknown \texorpdfstring{$\beta$}{Lg}}

Suppose the sparsity level $s$ is known or chosen appropriately with domain expertise. We focus on the adaption for the unknown smoothness $\beta$. For purpose of presentation, we denote $\tilde{\btheta}_{\delta,\lambda_{\delta}}$ by the approximate solution to the minimization problem (\ref{eq_penalized_obj}) with bandwidth parameter $\delta$ and tuning parameter $\lambda_\delta = C\sqrt{\frac{ \log d}{n\delta}}$. Consider the discrete set $\cD = \{1,\frac{1}{2},\dotso, \frac{1}{2^m} \}$ such that $\frac{1}{2^m} \leq \frac1n \leq \frac{1}{2^{m-1}}.$ Lepski's method seeks to choose the largest bandwidth in $\cD$ such that the bias is still dominated by the variance. Formally, the Lepski's method selects
\beq\label{leiski_1}
\hat{\delta} = \max_{\delta \in \cD } \bigg \{
\norm{\tilde{\btheta}_{\delta,\lambda_{\delta}} - \tilde{\btheta}_{\delta',\lambda_{\delta'}}}_2 \leq c \sqrt{\frac{s \log d}{n\delta'}}~~\forall \delta' \leq \delta, \delta' \in \cD
\bigg  \},
\eeq
for some constant $c > 0$. If the set to be maximized over is empty, we would choose $\hat{\delta} = 1/n$ by default. The following theorem shows the statistical performance of the estimator $\tilde{\btheta}_{\hat{\delta},\lambda_{\hat{\delta}}}$.

\begin{theorem}\label{theo_adapt_1}
	Suppose for each $\delta \in \cD$, $\tilde{\btheta}_{\delta,\lambda_{\delta}}$ is the approximate final estimator from the path-following algorithm with $\lambda_{\delta}$, a proper kernel function with order $l > \beta$ satisfying $\int |t^{\beta}||K(t)|dt <  \infty$ and the same parameters $\eta,\nu,\phi$ chosen in Theorem \ref{main_result1}. If Assumption \ref{assum_propensity} - \ref{RSC} hold, by choosing $\hat{\delta}$ following (\ref{leiski_1}),  and a large enough constant $C$ for each $\lambda_{\delta} = C \sqrt{\frac{\log d}{n\delta}}$, then with probability greater than $1 - 2\log_{2}(n)/d$ it holds that
	\[
	\norm{\tilde{\btheta}_{\hat{\delta},\lambda_{\hat{\delta}}} - \btheta^*}_2 \lesssim \bigg(\frac{s\log d  }{n}\bigg)^{\beta/(2\beta+1)}.
	\]
\end{theorem}
\begin{proof}
	See Section \ref{proof_adaptive1} for the  detailed proof.
\end{proof}

Theorem \ref{theo_adapt_1} suggests that with the data-driven bandwidth $\hat{\delta}$ the resulting estimator achieves the optimal rate $\big(
\frac{s \log d}{n}
\big)^{\beta/(2\beta+1)}$ up to a logarithmic factor. In this theorem, we assume that the order $l$ of the kernel function $K(\cdot)$ is chosen such that $l > \beta$. In practice, we can always choose a higher order kernel to meet this condition. Finally, we note that the same adaptation theory holds for the exact local solution $\hat\btheta_{tgt}$ provided the estimator $\tilde\btheta_{tgt}$ used in (\ref{leiski_1}) is replaced by $\hat\btheta_{tgt}$.

\subsection{Adaptation for unknown \texorpdfstring{$s$}{Lg}}
When we have information about the smoothness parameter $\beta$, a similar Lepski's procedure can be applied to adapt to the unknown sparsity $s$. Again, for purpose of presentation, we denote $\tilde{\btheta}_{\delta_s,\lambda_s}$ by the approximate solution to the minimization problem (\ref{eq_penalized_obj}) with bandwidth parameter $\delta_s = c \big(
\frac{s \log d}{n}\big)^{1/(2\beta+1)}$ and tuning parameter $\lambda_s = C\sqrt{\frac{\log d}{n\delta_s}}$. Consider $\cD' = \{1,2,\dotso,2^m \}$
such that $2^{m} \leq d \leq 2^{m+1}.$ We define the adaptive estimator $\hat{s}$ as

\beq\label{leiski_2}
\hat{s} = \min_{s \in \cD' } \bigg \{
\norm{\tilde{\btheta}_{\delta_{s},\lambda_{s}} - \tilde{\btheta}_{\delta_{s'},\lambda_{s'}}}_2 \leq \bar{c} \bigg(
\frac{s' \log d}{n}
\bigg)^{\beta/(2\beta+1)}~~\forall s' \geq s, s' \in \cD'
\bigg  \},
\eeq
where $\lambda_{s} = C\sqrt{\frac{\log d}{n\delta_{s}}}$ for some constant $C$ and $\delta_s = c \big(
\frac{s \log d}{n}
\big)^{1/(2\beta+1)}$ with $c^{\beta+0.5} \leq C$. Again if the set to be minimized over is empty, we choose $\hat s = 2^m$.

\begin{theorem}\label{theo_adapt_2}
	Suppose for each $s \in \cD'$, $\tilde{\btheta}_{\delta_{s},\lambda_{s}}$ is the approximate final estimator from the path-following algorithm with $\lambda_{s}$ and the same parameters $\eta,\nu,\phi$ chosen in Theorem \ref{main_result1}. If Assumption \ref{assum_propensity} - \ref{RSC} hold,   by choosing $\hat{s}$ following (\ref{leiski_2}), a proper kernel function with order $\floor{\beta}$ satisfying $\int |t^{\beta}||K(t)|dt <  \infty$ and a large enough $C$ for $\lambda_{s}$ in (\ref{leiski_2}), then with probability greater than $1 - 2\log_{2}(d)/d$ it holds that
	\[
	\norm{\tilde{\btheta}_{\delta_{\hat{s}},\lambda_{\hat{s}}} - \btheta^*}_2 \lesssim \bigg(\frac{s\log d  }{n}\bigg)^{\beta/(2\beta+1)}.
	\]
\end{theorem}
\begin{proof}
	See Section \ref{proof_adaptive2} for the  detailed proof.
\end{proof}

Similar to Theorem \ref{theo_adapt_1}, when the smoothness parameter $\beta$ is known, Theorem \ref{theo_adapt_2} implies that the bandwidth and regularization parameter can be adaptively chosen by a Lepski's type method to achieve the optimal $\ell_2$ estimation error. Finally, we note that if $s$ and $\beta$ are both unknown, the Lepski's method breaks down due to the lack of a strict ordering when the set (\ref{leiski_1}) or (\ref{leiski_2}) is defined in a multivariate grid. To the best of our knowledge, how to perform adaptive estimation in related settings is still an open problem. We leave it for future investigation.


\section{Proof of main results}\label{proofs}
In this section, we prove the main results in Section \ref{theory}, with the remaining proofs and auxiliary lemmas provided in the appendix.
\subsection{Proofs for the algorithm and upper bound}

The proofs of Theorem \ref{theorem_main2} and Corollary \ref{corollary_main1} are built upon the following Lemmas.

\begin{lemma}\label{lemma_initialization}
	Assume that the conditions of Proposition \ref{prop_bias} and Assumption \ref{RSC} hold. For $\lambda \geq \lambda_{tgt}$, if $\btheta \in \Omega$, $\norm{\btheta_{S^{*c}}}_0 \leq \tilde{s}$, $\omega_\lambda(\btheta) \leq \frac12\lambda,$ and $\norm{\nabla R_\delta^n(\btheta^*) - \nabla R_\delta(\btheta^*)}_\infty \leq \lambda/8$,
	we have
	\beq
	&\norm{\btheta - \btheta^*}_2 \leq \frac{\bar{C}_1}{\rho_{-}}\bigg( \delta^{\beta} \lor \sqrt{s}\lambda \bigg),\\
	&\norm{\btheta - \btheta^*}_1 \leq \frac{\bar{C}_2}{\rho_{-}}\bigg( \frac{\delta^{2\beta}}{\lambda} \lor \sqrt{s}\delta^{\beta} \lor   s\lambda  \bigg),\\
	&f_\lambda(\btheta) - f_\lambda(\btheta^*)\leq \frac{\bar{C}_2}{2\rho_{-}}\bigg( \delta^{2\beta} \lor \sqrt{s}\delta^{\beta}\lambda \lor   s\lambda^2  \bigg),
	\eeq
	where $f_\lambda(\btheta)$ denotes the objective function $R_\delta^n(\btheta) + \lambda\norm{\btheta}_1$ and $\bar{C}_1, \bar{C}_2 > 0$ are constants that depend on $C_2$ in Proposition \ref{prop_bias}.
\end{lemma}

\begin{proof}
	See Section \ref{proof_lemma_initialization} for a detailed proof.
\end{proof}
Lemma \ref{lemma_initialization} suggests that if $\btheta$ is sparse and is $\lambda/2$ optimal with respect to the suboptimality criterion defined in (\ref{equ_subopti}), then the distance between $\btheta$ and $\btheta^*$ as well as the corresponding function values can be characterized by $\delta$ and $\lambda$ along with the smoothness $\beta$ and sparsity $s$.  Recall that under the path-following framework, for each stage $t = 1,\dotso,N$, we exploit a warm start from the end of previous stage. Lemma  \ref{lemma_initialization} is exactly a quantification for the closeness of the initial value $\btheta_t^{0} = \tilde{\btheta}_{t-1}$ to $\btheta^*$ when $\omega_{\lambda_{t}}(\btheta_t^{0}) \leq \lambda_{t}/2. $ Note that this lemma is a deterministic result and here we do not specify the value of $\delta$, however it is easy to see that by taking $\delta = \big(
\frac{s \log d}{n}
\big)^{1/(2\beta+1)}$, each term within the three terms on the right hand side (RHS) are balanced to the same order.

The next two lemmas characterize the properties of the iterates  $\btheta_t^{1},\dotso$ at stage $t$.
\begin{lemma}\label{lemma_norm}
	Assume that the conditions of Proposition \ref{prop_bias} and Assumption \ref{RSC} hold. For $\lambda \geq \lambda_{tgt}$, if $\norm{\nabla R_\delta^n(\btheta^*) - \nabla R_\delta(\btheta^*)}_\infty \leq \lambda/8$, $\btheta \in \Omega$, $\norm{\btheta_{S^{*c}}}_0 \leq \tilde{s},$ and $f_\lambda(\btheta) - f_\lambda(\btheta^*) \leq \frac{\bar{C}_2}{2\rho_{-}}\big( \delta^{2\beta} \lor \sqrt{s}\delta^{\beta}\lambda \lor   s\lambda^2  \big)$ then we have
	\beq
	&\norm{\btheta - \btheta^*}_2 \leq \frac{\bar{C}'_1}{\rho_{-}}\bigg( \delta^{\beta} \lor s^{1/4}\sqrt{\delta^{\beta}\lambda}   \lor \sqrt{s}\lambda \bigg ),\\
	&\norm{\btheta - \btheta^*}_1 \leq \frac{2\bar{C}_2}{\rho_{-}}\bigg( \frac{\delta^{2\beta}}{\lambda} \lor \sqrt{s}\delta^{\beta} \lor   s\lambda  \bigg),
	\eeq
	where $\bar{C}'_1$ depends on $\bar{C}_1, \bar{C}_2$, which are constants defined in Lemma \ref{lemma_initialization}.
\end{lemma}
\begin{proof}
	See Section \ref{proof_lemma_norm} for a detailed proof.
\end{proof}
\begin{lemma}\label{lemma_sparse}
	Under the same conditions of Lemma \ref{lemma_norm}, if we choose $\lambda_{tgt} = C\sqrt{\frac{\log d}{n\delta}}$ for some large enough constant $C$, 
	and $\delta = c\big(\frac{s \log d}{n}\big)^{1/(2\beta+1)}$ for some constant $c > 0$, then
	\[
	\norm{\cS_{\lambda \eta}(\btheta,\RR^d)_{S^{*c}}}_0 \leq \tilde{s},
	\]
	where $\tilde{s} = 8\big(\frac{\bar{C}_2}{\eta\rho_{-}} + \frac{2\bar{C}'^2_1\rho_{+}^2}{\rho^2_{-}}+ 2C_2^2 \big)\cdot s $, and $\bar{C}'_1$, $\bar{C}_2$, $C_2$ are constants defined in Lemma \ref{lemma_norm}, Lemma \ref{lemma_initialization} and Proposition \ref{prop_bias}, respectively.
\end{lemma}
\begin{proof}
	See Section \ref{proof_lemma_sparse} for a detailed proof.
\end{proof}
Lemma \ref{lemma_norm} and \ref{lemma_sparse} together suggest that if the initialization at stage $t$ is sparse and satisfies $\omega_{\lambda_{t}}(\btheta_t^{0}) \leq \lambda_{t}/2$, then the next iterate $\btheta_t^{(1)}$ should also be sparse and have nice statistical properties. Under Assumption \ref{RSC}, we can also show that the objective function values $f_\lambda(\btheta_t^{0}), f_\lambda(\btheta_t^{1}),\dotso$ are decreasing (see Lemma \ref{lemma_monotone} in Section \ref{sec_appen_algo}), so the conditions in Lemma \ref{lemma_norm} and \ref{lemma_sparse} also hold for the whole path of iterates $\btheta_t^{(1)},\dotso$, ensuring sparsity and convergence to a local solution. This is formally stated in the following proposition.

\begin{proposition}\label{prop_key}
	Under Assumption \ref{assum_propensity} - \ref{RSC} , suppose $K$ is a proper kernel of order $l = \floor{\beta}$ satisfying $\int |t^{\beta}||K(t)|dt <  \infty$, $\eta \leq \frac{1}{\rho_{+}}$ and $\delta = c\big(\frac{s \log d}{n}\big)^{1/(2\beta+1)}$ for some constant $c > 0$.  If $\norm{\nabla R_\delta^n(\btheta^*) - \nabla R_\delta(\btheta^*)}_\infty \leq \lambda_{tgt}/8 \leq \lambda_{t}/8$ and at stage $t$, the proximal gradient method is initialized with $\btheta^0_t\in \Omega$ satisfying
	$$\norm{(\btheta^0_t)_{S^{*c}}}_0 \leq \tilde{s}\; \text{and}\; \omega_{\lambda_t}(\btheta^0_t) \leq \frac12 \lambda_t,$$
	then for $k = 1,2,\dotso,$ we have
	\begin{itemize}
		\item
		$\norm{(\btheta^k_t)_{S^{*c}}}_0 \leq \tilde{s}.$ 
		\item The sequence $\{\btheta_t^k\}_{k = 0}^\infty$ converges towards a unique local solution $\widehat{\btheta}_t$ satisfying the first-order optimality $\omega_{\lambda_t}(\widehat{\btheta}_t) \leq 0$ with $\norm{(\widehat{\btheta}_t)_{S^{*c}}}_0 \leq \tilde{s}$.
		\item $f_{\lambda_t}(\btheta_t^{k})- f_{\lambda_t}(\widehat{\btheta}_t) \leq (1-\frac{\eta\rho_{-}}{4})^k(f_{\lambda_t}(\btheta_t^{0} ) - f_{\lambda_t}(\widehat{\btheta}_t))$.
	\end{itemize}
	
\end{proposition}
\begin{proof}
	See Section \ref{proof_prop_key} for a detailed proof.
\end{proof}

The first two results in Proposition \ref{prop_key} suggest that sparsity is ensured for the whole sequence of iterates $\btheta^k_t$ at stage $t$, which converges to a unique local solution $\hat{\btheta}_t$, when initialized well. The third result implies geometric rate of convergence for a single stage of the path-following algorithm. Consequently, by computing $\tilde{\btheta}_{t-1}$ up to a certain precision, 
we would expect it to be close enough to the exact solution $\hat{\btheta}_{t-1}$ which lies in a fast convergence region for an exact solution $\hat{\btheta}_{t}$ of the next stage. By repeating this process for the whole regularization path until the target $\lambda_{tgt}$ achieves the optimal statistical performance, we will stop and obtain the final estimator $\tilde{\btheta}_{tgt}$. Indeed, the aforementioned argument is rigorously proved in Theorem \ref{theorem_main2}. See Section \ref{proof_theo_main2} for a detailed proof.

\subsection{Proof of minimax lower bound}

\begin{proof}\label{proof_minimax}
	The proof can be summarized in two steps.
	\begin{enumerate}
		\item We construct a finite set of hypotheses $\cH = \big \{(\btheta_i,P_i(X,Y,\bZ))  \big \} \subset \Theta_{\beta,s}$.
		\item We apply Theorem 2.7 in \citet{tsybakovintroduction} to show the desired results by checking the following conditions:
		\begin{enumerate}[label=(\Alph*)]
			\item $KL(P_{j},P_{0}) \leq \gamma \log|\cH|$ for some $\gamma \in (0,1/8)$, where $KL(P,Q)$ is the Kullback divergence between probability two measures $P$ and $Q$.
			\item For all $j\neq k$ and $q = 1,2$, $\norm{\btheta_j - \btheta_k}_q \geq 2t$, where $t \asymp s^{\frac 1q-\frac 12}\big(\frac{s\log (d/s)}{n}\big)^{\beta/(2\beta+1)}$.
		\end{enumerate}
	\end{enumerate}
	\textbf{Step 1:}

	For simplicity, throughout we define $c$ as a positive constant that may vary from line to line.
	Consider the set
	\[
	\cM= \{x \in \{0,1\}^{d}: \norm{x}_0 = s\}.
	\]
	It follows from Varshamov-Gilbert bound (see Lemma 2.9 in \cite{tsybakovintroduction}) and Lemma 8 in \cite{rey2013} that there exists a subset $\cH'$ of $\cM$ such that $\rho_H(x,x') > s/16$ for $x,x' \in \cH',x\neq x'$ and
	\begin{equation}\label{eqGilbert}
	\log|\cH'| \geq c' s \log(\frac{d}{s}),
	\end{equation}
	where $\rho_H$ denotes the Hamming distance and $c'$ is some absolute constant.
	Then we let $\bomega_0 = \mathbf{0} \in \RR^d$ and use $\bomega_j$ to denote the elements in $\cH'$ for $j = 1,\dotso,|\cH'|$.
	
	Now we start to construct $P_j(X,Y,\bZ)$. In the sequel, we set
	\beq\label{eq_minimax_2}
	\delta = \tilde{c}\bigg(\frac{s\log (d/s)}{n}\bigg)^{1/(2\beta+1)},
	\eeq
	where $\tilde{c}$ is a sufficiently small constant.	For all $j = 0,\dotso,|\cH'|$, we assume $\PP(Y = 1) = \PP(Y = -1) = 0.5$ and each of $Z_1,\dotso,Z_d$ follows a Uniform distribution on $[-1,1]$ independently.
	Recall that we choose $w(y) = 1/\PP(Y=y)$ as the weight function.
	For the probability density $f(x|y,\bz)$ of $X$ conditioned on $Y$ and $Z$,  consider the function
	\[
	t(x) = \int_{-\infty}^{x}\frac{h(u)}{h(u) + h(1-u)}du,
	\]
	where
	\[
	h(x) = \begin{cases}
	e^{-1/x}  &\text{if} ~x > 0,\\
	0 &\text{if} ~x \leq 0.
	\end{cases}
	\]
	The function $t(x)$ defined above satisfies the following properties:
	\begin{itemize}
		\item  $t(x) \in C^{\infty}$ on $\RR$.
		\item  $t(x) \geq 0$ and $t(x)$ is non-decreasing.
		\item $t(x) = 0~\forall~x \leq 0$ and $t'(x) = 1~\forall~x\geq 1$.
	\end{itemize}
	Now we define
	\[
	k(x) = \begin{cases}
	t(x+4)&\text{if} ~x \leq -1,\\
	2t(1) + 2 - t(-x) &\text{if} ~ -1 <x \leq 0,\\
	2t(1) + 2 - t(x)&\text{if} ~ 0<x\leq 1,\\
	t(-x+4)&\text{if} ~x > 1,
	\end{cases}
	\]
	and $g(x) = \frac{m}{\sigma}k(x/\sigma)$, where $m = 1/\int k(t)dt$ is an absolute normalizing constant. By $g(x)\geq 0$, then $g(x)$ is a proper density function. It's easy to check that $t^{(l)}(x) = 0 ~\forall~x<0,l = 0,1,\dotso,$ and $t^{(l)}(x) = 0 ~\forall~x> 1,l = 2,3,\dotso$. Since $t(x) \in C^{\infty}$ on $\RR$, $t(x)$ is also infinitely differentiable on $[0,1]$ which is compact. This implies that $t^{(l)}(x)$ is bounded for $l = 0,1,\dotso,\floor{\beta}$. Thus by choosing $\sigma > 0$
	 large enough, we can eusure  that $g(x) \in \Sigma (\beta,L/2)$, where recall that $\Sigma (\beta,L)$ denotes the H\"older class with constant parameters $\beta$ and $L$.
	
	Now let $H(u) = \exp(-\frac{1}{1 - u^2})\mathds{1}(|u|\leq 1)$ and define $R(u) = a(H(u) - H(u-1))$, where $a > 0$ is a sufficiently small constant.
	The function $R(u)$ above satisfies the following properties (see also Section 2.5 in \citet{tsybakovintroduction} )
	\begin{itemize}
		\item $R(u) \in C^{\infty}$ on $\RR$.
		\item $R(u)\leq a H_{\max}$ for some constant $H_{\max}<\infty$.
		\item $\int R(u) du = 0$.
		\item $R(u) = 0$ for $u \in (-\infty,-1] \cup \{\frac 12 \} \cup [2,\infty)$.
	\end{itemize}
	With $a > 0$ being sufficiently small, $R(u) \in C^\infty \cap \Sigma(\beta,1/2)$ will hold.
	
	Now we are ready to construct our hypotheses. For each $j = 0,\dotso,|\cH'|$, we let
	
	\[
	f_j(x|y = 1,\bz) = \begin{cases}
	g(x-2\sigma) &\text{if}~ j = 0,\\
	g(x-2\sigma) +  \tilde{g}_{j,\bz}(x)&\text{otherwise},
	\end{cases}
	\]
	where
	$$\tilde{g}_{j,\bz}(x) = L\delta^\beta R(\frac{x - x_{j,\bz}}{\delta })\frac{\bomega_j^T\bz}{\sqrt{s}}$$
	with $x_{j,\bz} = \frac{\delta^\beta L\sigma^2 ae^{-1}}{2m\sqrt{s}} \bomega_j^T\bz - \delta$.                                                                                                                                                                                                   
	Meanwhile, for $j = 0,\dotso,|\cH'|$, we let $f_j(x|y = -1,\bz) = g(x+2\sigma)$.
	We will first show that
	\beq\label{eq_minimax_1}
	\tilde{g}_{j\bz}(x) = 0 ~\forall~ x \in (-\infty,-\sigma] \cup [\sigma,\infty], \bz \in [0,1]^d.
	\eeq
	To see this, for any $|x|>\sigma$, we have
	\begin{align*}
	\Big|\frac{x - x_{j,\bz}}{\delta}\Big|\geq \frac{\sigma - |x_{j,\bz}|}{\delta}&\geq \frac{\sigma}{\delta}-\frac{\delta^{\beta-1} L\sigma^2 ae^{-1}}{2m\sqrt{s}} |\bomega_j^T\bz| - 1\\
	&\geq \frac{\sigma}{\delta}-\frac{\delta^{\beta-1} L\sigma^2 ae^{-1}\sqrt{s}}{2m} - 1,
	\end{align*}
	where the last step follows from the Holder inequality and $\|\bomega_j\|_1=s$.
	With a sufficiently large constant $C$ in the condition of Theorem \ref{theorem_minimax} that is independent of $d,n$ and $s$, we can ensure $\sqrt{s}\delta^{\beta-1}$ is bounded and thus (\ref{eq_minimax_1}) holds by the 4th property of $R(u)$.
	
	{
	Next, by the definition of $\tilde{g}_{j,\bz}$ and boundedness of $\sqrt{s}\delta^{\beta-1}$, with $a$ in the definition of $R(u)$ being sufficiently small, we can ensure that
	\beq \label{eq_minimax_3}
	\norm{\tilde{g}_{j,\bz}}_\infty \leq \frac{mt(1)}{2\sigma} = \frac 12 g(-3\sigma) ~\forall~\bz \in [0,1]^d.
	\eeq
	In the following, we apply (\ref{eq_minimax_1}) and (\ref{eq_minimax_3}) to show $f_j(x|y = 1,\bz)\geq 0$ for $x\in\RR$. Since $g(\cdot)$ is non-negative and (\ref{eq_minimax_1}) holds, it suffices to show it for $x\in[-\sigma,\sigma]$. Therefore, for $j\geq 1$
	\beq\label{eq_minimax_4}
	f_j(x|y = 1,\bz) = g(x-2\sigma) +  \tilde{g}_{j,\bz}(x) \geq g(x-2\sigma)- \frac 12 g(-3\sigma)\geq 0,
	\eeq
	where the second step follows from (\ref{eq_minimax_3}) and the last step follows from the definition of $g(\cdot)$ and the 2nd property of $t(x)$. Based on the above derivation, we conclude that
	\begin{itemize}
	    \item $f_0(x|y = 1,\bz)$ and $f_j(x|y = -1,\bz)$ for $j\geq0$ are well-defined density functions, since $g(x)$ is a well-defined density function.
	    \item $f_j(x|y = 1,\bz)$ for $j\geq1$ are well-defined density functions by (\ref{eq_minimax_4}) and
	    $$
	\int f_j(x|y = 1,\bz) dx=\int g(x-2\sigma)dx +\int L\delta^\beta R(\frac{x - x_{j,\bz}}{\delta })\frac{\bomega_j^T\bz}{\sqrt{s}} dx=1,
	$$
	where we use the 3rd property of $R(u)$.
	\end{itemize}{}
	}
	Therefore, the construction of $P_0,\dotso,P_{\cH'}$ are well-defined.
	The following lemmas characterize two key properties of $P_j$.

	\begin{lemma}\label{lemma_holder_verify}
		Under the conditions of Theorem \ref{theorem_minimax} and the construction of $P_j = P_j(X,Y,\bZ)$ above,
		 $P_j \in \cP(\beta,L)~\forall~j = 0,\dotso,|\cH'|$.
	\end{lemma}
\begin{proof}
	See Section \ref{proof_holder_verify} for the detailed proof.
\end{proof}
	
	\begin{lemma}\label{lemma_sparse_mini}
		Under the conditions of Theorem \ref{theorem_minimax} and the construction of $P_j = P_j(X,Y,\bZ)$ above, the unique minimizer $\btheta_j \in \RR^d$ of the risk $R_{P_j}(\btheta)$ is
		\[
		\btheta_j = \begin{cases}
		0 &\text{if}~j = 0,\\
		\frac{\delta^\beta L\sigma^2 a e^{-1}}{2m\sqrt{s}} \bomega_j &\text{otherwise}.
		\end{cases}
		\]
		In addition, $\lambda_{\min}(\nabla^2R_P(\btheta_j)) \geq \rho$ for some $\rho>0$.
	\end{lemma}
\begin{proof}
	See Section \ref{proof_lemma_sparse_mini} for a detailed proof.
\end{proof}

To be specific, Lemma \ref{lemma_holder_verify} implies that the constructed probability measure $P_j$ satisfies  smoothness condition in Assumption \ref{assum_density}, and Lemma \ref{lemma_sparse_mini} shows that the minimizer of $R_{P_j}(\btheta)$ exists, is unique and sparse by the construction of $\bomega_j$. Thus, our hypotheses satisfy
	\[
	\cH = \Big \{ (\btheta_j,P_j): j = 0,\dotso,|\cH'| \Big \}\subset \Theta_{\beta,s}.
	\]

	\textbf{Step 2:}
	In the sequel, we check the two conditions for the second step.\\
	\textbf{Condition (A):}\\
	Under the distributional assumption, we have
	\beq
	&KL(\PP_j,\PP_0)\\ \stackrel{(1)}{=}&\frac{n}{2}\int\int \bigg[
	g(x - 2\sigma) + L\delta^\beta R(\frac{x - x_{j,\bz}}{\delta})\frac{\bomega_j^T\bz}{\sqrt{s}}
	\bigg]\log\bigg(1 +
	\frac{ L\delta^\beta R(\frac{x - x_{j,\bz}}{\delta})\frac{\bomega_j^T\bz}{\sqrt{s}}}{g(x - 2\sigma) }
	\bigg)dxf(\bz)d\bz\\
	\stackrel{(2)}{\leq}&\frac{n}{2}\int\int
	\Big[L\delta^\beta R(\frac{x - x_{j,\bz}}{\delta})\frac{\bomega_j^T\bz}{\sqrt{s}} + \frac{L^2\delta^{2\beta} R^2(\frac{x - x_{j,\bz}}{\delta})\frac{|\bomega_j^T\bz|^2}{s} }{g(x - 2\sigma) }\Big]dxf(\bz)d\bz\\
	\stackrel{(3)}{\leq}&  \frac{n}{2}\int\int \frac{L^2\delta^{2\beta} R^2(\frac{x - x_{j,\bz}}{\delta})\frac{|\bomega_j^T\bz|^2}{s} }{g(x - 2\sigma) }dxf(\bz)d\bz\\		\stackrel{(4)}{\leq}&  \frac{3\delta R^2_{\max} \cdot nL^2\delta^{2\beta}}{2\inf_{(x - x_{j,\bz}/\delta) \in [-1,2]}g(x-2\sigma)} \EE [(\bw_j^T\bZ)^2/s] \\
	\stackrel{(5)}{\leq}& \frac{nL^2R^2_{\max}\sigma}{mt(1)}\delta^{2\beta+1},
	\eeq
	where step (1) follows the definition, step (2) applies the inequality $\log(1+x)\leq x$, step (3) follows from $\int R = 0$, step (4) follows from $\int R^2(\frac{x - x_{j,\bz}}{\delta})dx=\delta \int R^2(u)du\leq 3\delta R^2_{\max}$, and for (5), notice that (\ref{eq_minimax_1}) and (\ref{eq_minimax_3}) together imply that
	\[
	\inf_{(x - x_{j,\bz}/\delta) \in [-1,2]}g(x-2\sigma) \geq g(-3\sigma) - \norm{\tilde{g}_{j,\bz}}_{\infty} \geq \frac 12 g(-3\sigma) = \frac{mt(1)}{2\sigma}.
	\]

	Therefore by choosing a proper constant $\tilde{c}$ in (\ref{eq_minimax_2}) that is independent of $d,n$ and $s$, we can ensure that
	\beq
	KL(\PP_j,\PP_0)\leq \gamma c's\log (d/s) \leq \gamma\log |\cH|,
	\eeq
	where the last step follows from (\ref{eqGilbert}) which implies that condition (A) holds.
	
	\textbf{Condition (B):}\\
	By Lemma \ref{lemma_sparse_mini}, for all $j \neq 0$, we have
	\beq
	\norm{\btheta_0 - \btheta_j}_2 &= c\delta^{\beta}\norm{\bomega_j}_2/\sqrt{s} = c\delta^{\beta},\\
	\norm{\btheta_0 - \btheta_j}_1 &= c\delta^{\beta}\norm{\bomega_j}_1/\sqrt{s} = c\sqrt{s}\delta^{\beta}.
	\eeq
	And for all $j,k \neq 0$, we have
	\beq
	\norm{\btheta_j - \btheta_k}_2 &= c\delta^{\beta}\norm{\bomega_j - \bomega_k}_2/\sqrt{s} \geq  c\delta^{\beta},\\
	\norm{\btheta_j - \btheta_k}_1 &= c\delta^{\beta}\norm{\bomega_j - \bomega_k}_1/\sqrt{s} \geq  c\sqrt{s}\delta^{\beta}.
	\eeq

	Therefore Condition (B) holds. This proof is completed by applying Theorem 2.7 of \cite{tsybakovintroduction}.
\end{proof}


\section{Simulation studies}\label{simulation}
In this section, we conduct simulation studies to evaluate the finite sample performance of the proposed approach. We consider the following two classes of models:
\begin{itemize}
	 \item \textbf{Binary response model:}
	We consider $Y = \sign{\tilde{Y}},$ where
	\[\tilde{Y} = X - \btheta^T\bZ + u\]
	$X \in \RR, \bZ \in \RR^d$, and $u$ is a random noise such that $
	\text{Median}(\tilde{Y}|X,\bZ) = X - \btheta^T\bZ.$
	\item \textbf{Conditional mean model:} We consider $Y \in \{-1,1 \}$, $\bZ \in \RR^d$, and
	\[
	X= \mu Y + \btheta^T\bZ + u,
	\]
	where $u\independent Y,\bZ$ is a random noise and $\mu > 0$ is a constant.
\end{itemize}
Note that for both models, it can be shown that the parameter $\btheta$ coincides with the estimand in (\ref{MCID_linear}) with equal weights. Classical models such as logistic and probit regression belong to the class of binary response model by setting $u$ to follow logistic or Gaussian distribution independent with $(X,\bZ)$, respectively. For simplicity, we choose $u$ to be Gaussian to demonstrate the numerical performance.

Under each model, we simulate i.i.d. samples with sample size $n = 2000$ and dimension $d = 2500$. We set the sparsity $s = \sqrt{d} = 50$ and generate $\btheta^*$ by setting the first fifty coordinates to be $1$, and then normalize it such that $\norm{\btheta^*}_2 = 1$. For binary response model, we generate $X\sim N(0,1)$, $\bZ \sim N_d(0,\bI)$, and $u \sim N(0,(0.1)^2)$. For conditional mean model, we generate $Y \sim \text{Uniform}(\{-1,1\})$, $\bZ \sim N_d(0,\bI)$, $u \sim N(0,(0.1)^2)$ and set $\mu = 2$.
To apply the path-following algorithm, we set the bandwidth parameter $\delta = 1$ and
use the standard Gaussian density as the kernel function $K$. We fix the number of regularization stages $N = 10$ and choose $\nu = 1/4$, $\phi = \big( \lambda_{tgt}/\lambda_0\big )^{1/N}$ and $\eta = 1$  for Algorithm \ref{pf_algo}. For all scenarios, the tuning parameter $\lambda_{tgt}$ is chosen by 5-fold cross-validation based on the surrogate loss with Gaussian kernel following the ``one standard error rule'', i.e., we pick  $\lambda_{tgt}$ as the largest $\lambda$ over a grid that is within one standard error of the minimum cross validation error.

\begin{figure}[ht]
	\centering
	{\subfigure[]{\label{fig:a}\includegraphics[width=70mm,height = 50mm]{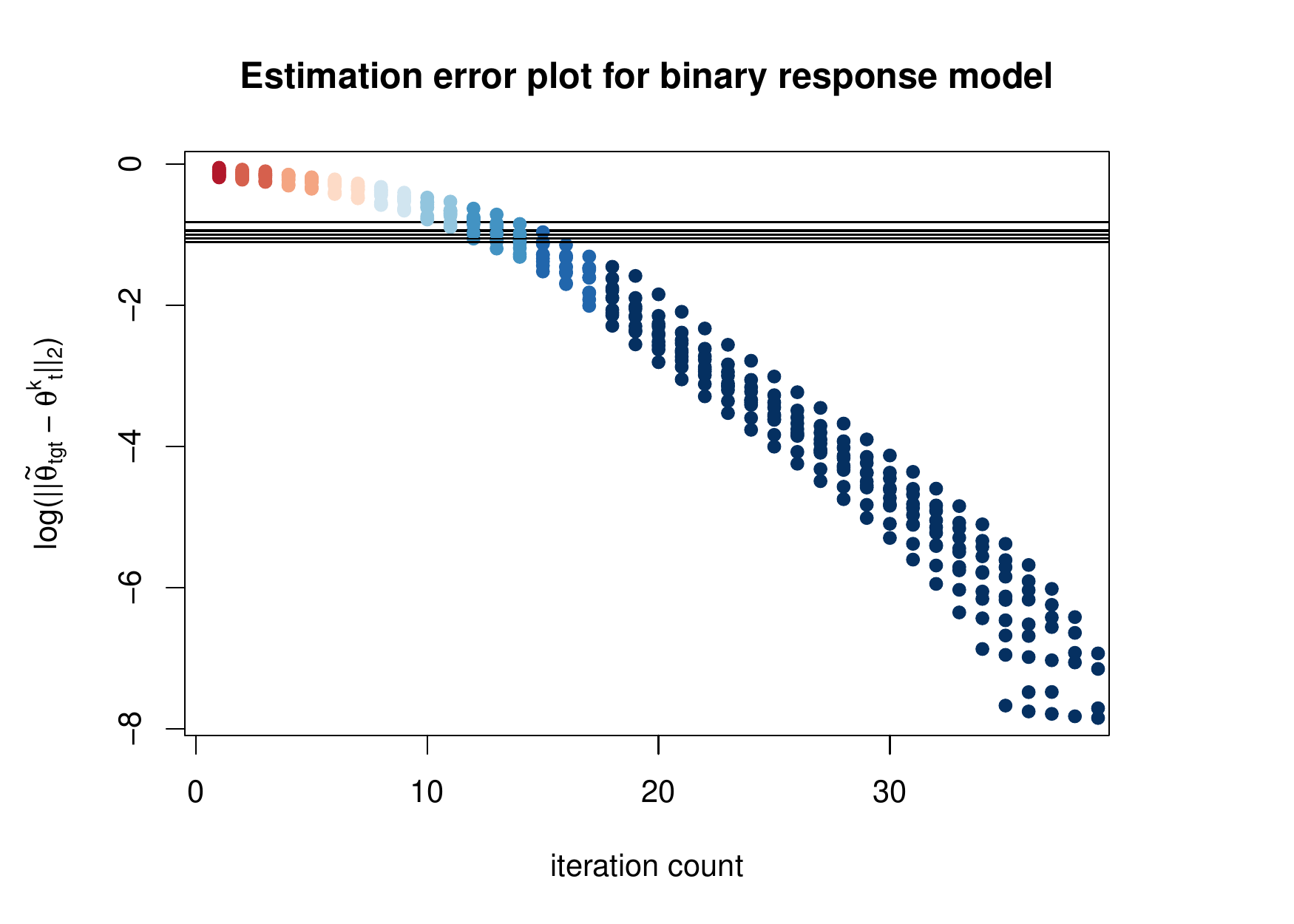}}
		\subfigure[]{\label{fig:b}\includegraphics[width=70mm,height = 50mm]{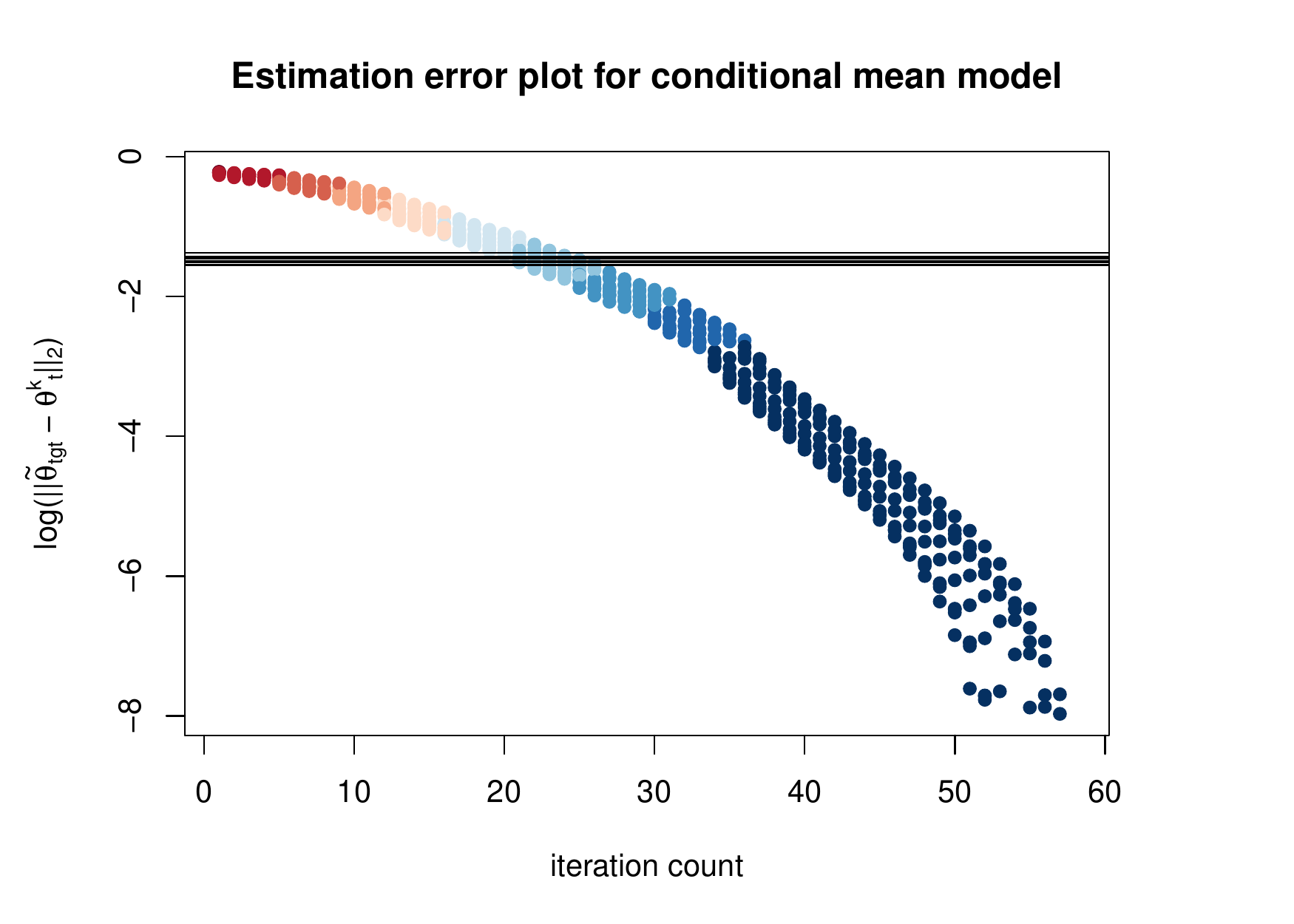}}}
	\caption{The optimization and statistical estimation error against the number of proximal
		-gradient iterates. The black horizontal lines represent the statistical estimation error $\log(\norm{\tilde{\btheta}_{tgt} - \btheta^*}_2)$, and the colored dots depict the optimization error $\log(\norm{\tilde{\btheta}_{tgt} - \btheta^k_t}_2)$ along the regularization path, where $\tilde{\btheta}_{tgt}$ is the estimator calculated by the path-following algorithm. 
	}   \label{esti_error}
\end{figure}

\begin{figure}[ht]
	\centering
	{\subfigure[]{\includegraphics[width=70mm,height = 50mm]{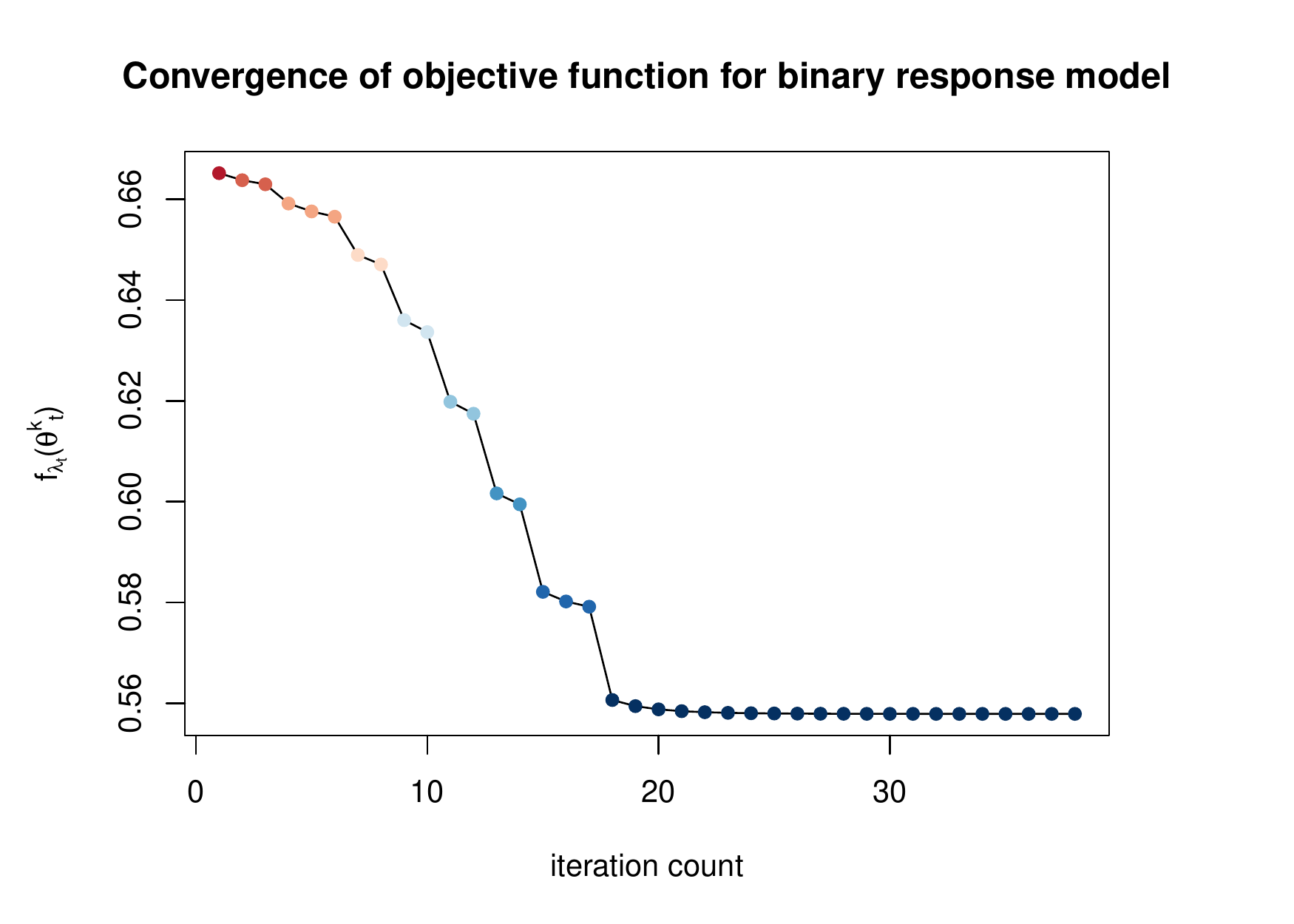}}
		\subfigure[]{\includegraphics[width=70mm,height = 50mm]{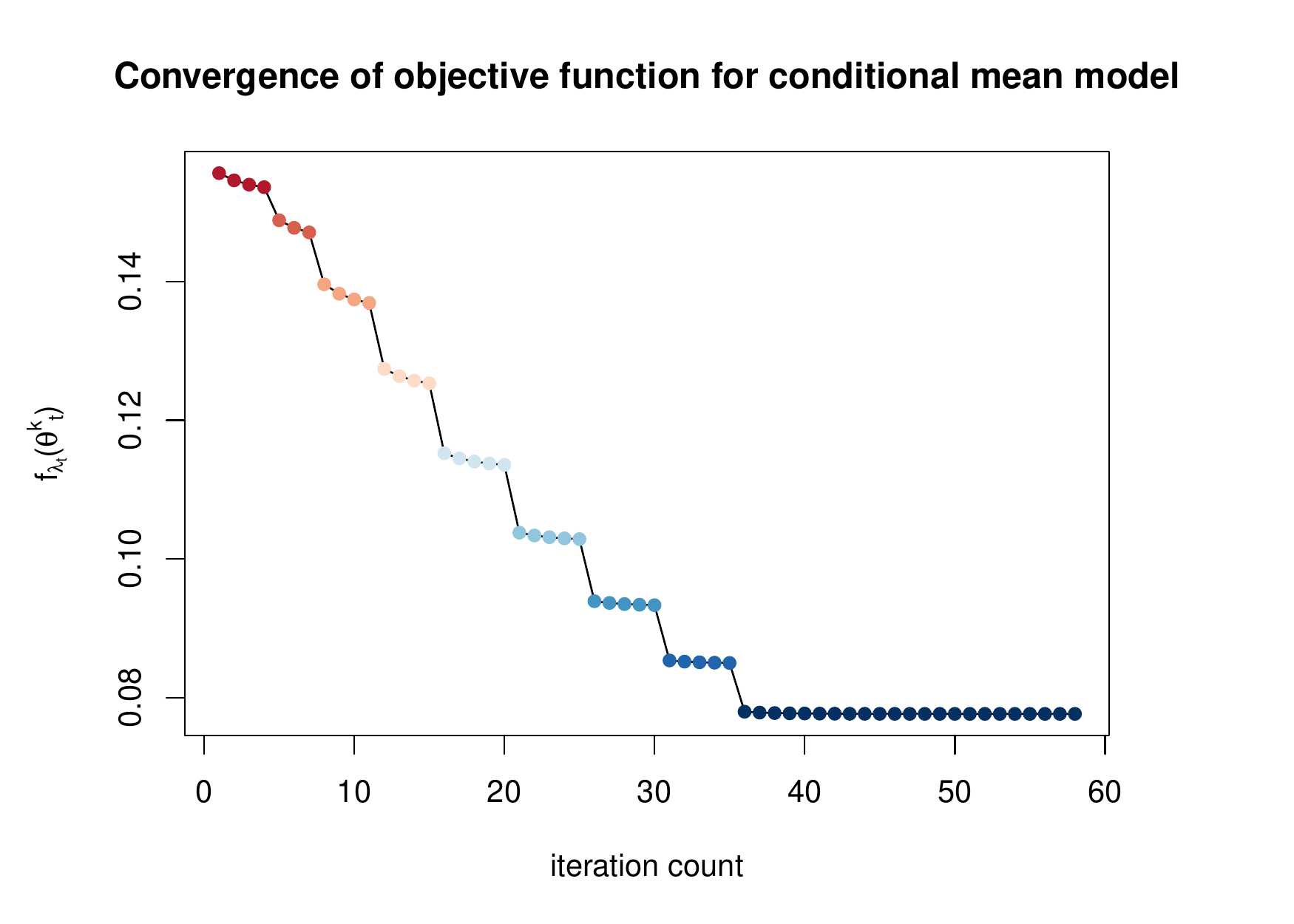}}}
	\vskip \baselineskip
	{\subfigure[]{\includegraphics[width=70mm,height = 50mm]{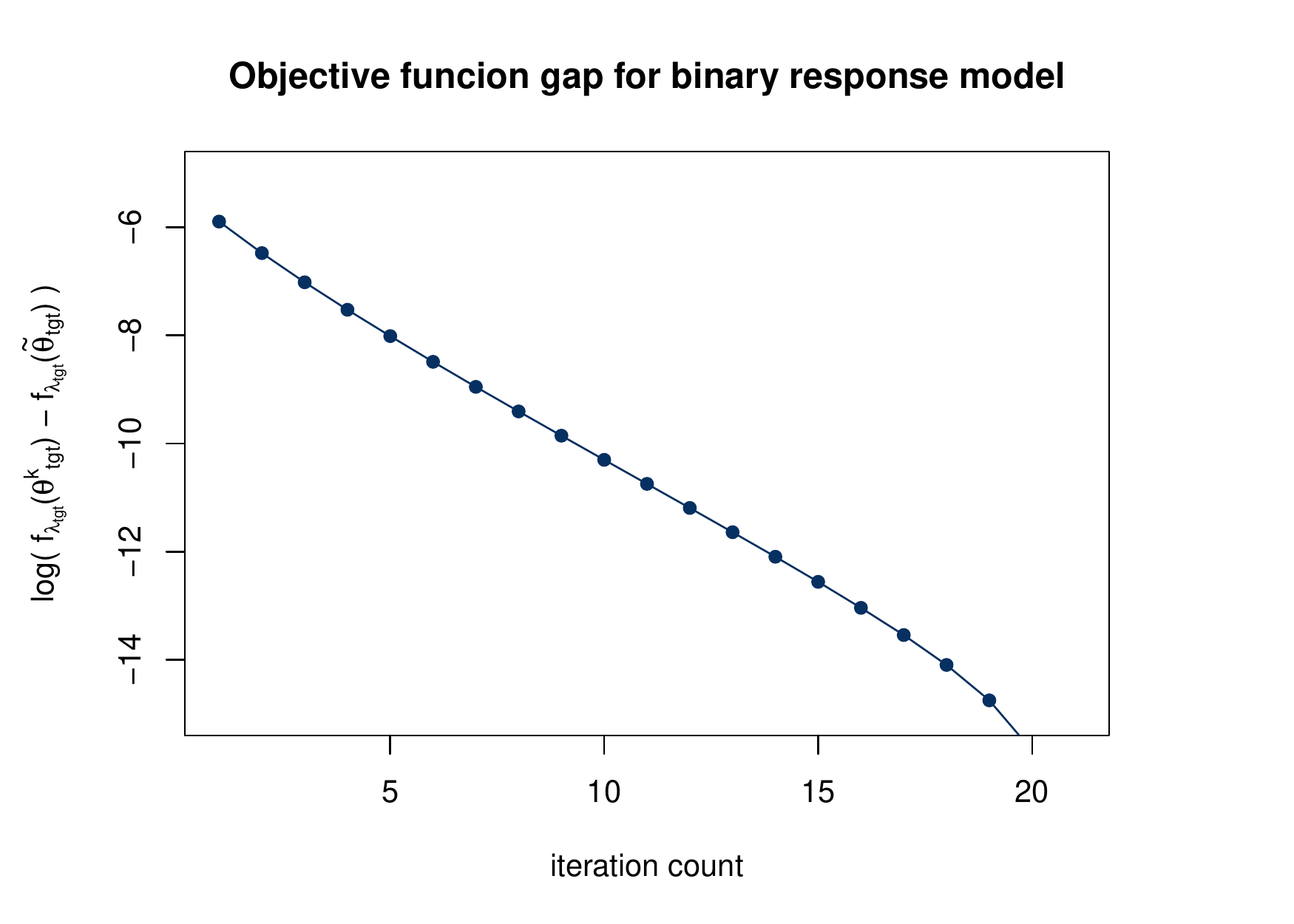}}
		\subfigure[]{\includegraphics[width=70mm,height = 50mm]{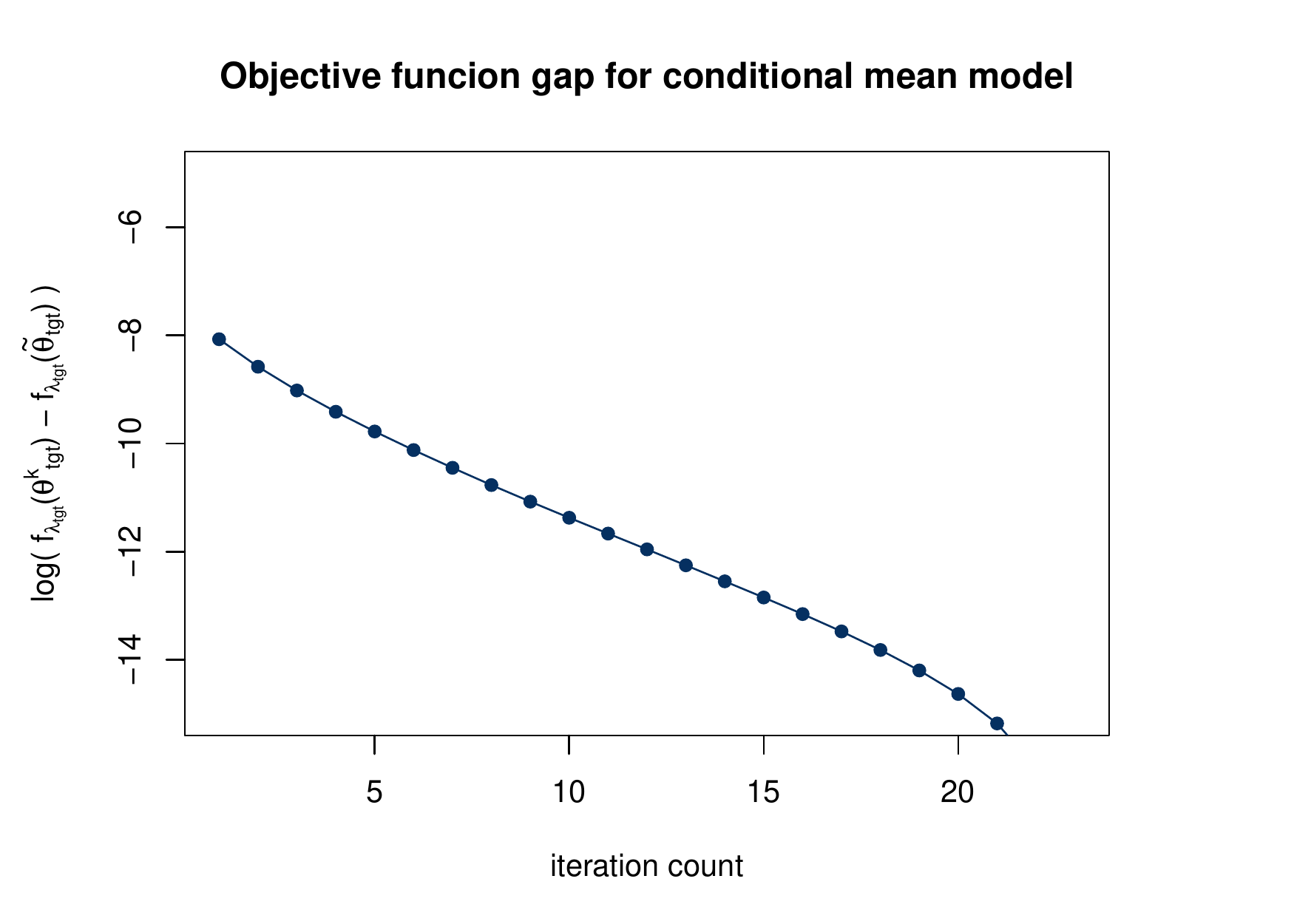}}}
	\caption{Panel (a) and (b): The objective function value of a single regularization path. Each color depicts a single regularization stage. Panel (c) and (d): The objective function value gap $\log(f_{\lambda_{tgt}}(\btheta_{tgt}^k) - f_{\lambda_{tgt}}(\tilde{\btheta}_{tgt}) )$ for the last regularization stage.
	}   \label{func_convergence}
\end{figure}

Figure \ref{esti_error} shows the statistical and optimization error for both models, where each plot shows the results from 10 random generated sets of samples. In each plot, the black horizontal lines depict the final statistical estimation error, measured by $\log(\norm{\tilde{\btheta}_{tgt} - \btheta^*}_2)$. The colored dots depict the optimization error for each run of proximal-gradient method, represented by $\log(\norm{\tilde{\btheta}_{tgt} - \btheta^k_t}_2)$, where each color corresponds to a single stage. These plots illustrate fast convergence of the path-following algorithm on a log scale for both cases considered.

Figure \ref{func_convergence} shows the pattern of the decreasing objective function values along a single regularization path for each of the two models, where each color also denotes a single stage $t$. From panel (a) and (b), we can see that the objective function is monotone decreasing, which validates the results in Proposition \ref{prop_key}. In addition, panel (c) and (d) depict the linear convergence of the objective function value in log scale, as proved in Theorem \ref{theorem_main2}.

In the next experiment, we compare the estimation performance of the proposed method with $\ell_1$ penalized logistic regression  and $\ell_1$ penalized support vector machine (SVM), under both conditional mean model and binary response model. For both models, we consider low-dimensional case ($d = 64$, $s = 8$) and high-dimensional case ($d = 2500$, $s = 50$) with sample size $n = 2000$. We apply the same data generating procedure as in the previous experiment for conditional mean model. However, under the binary response model, for a direct comparison between  logistic regression and the proposed method, we apply the same data generating procedure for $X$ and $\bZ$, except that we simulate the noise $u$ from standard logistic distribution. In this case, the underlying true model is correctly specified by logistic regression.  The simulation is repeated 100 times.

\begin{table}[ht]
	\caption{Comparison of estimation error for conditional mean model and binary response model with logistic noise among the proposed method, $\ell_1$ penalized logistic regression (Logit) and $\ell_1$ penalized SVM. The number in the parentheses are standard deviations. }\label{table_error}
	\footnotesize
	\begin{center}
		\begin{tabular}{ c c c c c c c c}
			\hline
			Error&$d$&\multicolumn{3}{c}{Conditional mean model }&\multicolumn{3}{c}{Logistic model}  \\
			&& Proposed & Logit & SVM & Proposed & Logit & SVM\\
			\hline
			\multirow{2}{*}{$\ell_1$}&64&0.244 (0.047)&0.795 (0.028)&0.908 (0.335)&0.750 (0.210) &1.031 (0.254)    &1.001 (0.323) \\
			&2500&2.026 (0.137)&5.076 (0.241)&8.382 (2.465)&8.116 (0.503)&7.794 (0.911)& 9.270 (4.895)\\
			\hline
			\multirow{2}{*}{$\ell_2$}&64&0.076 (0.012)&0.289 (0.017)&0.236 (0.058) &0.269 (0.061)  &0.255 (0.044)    & 0.335 (0.085)\\
			&2500&0.233 (0.017)&0.755 (0.030)&0.796 (0.071)& 0.869 (0.033) & 0.816 (0.030) &0.972 (0.075)  \\
			\hline
			\multirow{2}{*}{$\ell_\infty$}&64&0.026 (0.004)&0.115 (0.007)&0.123 (0.040)&0.170 (0.068) &0.132 (0.033)    &0.136 (0.040)  \\
			&2500&0.058 (0.006)&0.141 (0.002)&0.141 (0.001)&0.155 (0.021)&0.143 (0.004)& 0.144 (0.007) \\
						\hline
		\end{tabular}
	\end{center}	
\end{table}

Table \ref{table_error} shows the statistical error defined as $\norm{\hat{\btheta} - \btheta^*}$ in $\ell_1$, $\ell_2$ and $\ell_\infty$ norm, respectively, where $\hat{\btheta}$ is the estimator from different methods. The tuning parameters in these three methods are chosen via the corresponding 5-fold cross validation. Note that for logistic regression and SVM, we firstly fit the model using both $X$ and $\bZ$ as predictors, and then rescale the coefficients so that the coefficient corresponding to $X$ is 1. The results suggest that under the conditional mean model, the proposed method outperforms both logistic regression and SVM by a large factor. When the underlying true model is truly logistic model, the performance of all these methods are similar. This experiment validates the statistical performance of the proposed estimator.


\section{Real data application}\label{real_data}

In this section we apply the proposed method to a dataset on the ChAMP (Chondral Lesions And Meniscus Procedures) study \citep{bisson2017patient}.
This study is a double-blinded randomized controlled trial on patients undergoing arthroscopic partial meniscectomy (APM), a knee surgery for meniscal tears. This dataset contains information about the basic demographic information as well as preoperative and postoperative outcome measures, including Short Form-36 (SF-36) health survey, Western Ontario and McMaster Universities Osteoarthritis Index (WOMAC) and Knee Injury and Osteoarthritis Outcome Score (KOOS).
The primary measurement for the success of APM is the WOMAC pain score, where a higher score indicates a better outcome. In particular, the difference from WOMAC pain score before the treatment to that one year after the surgery, denoted by  $\text{WOMAC}_{\text{post}} - \text{WOMAC}_{\text{pre}}$, serves as a primary diagnostic measure in practice.
Meanwhile, clinicians tend to find some alternatives to evaluate the clinical significance through other venues. One method is based upon the so-called anchor question. In particular, one question in the SF-36 survey about how the patient feels about his/her pain would be asked and a binary variable $Y$ can be obtained from this patient reported outcome.
Let $Y_i = 1$ denote that the $i$th patient is healthy/satisfactory  and $-1$ otherwise. Let $X_i = \text{WOMAC}_{\text{post},i} - \text{WOMAC}_{\text{pre},i}$ be the pain score difference for this patient and $\bZ_i$ be additional demographic statistics and clinical biomarkers.

The primary goal in our application is to determine the minimal clinically important difference, defined as a linear combination of the variable $\bZ$, $\btheta^T\bZ$, such that the treatment of debridement of chondral lesions during the surgery can be claimed as clinically significant/successful by comparing the WOMAC pain score change $X$ with this individualized cut point. This determination includes not only the selection of active variables (with non-zero coefficients) but also the estimation of all of these non-zero coefficients.
This can be formulated as an estimation problem exactly of the form (\ref{MCID_linear}) or (\ref{MCID_linear2}) with weights $w(y) = 1/\PP(Y = y)$, namely, our goal is to estimate
\[
\btheta^{*} = \argmin_{\btheta}\left\{\PP_P(X\geq \btheta^T\bZ|Y =-1) + \PP_P(X<  \btheta^T\bZ|Y =1)\right\}.
\]

After removing redundant features and observations (patients) who ceased to participate during the follow-up period, the final dataset contains $n = 138$ observations and $d = 160$ measurements apart from $X$. We apply the proposed method using the Gaussian kernel and set the bandwidth $\delta = 1$. For comparison, we also apply $\ell_1$ penalized logistic regression and $\ell_1$ penalized SVM on the same dataset. Note that for logistic regression and SVM, we do not enforce $X$ to be active but instead treat it as one of the covariates. We also flip the sign for covairates $\bZ$ in logistic regression and SVM for better comparison. Since the fitted values of coefficients from these methods are not directly comparable, we focus on the regularization path for each method by starting from a large regularization parameter and decreasing it gradually.

\begin{table}[ht]
	\caption{First five active variables in the proposed method, logistic regression and SVM}\label{table_top5}
	\small
	\begin{center}
		\begin{tabular}{c c c c c c }
			\hline
			Method &First& Second&Third&Fourth&Fifth\\
			Proposed&WFunc\_6mo &KSymp\_3mo &exten\_inj\_pre &KPain\_3mo &SF36Soc\_6mo\\
			Logit&KSymp\_3mo &WFunc\_6mo &X &PatellaCenLes &SF36Soc\_6mo \\
			SVM &X\_yrComplete &FemurMedLes &NormXray &effusion\_inj\_3mo &KSymp\_3mo\\
			\hline
		\end{tabular}
	\end{center}	
\end{table}

\begin{figure}[ht]
	\centering
	{\subfigure[]{\includegraphics[width=70mm,height = 70mm]{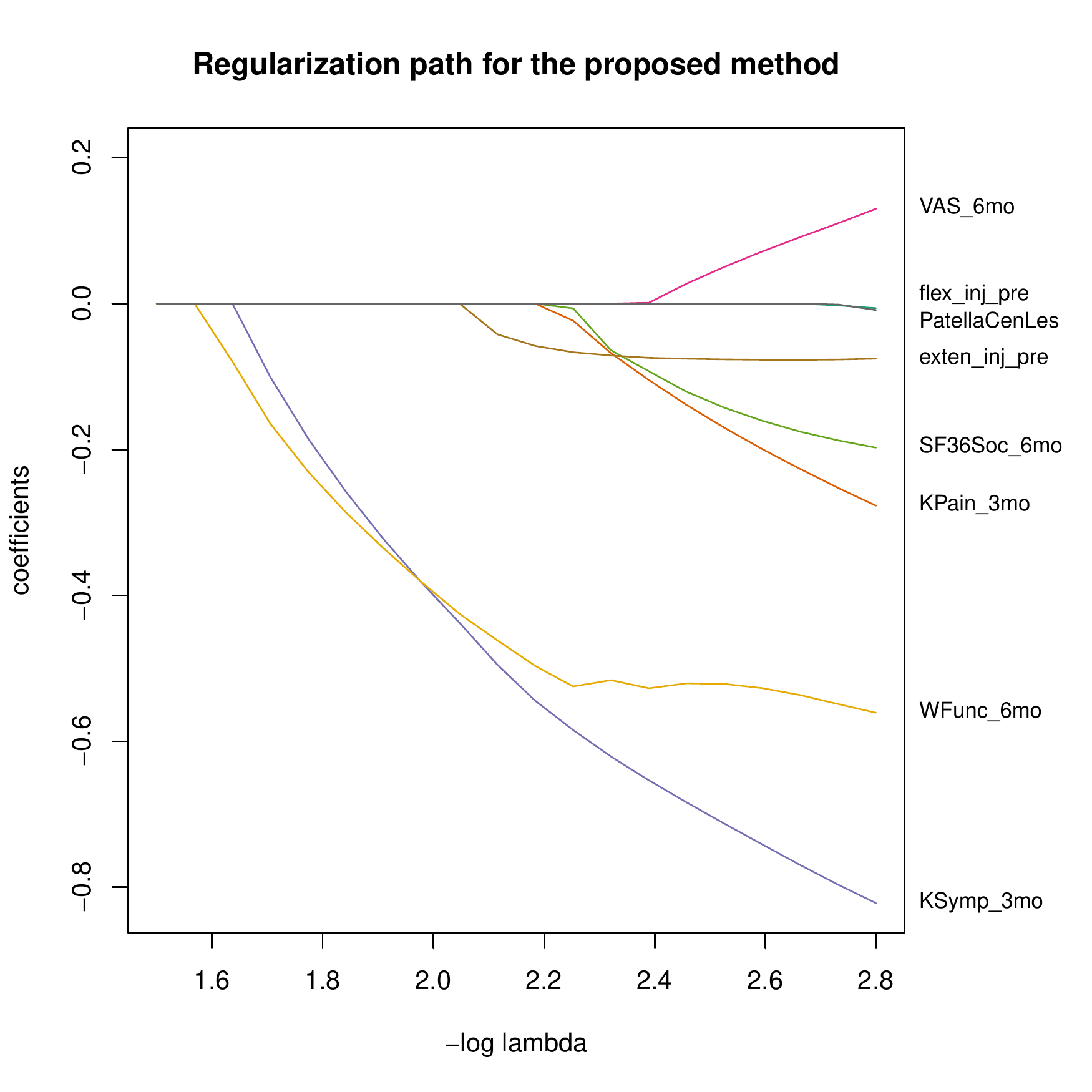}}
		\subfigure[]{\includegraphics[width=70mm,height = 70mm]{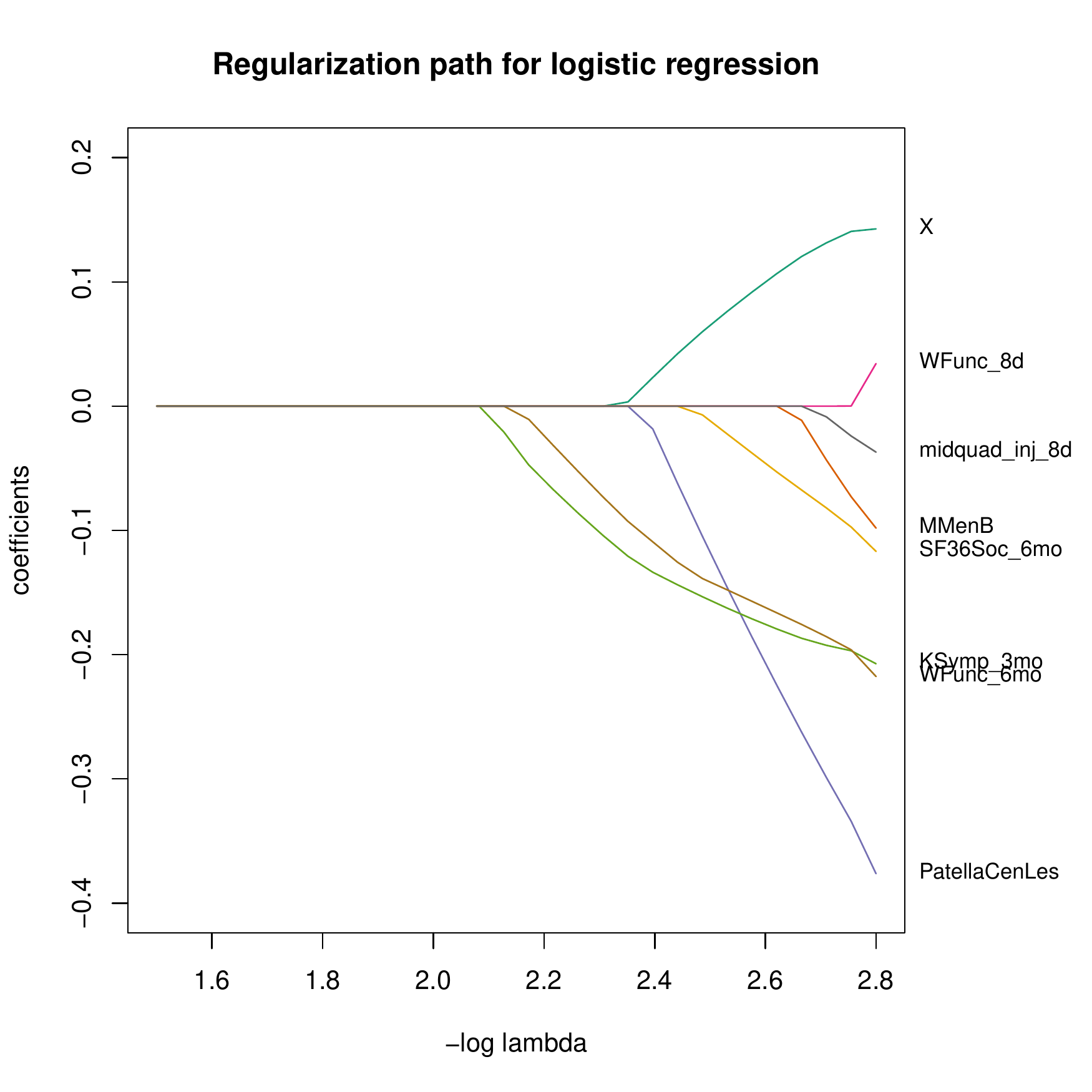}}}
	\caption{The regularization path for the proposed method and logistic regression on the ChAMP Trial dataset.
	}   \label{path_champ}
\end{figure}

Table \ref{table_top5} shows the first five variables that become active (have non-zero coefficient) along the regularization path. For all these methods, the variable \textit{KSymp\_3mo} (the KOOS score for other symptoms at 3-month) becomes active at early stages of the regularization path. In addition, the result shows that logistic regression and the proposed method yield very similar output, where the majority of the variables (\textit{Wfunc\_6mo}, \textit{KSymp\_3mo} and \textit{SF36Soc\_6mo}, with the interpretations the WOMAC score for physical functioning of the joints at 6-month, the KOOS score for other symptoms at 3-month and the SF-36 score for social role functioning at 6-month, respectively) are overlapped. In particular, the main measurement $X$, WOMAC pain score difference, treated as a covariate, is among the first five active variables in logistic regression. This validates the clinical practice of using $X$ as a main marker for patient recovery. 

To further compare the outputs from the proposed method and logistic regression, Figure \ref{path_champ} shows the regularization path for both methods. We find that in logistic regression variable $X$  has positive sign and most of the other active variables have negative sign. This is indeed expected because all additional covariates are encoded such that a higher value indicates a better health condition. Thus, for the patient with a higher intermediate health condition, the individualized threshold for being identified as recovered/healthy becomes lower, which is clinically more meaningful than using a common threshold for all patients. We defer the regularization path for SVM and further discussion to the supplementary material.


\section{Discussion}\label{discussion}
This paper studies the problem of estimating a sparse functional $\btheta^*(P)$ of the underlying probability measure $P = P(X,Y,\bZ)$ of the form (\ref{MCID_linear}) under the high-dimensional and model-free regime.
The absence of explicit modeling assumptions for the dependence of the response $Y$ on $X,\bZ$ distinguishes this paper with most existing works. To estimate $\btheta^*$, we propose a smoothed surrogate loss function that is asymptotically Fisher consistent. Based on the surrogate loss function, we estimate $\btheta^*$ by minimizing a penalized smoothed empirical risk function. Compared to existing works on high-dimensional estimation, the statistical performance of the proposed estimator depends additionally on the smoothness of the underlying distribution, due to the approximation bias. We prove that the estimator achieves the rate $\big(s(\log d)/n\big)^{\beta/(2\beta + 1)}$ for the $\ell_2$ estimation error and the rate is shown to be minimax optimal up to a logarithmic factor. From the computational perspective, we develop a path-following algorithm for the proposed estimator which achieves the optimal statistical performance with geometric rate of convergence in optimization.


Our work can be extended in several directions. First, this work can be extended to the case where the high-dimensional covariates $\bZ$ have more complex structures, for instance group sparsity. Second, in many practical situations such as clinical diagnostics, it is of great interest to quantify the uncertainty of $\hat\btheta$ and more importantly the estimated threshold $\bZ^T\hat\btheta$. How to perform statistical inference based on the proposed framework is currently under investigation. Third, when the response $Y$ has more than two groups, it might take extra effort for a more involved theoretical analysis as well as a new algorithm to compute the estimator.



\appendix

\section{Proofs}

\subsection{Proof of Proposition \ref{prop_var}}\label{proof_prop_var}
We firstly prove the following lemma:
\begin{lemma}\label{lemma_variance}
	Under Assumption \ref{assum_propensity} - \ref{assum_density}, if we choose a proper kernel $K$, then we have for all $j = 1\dotso,d$
	\[
	\EE\bigg[ \bigg(\frac{YZ_j}{\delta} K(\frac{Y(X - \btheta^T\bZ)}{\delta})\bigg)^2 \bigg] \leq \frac{C}{\delta},
	\]
	where $C =M_2p_{\max} \int K^2(u) du$ is a constant.
\end{lemma}

\begin{proof}\label{proof_lemma_variance}
	By definition, we have
	\beq\label{thom_vari_eq1}
	&\EE\bigg[ \bigg(\frac{YZ_j}{\delta} K(\frac{Y(X - \btheta^T\bZ)}{\delta})\bigg)^2 \bigg]\\ &=
	\pi \EE\bigg[ \bigg(\frac{Z_j}{\delta} K(\frac{X - \btheta^T\bZ}{\delta})\bigg)^2 \bigg| Y = 1\bigg] +
	(1 - \pi)\EE\bigg[ \bigg(\frac{Z_j}{\delta} K(\frac{X - \btheta^T\bZ}{\delta})\bigg)^2 \bigg| Y = -1\bigg].
	\eeq
	
	Here we bound the first term on the RHS and the second term follows similarly. With some algebra we obtain
	\beq
	\EE\bigg[ \bigg(\frac{Z_j}{\delta} K(\frac{X - \btheta^T\bZ}{\delta})\bigg)^2 \bigg| Y = 1\bigg] &=
	\int_Z \frac{z_j^2}{\delta^2}\underbrace{\int K^2(\frac{x - \btheta^T\bz}{\delta}) f(x|\bz,Y = 1) dx}_{(B)} f(\bz|Y = 1)d\bz.
	\eeq
	With some change of variable, $(B)$ can be expressed as
	\beq
	(B) &=\delta \int K^2(u) f(u\delta + \btheta^T\bz|\bz,Y = 1) du\\
	&\leq \delta p_{\max} \int K^2(u) du,
	\eeq
	where the inequality follows from Assumption \ref{assum_density}.
	By Assumption \ref{assum_bounded}, we obtain
	\beq
	\EE\bigg[ \bigg(\frac{Z_j}{\delta} K(\frac{X - \btheta^T\bZ}{\delta})\bigg)^2 \bigg| Y = 1\bigg] \leq M_2p_{\max} \int K^2(u) du\cdot \frac 1\delta.
	\eeq
	With the same derivation for the second term in (\ref{thom_vari_eq1}), we conclude that
	\beq
	\EE\bigg[ \bigg(\frac{YZ_j}{\delta} K(\frac{Y(X - \btheta^T\bZ)}{\delta})\bigg)^2 \bigg] \leq \frac{C}{\delta},
	\eeq
	where $C =M_2p_{\max} \int K^2(u) du.$
\end{proof}

Now we are ready to prove Proposition \ref{prop_var}.

\begin{proof}
    Denote $T = \nabla R_\delta^n(\btheta^*) - \nabla R_\delta(\btheta^*)$.  Here by definition
	\beq
	\norm{T}_\infty = \bignorm{\frac1n\sum_{i = 1}^{n} w(y_i)\frac{y_i\bz_i}{\delta}K(\frac{y_i(x_i - \btheta^T\bz_i)}{\delta})  -
		\EE\bigg[ w(Y)\frac{YZ}{\delta} K(\frac{Y(X - \btheta^T\bZ)}{\delta}) \bigg] }_\infty.
	\eeq
	Writing $\pi_{\max} = \max_y w(y)$, which is bounded by a constant by Assumption \ref{assum_propensity}, we have
	\beq
	    |T_{ij}| &= \bigg|w(y_i)\frac{y_i\bz_{ij}}{\delta}K(\frac{y_i(x_i - \btheta^T\bz_i)}{\delta})  -
	\EE\bigg[ w(Y)\frac{YZ_j}{\delta} K(\frac{Y(X - \btheta^T\bZ)}{\delta}) \bigg]\bigg|\\
	&\leq 2\pi_{\max}K_{\max}M_n/\delta,
	\eeq
	and
	\beq
	    \EE [T_{ij}^2] \leq \pi_{\max}^2 C/\delta,
	\eeq
	where $C$ is the constant defined in Lemma \ref{lemma_variance}.
	Applying Bernstein inequality gives us
	\beq
	    \PP(\norm{T}_{\infty} > t) &\leq \sum_{j = 1}^{d}\PP(|T_j| > t)\\
	    &\leq 2d \exp\bigg(
	    -\frac{\frac12t^2n}{\pi_{\max}^2 C/\delta + \frac{t}{3}2\pi_{\max}K_{\max}M_n/\delta  }
	    \bigg ).
	\eeq
	Note that $M_n \lesssim \sqrt{\frac{\log d}{n\delta}}$. Thus by taking $t = C_1\sqrt{\frac{n\delta}{\log d}}$ for some constant $C_1$ sufficiently large, we can ensure that
	\beq
	\PP(\norm{T}_{\infty} > t) \leq 2d^{-1}.
	\eeq
	This completes the proof.
\end{proof}

\subsection{Proof of Proposition \ref{prop_bias}}
\begin{proof}\label{proof_prop_bias}
	By definition, we have
	\beq\label{thom_bias_e1}
	\nabla R_\delta(\btheta^*) - \nabla R(\btheta^*) &= \bigg(\EE\bigg[\frac{\bZ}{\delta} K(\frac{X - \btheta^{*T}\bZ}{\delta})\bigg|Y = 1\bigg] -  \int_{\bZ} \bz f(\btheta^{*T}\bz,\bz|Y = 1)d\bz\bigg) \\
	&-  \bigg(\EE\bigg[\frac{\bZ}{\delta} K(\frac{ \btheta^{*T}\bZ - X}{\delta})\bigg |Y = -1\bigg] - \int_{\bZ} \bz f(\btheta^{*T}\bz,\bz|Y = -1)d\bz\bigg).
	\eeq
	
	Here we focus on the first term, and the property of the second term will be similar. Here the first term can be expressed as
	\beq\label{thom_bias_e2}
	\int_{\bZ} \bz \bigg(\underbrace{\int_X\frac 1\delta   K(\frac{x - \btheta^{*T}\bz}{\delta}) f(x|\bz,Y = 1)dx - f(\btheta^{*T}\bz|\bz,Y = 1)}_{(A)}\bigg) f(\bz|Y =1) d\bz.
	\eeq
	
	With some change of variable $(A)$ becomes
	\beq
	\int K(u)\bigg(f(u\delta + \btheta^{*T}\bz|\bz,Y =1) - f(\btheta^{*T}\bz|\bz,Y =1) \bigg)du.
	\eeq
	
	Since $f(x|y,\bz)$ is $l$ times differentiable, we have by Taylor expansion
	\beq
	f(u\delta + \btheta^{*T}\bz|\bz,Y =1) - f(\btheta^{*T}\bz|\bz,Y =1) &= \sum_{i = 1}^{l-1} \frac{f^{(i)}(\btheta^{*T}\bz|\bz,Y =1)}{i!}(u\delta)^i\\
	&\qquad + \frac{(u\delta)^l}{l!}f^{(l)}(\btheta^{*T}\bz+\tau u\delta|\bz,Y =1),
	\eeq
	for some $\tau \in [0,1]$ where $l = \floor{\beta}.$
	By choosing the kernel of order $l$, $(A)$ becomes
	\beq
	(A)	 
	&= \int K(u) \frac{(u\delta)^l}{l!}\bigg(f^{(l)}(\btheta^{*T}\bz+\tau u\delta|\bz,Y =1) - f^{(l)}(\btheta^{*T}\bz|\bz,Y =1)  \bigg)du,
	\eeq
	and therefore (\ref{thom_bias_e2}) becomes
	\beq
	\int K(u) \frac{(u\delta)^l}{l!} \int_Z \bz \bigg ((f^{(l)}(u\delta + \btheta^{*T}\bz|\bz,Y =1) - f^{(l)}(\btheta^{*T}\bz|\bz,Y =1) \bigg) f(\bz|Y =1) d\bz du,
	\eeq
	and hence for any $\bv$ with $\norm{\bv}_0 \leq s'$
	\beq
	&\bigg | \int K(u) \frac{(u\delta)^l}{l!} \int_Z \bv^T\bz \bigg ((f^{(l)}(u\delta + \btheta^{*T}\bz|\bz,Y =1) - f^{(l)}(\btheta^{*T}\bz|\bz,Y =1) \bigg) f(\bz|Y =1) d\bz du\bigg |\\
	\leq& \int |K(u)| \frac{|u\delta|^l}{l!}L \norm{\bv}_2|u\delta|^{\beta - l}du\\
	\leq& L\norm{\bv}_2 \int K(u) \frac{(u\delta)^\beta}{l!}du,
	\eeq
	where the first inequality follows from Assumption \ref{assum_density}.
	This result follows similarly for the second term in (\ref{thom_bias_e1}), and thus we know for any $\bv \in \RR^d$ with $\norm{\bv}_0 \leq s'$
	
	\beq
	\big | \bv^T(\nabla R_\delta(\btheta^*) - \nabla R(\btheta^*))\big |
	&\leq
	2L\norm{\bv}_2\delta^{\beta} \int |K(u)| \frac{u^\beta}{l!}du.
	\eeq
	Writing $C_2 =2L \int |K(u)| \frac{u^\beta}{l!}du$  completes the proof.
\end{proof}

\subsection{Auxiliary results for the proof of Theorem \ref{theorem_main2}}\label{sec_appen_algo}

\subsubsection{Proof of Lemma \ref{lemma_initialization}}\label{proof_lemma_initialization}
\begin{proof}
	Since $\norm{\btheta_{S^{*c}}}_0 \leq \tilde{s}$ 
	 by applying (\ref{RSC_condi}) in Assumption \ref{RSC}, we can show that
	\beq\label{lm1e1}
	(\btheta - \btheta^*)^T(\nabla R_\delta^n(\btheta) - \nabla R_\delta^n(\btheta^*)) \geq \rho_{-}\norm{\btheta - \btheta^*}^2_2.
	\eeq
	On the other hand, by the definition of $\omega_\lambda(\btheta)$, since $\btheta^* \in \Omega$, suppose $\bxi$ is the subgradient that attains the minimum in (\ref{equ_subopti}), we have 	
	\beq
	\frac{(\btheta - \btheta^*)^T(\nabla R_\delta^n(\btheta) + \lambda\bxi) }{\norm{\btheta - \btheta^*}_1}
	&\leq \max_{\btheta' \in \Omega} \bigg \{ \frac{(\btheta - \btheta')^T}{\norm{\btheta - \btheta'}_1}(\nabla R_\delta^n(\btheta) + \lambda \bxi) \bigg \}\\
	&=w_\lambda(\btheta).
	\eeq
	This implies that
	\beq\label{lm1e2}
	(\btheta - \btheta^*)^T(\nabla R_\delta^n(\btheta) + \lambda\bxi)
	&\leq \norm{\btheta - \btheta^*}_1\omega_\lambda(\btheta)\\
	&\leq \frac12\lambda \norm{\btheta - \btheta^*}_1\\
	&=  \frac12\lambda (\norm{(\btheta - \btheta^*)_{S^*}}_1 + \norm{(\btheta - \btheta^*)_{S^{*c}}}_1).
	\eeq
	
	Combining (\ref{lm1e1}) and (\ref{lm1e2}) gives
	\beq
	 \frac{1}{2}\lambda (\norm{(\btheta - \btheta^*)_{S^*}}_1 + \norm{(\btheta - \btheta^*)_{S^{*c}}}_1) &\geq
	(\btheta - \btheta^*)^T\nabla R_\delta^n(\btheta) + \lambda(\btheta - \btheta^*)^T\bxi\\
	&\geq (\btheta - \btheta^*)^T\nabla R_\delta^n(\btheta^*) +  \rho_{-}\norm{\btheta - \btheta^*}^2_2 + \lambda(\btheta - \btheta^*)^T\bxi.
	\eeq

	Using the fact that
	$$(\btheta - \btheta^*)^T\bxi \geq -\norm{(\btheta - \btheta^*)_{S^*}}_1 + \norm{(\btheta - \btheta^*)_{S^{*c}}}_1,$$
	we further obtain
	\beq\label{lm1e3}
	&\frac{3}{2}\lambda\norm{(\btheta - \btheta^*)_{S^*}}_1 - (\btheta - \btheta^*)^T\nabla R_\delta^n(\btheta^*) \geq \frac{1}{2}\lambda\norm{(\btheta - \btheta^*)_{S^{*c}}}_1 + \rho_{-}\norm{\btheta - \btheta^*}^2_2.
	\eeq
	Notice that Proposition \ref{prop_bias} and the sparsity of $\btheta,\btheta^*$ imply that
	\beq\label{lm1e6}
	\bigg| (\btheta - \btheta^*)^T\nabla R_\delta^n(\btheta^*)\bigg| &= \bigg | (\btheta - \btheta^*)^T\bigg(
	\underbrace{\nabla R_\delta^n(\btheta^*) - \nabla R_\delta(\btheta^*)}_{E_1} + \underbrace{\nabla R_\delta(\btheta^*) - \nabla R(\btheta^*)}_{E_2} \bigg |
	\bigg)\\
	&\leq \norm{\btheta - \btheta^*}_1\norm{E_1}_\infty+ C_2\delta^{\beta}\norm{\btheta - \btheta^*}_2,
	\eeq
	where $C_2$ is the constant defined in Proposition \ref{prop_bias}.
	Plug (\ref{lm1e6}) in (\ref{lm1e3}) gives
	\beq\label{lm1e7}
	\rho_{-}\norm{\btheta - \btheta^*}^2_2 &\leq C_2\delta^{\beta}\norm{\btheta - \btheta^*}_2 + (\frac{3}{2}\lambda + \norm{E_1}_\infty)\norm{(\btheta - \btheta^*)_{S^*}}_1\\ &- (\frac{1}{2}\lambda- \norm{E_1}_\infty)\norm{(\btheta - \btheta^*)_{S^{*c}}}_1.
	\eeq
	
	Now consider two cases:\\
	If $
	\rho_{-}\norm{\btheta - \btheta^*}^2_2 \leq 3C_2\delta^{\beta}\norm{\btheta - \btheta^*}_2
	$, then it holds trivially that
	\beq
	\norm{\btheta - \btheta^*}_2 \leq \frac{3}{\rho_{-}}C_2\delta^{\beta}.
	\eeq
	If $
	\rho_{-}\norm{\btheta - \btheta^*}^2_2 > 3C_2\delta^{\beta}\norm{\btheta - \btheta^*}_2
	$, then we obtain
	\beq
	\frac{2\rho_{-}}{3}\norm{\btheta - \btheta^*}^2_2 &\leq (\frac{3}{2}\lambda+ \norm{E_1}_\infty)\norm{(\btheta - \btheta^*)_{S^*}}_1\\ &- (\frac{1}{2}\lambda - \norm{E_1}_\infty)\norm{(\btheta - \btheta^*)_{S^{*c}}}_1.
	\eeq
	
	The condition of $\lambda_{tgt}$ ensures that $\frac{1}{2}\lambda - \norm{E_1}_\infty \geq 0$, and thus we obtain
	\beq
	\norm{\btheta - \btheta^*}_2 &\leq \frac{3\sqrt{s}(\frac32\lambda + \norm{E_1}_\infty)}{2\rho_{-}}\\
	&\leq \frac{3}{\rho_{-}}\sqrt{s}\lambda.
	\eeq
	
	Combine the above two cases, we conclude that
	\beq\label{eq_A29}
	\norm{\btheta - \btheta^*}_2 \leq \frac{3}{\rho_{-}}\bigg( C_2\delta^{\beta} \lor \sqrt{s}\lambda \bigg).
	\eeq
	
	For $\norm{\btheta - \btheta^*}_1$, define $\gamma = \frac{\frac{3}{2}\lambda + \norm{E_1}_\infty}{ \frac{1}{2}\lambda - \norm{E_1}_\infty}\leq \frac{13}{3}$ and we also consider the following two cases:\\
	If $\norm{(\btheta - \btheta^*)_{S^{*c}}}_1 \leq 2\gamma \norm{(\btheta - \btheta^*)_{S^*}}_1,$
	we have by (\ref{eq_A29})
	\beq
	\norm{\btheta - \btheta^*}_1 &\leq \sqrt{s}(1 + 2\gamma)\norm{\btheta - \btheta^*}_2\\
	&\leq \frac{29}{3}\sqrt{s}\norm{\btheta - \btheta^*}_2\\
	&\leq \frac{29}{\rho_{-}}\bigg( C_2\sqrt{s}\delta^{\beta} \lor s\lambda \bigg).
	\eeq
	If $\norm{(\btheta - \btheta^*)_{S^{*c}}}_1 > 2\gamma \norm{(\btheta - \btheta^*)_{S^*}}_1,$ it follows that
	\beq
	\norm{\btheta - \btheta^*}_1  &\leq (1 + \frac{1}{2\gamma})\norm{(\btheta - \btheta^*)_{S^{*c}}}_1.
	\eeq
	Notice that in this case, (\ref{lm1e7}) implies that
	
	\beq
	(\frac12\lambda - \norm{E_1}_\infty)\norm{(\btheta - \btheta^*)_{S^{*c}}}_1  &\leq  2\bigg (
	C_2\delta^{\beta}\norm{\btheta - \btheta^*}_2 - \rho_{-}\norm{\btheta - \btheta^*}^2_2
	\bigg).
	\eeq
	Therefore, by the condition of $\lambda_{tgt}$ and (\ref{eq_A29})
	\beq
	\norm{\btheta - \btheta^*}_1  &\leq \frac{(2 + \frac{1}{\gamma})}{\frac12\lambda - \norm{E_1}_\infty} C_2\delta^{\beta}\norm{\btheta - \btheta^*}_2\\
	&\leq \frac{56C_2\delta^{\beta}}{3\rho_{-}}\bigg( \frac{C_2\delta^{\beta}}{\lambda} \lor \sqrt{s} \bigg).
	\eeq
	Conclude the above two cases gives
	\beq
	\norm{\btheta - \btheta^*}_1 \leq \frac{1}{\rho_{-}}\bigg(
	29\sqrt{s}C_2\delta^{\beta}  \lor  29s\lambda  \lor \frac{56C_2^2\delta^{2\beta}}{3\lambda}\bigg ).
	\eeq
	Finally, we obtain
	\beq
	f_\lambda(\btheta) - f_\lambda(\btheta^*) &\leq (\btheta - \btheta^*)^T(\nabla R_\delta^n(\btheta) + \lambda\bxi)\\
	&\leq  \frac{1}{2}\lambda\norm{\btheta - \btheta^*}_1\\\
	&\leq \frac{1}{2\rho_{-}}\bigg(
	29\sqrt{s}C_2\delta^{\beta}\lambda  \lor  29s\lambda^2  \lor \frac{56C_2^2\delta^{2\beta}}{3}\bigg ),
	\eeq
	where the first inequality follows from the restricted strong convexity and the second inequality follows from (\ref{lm1e2}). Writing
	\beq
	&\bar{C}_1 = 3(C_2 \lor 1),\\
	&\bar{C}_2 = 29C_2 \lor 29 \lor \frac{56C_2^2}{3}
	\eeq completes the proof.
\end{proof}

\subsubsection{Proof of Lemma \ref{lemma_norm}}\label{proof_lemma_norm}

\begin{proof}
	From the assumption we have
	\beq\label{lm2e1}
	R_\delta^n(\btheta) + \lambda\norm{\btheta}_1 \leq R_\delta^n(\btheta^*) + \lambda\norm{\btheta^*}_1 +  \frac{\bar{C}_2}{2\rho_{-}}\bigg( \delta^{2\beta} \lor \sqrt{s}\delta^{\beta}\lambda \lor   s\lambda^2  \bigg).
	\eeq
	Since $\norm{\btheta_{S^{*c}}}_0 \leq \tilde{s}$ 
	by applying (\ref{RSC_condi}) in Assumption \ref{RSC}, we can show that
	\beq\label{lm2e2}
	R_\delta^n(\btheta) &\geq R_\delta^n(\btheta^*) + \nabla R_\delta^n(\btheta^*)^T(\btheta - \btheta^*) + \frac{1}{2} \rho_{-}\norm{\btheta - \btheta^*}^2_2\\
	\eeq
	Subtract (\ref{lm2e2}) from (\ref{lm2e1}) gives
	\beq
	&\frac 12\rho_{-}\norm{\btheta - \btheta^*}^2_2\\\leq&
	\lambda\norm{\btheta^*}_1 - \lambda\norm{\btheta}_1 -  \nabla R_\delta^n(\btheta^*)^T(\btheta - \btheta^*) + \frac{\bar{C}_2}{2\rho_{-}}\bigg( \delta^{2\beta} \lor \sqrt{s}\delta^{\beta}\lambda \lor   s\lambda^2  \bigg)\\
	\leq& \lambda\norm{(\btheta - \btheta^*)_{S^*}}_1 - \lambda\norm{(\btheta - \btheta^*)_{S^{*c}}}_1 + \norm{\btheta - \btheta^*}_1\norm{E_1}_\infty\\
	&~~~~~~~~~~~~~~~~~~ + C_2\delta^{\beta}\norm{\btheta - \btheta^*}_2+ \frac{\bar{C}_2}{2\rho_{-}}\bigg( \delta^{2\beta} \lor \sqrt{s}\delta^{\beta}\lambda \lor   s\lambda^2  \bigg),
	\eeq
	where the second inequality follows from triangle inequality and (\ref{lm1e6}). Separating $\norm{\btheta - \btheta^*}_1$ into $\norm{(\btheta - \btheta^*)_{S^*}}_1 + \norm{(\btheta - \btheta^*)_{S^{*c}}}_1$ gives
	\beq\label{lm2e3}
	\frac 12\rho_{-}\norm{\btheta - \btheta^*}^2_2 &\leq  (\lambda + \norm{E_1}_\infty)\norm{(\btheta - \btheta^*)_{S^*}}_1- (\lambda- \norm{E_1}_\infty)\norm{(\btheta - \btheta^*)_{S^{*c}}}_1 \\
	&+ C_2\delta^{\beta}\norm{\btheta - \btheta^*}_2+ \frac{\bar{C}_2}{2\rho_{-}}\bigg( \delta^{2\beta} \lor \sqrt{s}\delta^{\beta}\lambda \lor   s\lambda^2  \bigg).
	\eeq
	Now we consider two cases:
	\begin{itemize}
		\item If $\norm{\btheta - \btheta^*}_1 \leq\frac{\bar{C}_2}{\rho_{-}}\bigg( \frac{\delta^{2\beta}}{\lambda} \lor \sqrt{s}\delta^{\beta} \lor   s\lambda  \bigg)$,   we have by (\ref{lm2e3}) and the condition of $\lambda_{tgt}$
		\beq
		\frac 12\rho_{-}\norm{\btheta - \btheta^*}^2_2 &\leq
		\frac98\lambda  \norm{\btheta - \btheta^*}_1 + C_2\delta^{\beta}\norm{\btheta - \btheta^*}_2 + \frac{\bar{C}_2}{2\rho_{-}}\bigg( \delta^{2\beta} \lor \sqrt{s}\delta^{\beta}\lambda \lor   s\lambda^2  \bigg)\\
		&\leq C_2\delta^{\beta}\norm{\btheta - \btheta^*}_2 + \frac{13\bar{C}_2}{8\rho_{-}}\bigg( \delta^{2\beta} \lor \sqrt{s}\delta^{\beta}\lambda \lor   s\lambda^2  \bigg).
		\eeq
		Now we further consider two cases:
		\begin{enumerate}
			\item If $\frac 12\rho_{-}\norm{\btheta - \btheta^*}^2_2 \leq 3C_2\delta^{\beta}\norm{\btheta - \btheta^*}_2$, then it holds trivially that
			\[
			\norm{\btheta - \btheta^*}_2 \leq \frac{6C_2}{\rho_{-}}\delta^{\beta} \leq \frac{2\bar{C_2}}{\rho_{-}}\delta^{\beta}.
			\]
			\item If $\frac 12\rho_{-}\norm{\btheta - \btheta^*}^2_2 > 3C_2\delta^{\beta}\norm{\btheta - \btheta^*}_2$, plug in this into (\ref{lm2e3}) implies
			\[
			\frac 13\rho_{-}\norm{\btheta - \btheta^*}^2_2 \leq  \frac{13\bar{C}_2}{8\rho_{-}}\bigg( \delta^{2\beta} \lor \sqrt{s}\delta^{\beta}\lambda \lor   s\lambda^2  \bigg),
			\]
			which further implies that
			\[
			\norm{\btheta - \btheta^*}_2 \leq \frac{\sqrt{39\bar{C}_2/8}}{\rho_-} \sqrt{\delta^{2\beta} \lor \sqrt{s}\delta^{\beta}\lambda \lor   s\lambda^2  }.
			\]
			
		\end{enumerate}
	\item If $\norm{\btheta - \btheta^*}_1 > \frac{\bar{C}_2}{\rho_{-}}\bigg( \frac{\delta^{2\beta}}{\lambda} \lor \sqrt{s}\delta^{\beta} \lor   s\lambda  \bigg)$, then we obtain
	\beq
	\frac12\lambda \norm{\btheta - \btheta^*}_1 \geq \frac{\bar{C}_2}{2\rho_{-}}\bigg( \delta^{2\beta} \lor \sqrt{s}\delta^{\beta}\lambda \lor   s\lambda^2  \bigg).
	\eeq
	Plug in this into (\ref{lm2e3}) gives
	 	\beq
	 \frac 12\rho_{-}\norm{\btheta - \btheta^*}^2_2 &\leq  (\frac32\lambda + \norm{E_1}_\infty)\norm{(\btheta - \btheta^*)_{S^*}}_1- (\frac12\lambda- \norm{E_1}_\infty)\norm{(\btheta - \btheta^*)_{S^{*c}}}_1\\
	 &+ C_2\delta^{\beta}\norm{\btheta - \btheta^*}_2.
	 \eeq
	 Following a similar derivation of Lemma \ref{lemma_initialization}, we can show that
	 \beq
	 \norm{\btheta - \btheta^*}_2 \leq \frac{2\bar{C}_1}{\rho_{-}}\bigg( \delta^{\beta} \lor \sqrt{s}\lambda \bigg) ,
	 \eeq
	 and
	 \beq
	 \norm{\btheta - \btheta^*}_1 \leq \frac{2\bar{C}_2}{\rho_{-}}\bigg( \frac{\delta^{2\beta}}{\lambda} \lor \sqrt{s}\delta^{\beta} \lor   s\lambda  \bigg),
	 \eeq
	 where $\bar{C}_1$ and $\bar{C}_2$ are constants defined in Lemma \ref{lemma_initialization}.
	\end{itemize}
	Combining the above cases, we conclude that
	\beq
	&\norm{\btheta - \btheta^*}_1 \leq \frac{2\bar{C}_2}{\rho_{-}}\bigg( \frac{\delta^{2\beta}}{\lambda} \lor \sqrt{s}\delta^{\beta} \lor   s\lambda  \bigg),\text{and}\\
	&\norm{\btheta - \btheta^*}_2 \leq \frac{2\bar{C}_1 \lor \sqrt{39\bar{C}_2/8}}{\rho_{-}}\bigg( \delta^{\beta} \lor s^{1/4}\sqrt{\delta^{\beta}\lambda}   \lor \sqrt{s}\lambda \bigg ).
	\eeq
	Writing $\bar{C}'_1 = 2\bar{C}_1 \lor \sqrt{39\bar{C}_2/8}$ completes the proof.
\end{proof}


\subsubsection{Proof of Lemma \ref{lemma_sparse}}\label{proof_lemma_sparse}
\begin{proof}
	Notice that when the constraint set $\Omega = \RR^d$, the updating rule becomes to the soft-thresholding operator
	\beq
	(\cS_{\lambda \eta}(\btheta, \RR^d ))_j =
	\begin{cases}
		0 &\text{if } |\bar{\theta}_j| \leq \lambda \eta,\\
		\sign{\bar{\theta}_j}(|\bar{\theta}_j| - \lambda \eta) &\text{if } |\bar{\theta}_j| > \lambda \eta,\\
	\end{cases}
	\eeq
	where $\bar{\btheta} = \btheta - \eta\nabla R_\delta^n(\btheta)$.
	Now for $\Omega = \{\bv:\norm{\bv}_2 \leq R \}$ with some $R > 0$, following a standard Lagrangian argument it can be shown that $\cS_{\lambda \eta}(\btheta, \Omega )$ can be obtained by firstly calculating the unconstrained version above, and then project it on to $\Omega$. This can be achieved by scaling, which does not affect the sparsity pattern.
	
	Therefore we have
	\beq
	\norm{\cS_{\lambda \eta}(\btheta,\Omega)_{S^{*c}}}_0 = \bigg|\bigg\{
	j\in S^{*c}: |\bar{\theta}_j| > \lambda\eta
	\bigg\} \bigg|.
	\eeq
	Notice that by definition $$\bar{\btheta} = {\btheta} - \eta(\nabla R_\delta^n(\btheta^*) - \nabla R_\delta(\btheta^*)) - \eta(\nabla R_\delta(\btheta^*) - \nabla R(\btheta^*)) + \eta(\nabla R_\delta^n(\btheta^*) - \nabla R_\delta^n(\btheta))$$
	Therefore it suffices to show that $|S_1| + |S_2| + |S_3| + |S_4| \leq \tilde{s}$, where
	\beq
	&S_1 = \bigg| \bigg \{
	j \in S^{*c} : |\theta_j| \geq \frac 14\lambda\eta
	\bigg \} \bigg|\\
	&S_2 = \bigg| \bigg \{
	j \in S^{*c} : | \eta(\nabla R_\delta^n(\btheta^*) - \nabla R_\delta^n(\btheta))_j| \geq \frac 14\lambda\eta
	\bigg \} \bigg|\\
	&S_3 = \bigg| \bigg \{
	j \in S^{*c} : |\eta(\nabla R_\delta^n(\btheta^*) - \nabla R_\delta(\btheta^*))_j| \geq \frac 14\lambda\eta
	\bigg \} \bigg|\\
	&S_4 = \bigg| \bigg \{
	j \in S^{*c} : |\eta(\nabla R_\delta(\btheta^*) - \nabla R(\btheta^*))_j| \geq \frac 14\lambda\eta
	\bigg \} \bigg|
	\eeq
	
	For $S_1$, we have by definition
	\beq\label{s1}
	|S_1| &= \sum_{j\in S^{*c}}\mathds{1}(|(\btheta - \btheta^*)_j| \geq \frac14\eta\lambda )\\
	&\leq \frac{4}{\eta\lambda}	\sum_{j\in S^{*c}}|(\btheta - \btheta^*)_j|\\
	&\leq \frac{4}{\eta\lambda} \norm{\btheta - \btheta^*}_1.\\
	&\leq \frac{8\bar{C}_2}{\eta\rho_{-}}\bigg( \frac{\delta^{2\beta}}{\lambda^2} \lor \frac{\sqrt{s}\delta^{\beta} }{\lambda}\lor   s  \bigg),
	\eeq
	where the last inequality follows from Lemma \ref{lemma_norm}.
	
	For $S_2$, 
	consider a vector $\bv \in \RR^d$ such that $v_j = \sign{(\nabla R_\delta^n(\btheta^*) - \nabla R_\delta^n(\btheta))_j}$ for $j \in S_2$ and $v_j = 0$ otherwise. Then we have
	\beq
	\sqrt{|S_2|}\norm{\nabla R_\delta^n(\btheta^*) - \nabla R_\delta^n(\btheta)}_2 &= \norm{\bv}_2\norm{\nabla R_\delta^n(\btheta^*) - \nabla R_\delta^n(\btheta)}_2\\ &\geq v^T(\nabla R_\delta^n(\btheta^*) - \nabla R_\delta^n(\btheta))\\
	&= \sum_{j\in S_2}\bigg|(\nabla R_\delta^n(\btheta^*) - \nabla R_\delta^n(\btheta))_j\bigg|\\
	&\geq  \frac14 \lambda |S_2|.
	\eeq
	By the assumption that $\norm{\btheta}_0 \leq \tilde{s}$, we know $\norm{\btheta - \btheta^*}_0\leq s + \tilde{s}.$ Therefore (\ref{RS_condi}) in Assumption \ref{RSC} implies that
	\[
	\norm{\nabla R_\delta^n(\btheta^*) - \nabla R_\delta^n(\btheta)}_2 \leq \rho_{+}\norm{\btheta - \btheta^*}_2,
	\]
	which further implies that
	\beq\label{s2}
	|S_2| &\leq \frac{16}{\lambda^2}\rho_{+}^2\norm{\btheta - \btheta^*}_2^2\\
	&\leq \frac{16\bar{C}'^2_1\rho_{+}^2}{\rho^2_{-}}\bigg( \frac{\delta^{2\beta}}{\lambda^2} \lor \frac{\sqrt{s}\delta^{\beta} }{\lambda}\lor   s  \bigg).
	\eeq
	For $S_4$, using a similar derivation to that for $S_2$ above and Proposition \ref{prop_bias} implies that
	\beq
	\sqrt{|S_4|}C_2\delta^{\beta} \geq \frac 14 \lambda |S_4|,
	\eeq
	which further implies that
	\beq\label{s4}
	|S_4| \leq \frac{16C_2^2\delta^{2\beta}}{\lambda^2}.
	\eeq
	Finally for $S_3$,
	by the condition of $\lambda_{tgt}$, we have $\lambda > 4 \norm{\nabla R_\delta^n(\btheta^*) - \nabla R_\delta(\btheta^*)}_\infty$, which implies that
	\beq\label{s3}
	|S_3| = 0.
	\eeq
		
	Recall that we have $\lambda \geq \lambda_{tgt} = C\sqrt{\frac{\log d}{n\delta}}$,
	by choosing $$\delta = c\bigg(\frac{s \log d}{n}\bigg)^{1/(2\beta+1)},$$
	where $c^{\beta + 0.5} \leq  C$, we obtain
	\beq
	\frac{\delta^{2\beta}}{\lambda^2} \leq \frac{\sqrt{s}\delta^{\beta} }{\lambda} \leq s.
	\eeq
	Combining (\ref{s1}), (\ref{s2}), (\ref{s4}) and (\ref{s3}), we obtain
	\beq
	\norm{\cS_{\lambda \eta}(\btheta, \Omega)_{S^{*c}}}_0  \leq 8\bigg(\frac{\bar{C}_2}{\eta\rho_{-}} + \frac{2\bar{C}'^2_1\rho_{+}^2}{\rho^2_{-}} + 2C_2^2\bigg)\cdot s .
	\eeq
	This completes the proof.
\end{proof}

\begin{lemma}[ \cite{nesterov2013gradient}, Theorem 1]\label{lemma_monotone}
	Under the same conditions of Lemma \ref{lemma_sparse}, we have
	\[
	f_\lambda( \cS_{\lambda \eta}(\btheta,\Omega)) \leq f_\lambda(\btheta) - \frac{1}{2\eta}\norm{\cS_{\lambda \eta}(\btheta,\Omega) - \btheta}_2^2.
	\]	
\end{lemma}

\begin{lemma}[\cite{nesterov2013gradient}, Corollary 1]\label{lemma_nes_subop}
	Under the same conditions of Lemma \ref{lemma_sparse}, we have
	\[
	\omega_{\lambda_t}(\cS_{\lambda \eta}(\btheta,\Omega)) \leq (\frac 1\eta + \rho_{+})\norm{\cS_{\lambda \eta}(\btheta,\Omega) - \btheta}_2.
	\]
\end{lemma}

\begin{lemma}\label{good_ini}
	If $\omega_{\lambda_{t-1}}(\tilde{\btheta}_{t-1}) \leq \frac 14 \lambda_{t-1}$, then we have
	$$\omega_{\lambda_t}(\tilde{\btheta}_{t-1}) \leq \frac12 \lambda_{t}.$$
\end{lemma}
\begin{proof}
	Suppose $\bxi$ attains the minimum in $\omega_{\lambda_{t-1}}(\tilde{\btheta}_{t-1})$, we have
	\beq
	\omega_{\lambda_t}(\tilde{\btheta}_{t-1})  & = \min_{\bxi' \in \partial \norm{\tilde{\btheta}_{t-1}}_1}\max_{\btheta \in \Omega} \bigg \{
	\frac{(\tilde{\btheta}_{t-1} - \btheta)^T}{\norm{\tilde{\btheta}_{t-1} - \btheta}_1}(\nabla R_{\delta}^n(\tilde{\btheta}_{t-1}) + \lambda_t\bxi'   )
	\bigg \}\\
	&\leq \max_{\btheta \in \Omega} \bigg \{
	\frac{(\tilde{\btheta}_{t-1} - \btheta)^T}{\norm{\tilde{\btheta}_{t-1} - \btheta}_1}(\nabla R_{\delta}^n(\tilde{\btheta}_{t-1}) + \lambda_t\bxi   )\bigg \}\\
	&\leq \max_{\btheta \in \Omega} \bigg \{
	\frac{(\tilde{\btheta}_{t-1} - \btheta)^T}{\norm{\tilde{\btheta}_{t-1} - \btheta}_1}(\nabla R_{\delta}^n(\tilde{\btheta}_{t-1}) + \lambda_{t-1}\bxi   )\bigg \} + \max_{\btheta \in \Omega} \bigg \{
	\frac{(\tilde{\btheta}_{t-1} - \btheta)^T}{\norm{\tilde{\btheta}_{t-1} - \btheta}_1}(\lambda_{t-1} - \lambda_t  )\bxi\bigg \}\\
	&\leq \frac14 \lambda_{t-1} + \norm{(\lambda_{t-1} - \lambda_t  )\bxi}_\infty \\
	&\leq \frac14 \lambda_{t-1} + (1-\phi)\lambda_{t-1}\\
	&<  \frac12\lambda_{t},
	\eeq
	where the third inequality follows the duality between $\ell_1$ and $\ell_\infty$ norm, and the last inequality follows from $\lambda_{t} = \phi\lambda_{t-1}$ by definition and $ \phi = 0.9$.
\end{proof}

\begin{lemma}\label{smaller_lambda}
	Suppose Assumption \ref{RSC} holds. 
	If $\lambda \geq \lambda_{tgt}$, $\norm{\nabla R_\delta^n(\btheta^*) - \nabla R_\delta(\btheta^*)}_\infty \leq \lambda_{tgt}/8$, $\omega_{\lambda}(\btheta) \leq \frac12 \lambda$, $\norm{\btheta_{S^{*c}}}_0 \leq \tilde{s}$ and $\widehat{\btheta}_{\lambda} \in \Omega$ is a minimizer of $f_{\lambda}$ satisfying  $\norm{(\widehat{\btheta}_{\lambda})_{S^{*c}}}_0 \leq \tilde{s},$
	then  we have
	$$
	f_{\lambda}(\btheta) - f_{\lambda}(\widehat{\btheta}_{\lambda}) \leq \frac{\bar{C}_2}{\rho_{-}}\bigg( \delta^{2\beta} \lor \sqrt{s}\delta^{\beta}\lambda \lor   s\lambda^2  \bigg) .
	$$
\end{lemma}

\begin{proof}
	Suppose $\bxi \in \partial \norm{\btheta}_1$ attains the minimum in $\omega_{\lambda}(\btheta)$. By the convexity of $\ell_1$ penalty and the restricted strong convexity by Assumption \ref{RSC} we obtain
	\beq\label{lm6e1}
	f_{\lambda}(\btheta) - f_{\lambda}(\widehat{\btheta}_{\lambda}) &\leq
	(\nabla R_\delta^n(\btheta) + \lambda\bxi)^T(\btheta - \widehat{\btheta}_{\lambda}) \\
	&\leq \frac12\lambda \norm{\btheta - \widehat{\btheta}_{\lambda}}_1,
	\eeq
	where the second inequality follows a similar derivation to Lemma \ref{lemma_initialization}.
	
	Now it suffices to bound $\norm{\btheta - \widehat{\btheta}_{\lambda}}_1$.
	Since $\norm{\btheta_{S^{*c}}}_0 \leq \tilde{s}$, $\norm{(\widehat{\btheta}_{\lambda})_{S^{*c}}}_0 \leq \tilde{s}$ and $\omega_{\lambda}(\btheta) \leq \frac12 \lambda$, Lemma \ref{lemma_initialization} implies that
	\beq
	f_{\lambda}(\btheta) - f_{\lambda}(\widehat{\btheta}_{\lambda}) &\leq \frac12\lambda \norm{\btheta - \widehat{\btheta}_{\lambda}}_1\\
	&\leq \frac12\lambda(\norm{\btheta - \btheta^* }_1  + \norm{\widehat{\btheta}_{\lambda} -\btheta^*}_1 )\\
	&\leq \frac{\bar{C}_2}{\rho_{-}}\bigg( \delta^{2\beta} \lor \sqrt{s}\delta^{\beta}\lambda \lor   s\lambda^2  \bigg).
	\eeq
	This completes the proof.
\end{proof}

\subsubsection{Proof of Proposition \ref{prop_key}}\label{proof_prop_key}

\begin{proof}
	The proof structure is parallel with \cite{wang2014optimal}.
	By Lemma \ref{lemma_initialization}, the initialization implies that
	\beq
	f_{\lambda_t}(\btheta_t^0) - f_{\lambda_t}(\btheta^*) \leq \frac{\bar{C}_2}{2\rho_{-}}\bigg( \delta^{2\beta} \lor \sqrt{s}\delta^{\beta}\lambda_t \lor   s\lambda_t^2  \bigg).
	\eeq
	We firstly prove the sparsity along the path $\{\btheta_t^k\}_{k = 0}^\infty$ by induction.
	Assume at iteration $k-1$ we have
	$$\norm{(\btheta^{k-1}_t)_{S^{*c}}}_0 \leq \tilde{s}\; \text{and}\; f_{\lambda_t}(\btheta_t^{k-1}) - f_{\lambda_t}(\btheta^*) \leq  \frac{\bar{C}_2}{2\rho_{-}}\bigg( \delta^{2\beta} \lor \sqrt{s}\delta^{\beta}\lambda_t \lor   s\lambda_t^2  \bigg).$$
	
	According to Lemma \ref{lemma_sparse}, we know
	
	$$\norm{(\btheta^{k}_t)_{S^{*c}}}_0 \leq \tilde{s}.$$
	On the other hand, Lemma \ref{lemma_monotone} implies	\[
	f_{\lambda_t}( \btheta_t^{k}) \leq f_{\lambda_t}(\btheta_t^{k-1}) - \frac{1}{2\eta}\norm{\btheta_t^{k}-\btheta_t^{k-1}}_2^2,
	\]
	which further implies that
	\beq\label{thom1e1}
	f_{\lambda_t}(\btheta_t^{k}) - f_{\lambda_t}(\btheta^*) \leq \frac{\bar{C}_2}{2\rho_{-}}\bigg( \delta^{2\beta} \lor \sqrt{s}\delta^{\beta}\lambda_t \lor   s\lambda_t^2  \bigg).
	\eeq
	
	Combining above, at iteration $k$, we have
	
	$$\norm{(\btheta^{k}_t)_{S^{*c}}}_0 \leq \tilde{s}\; \text{and}\; f_{\lambda_t}(\btheta_t^{k}) - f_{\lambda_t}(\btheta^*) \leq \frac{\bar{C}_2}{2\rho_{-}}\bigg( \delta^{2\beta} \lor \sqrt{s}\delta^{\beta}\lambda_t \lor   s\lambda_t^2  \bigg).$$
	
	Thus the induction holds and we conclude that $\norm{(\btheta^{k}_t)_{S^{*c}}}_0 \leq \tilde{s}\;\forall\;k = 0,\dotso.$
	
	Now we start to prove that the sequence converges to a unique local solution. Since $\norm{(\btheta^0_t)_{S^{*c}}}_0 \leq \tilde{s}$, by the restricted strong convexity in Assumption \ref{RSC}, we know the level set
	\[
	\{\btheta : f_{\lambda_t}(\btheta) \leq  f_{\lambda_t}(\btheta_t^0), \norm{\btheta}_0, \norm{\btheta_t^0}_0\leq s + \tilde{s} \}
	\]
	is bounded below. This further implies that the sequence $\{f_{\lambda_t}(\btheta_t^k)\}_{k = 0}^\infty$ is also bounded below.
	
	Since $\norm{(\btheta^{k}_t)_{S^{*c}}}_0 \leq \tilde{s}\;\forall\;k = 0,\dotso,$ by Lemma \ref{lemma_monotone}, we obtain that the function value along the path is decreasing monotonically, i.e.,
	\[
	f_{\lambda_t}( \btheta_t^{0}) \geq  f_{\lambda_t}( \btheta_t^{1}) \geq \dotso .
	\]
	This implies that
	\[
	\lim\limits_{k\rightarrow\infty}\{
	f_{\lambda_t}( \btheta_t^{k}) - f_{\lambda_t}( \btheta_t^{k+1})
	\} = 0.
	\]
	In addition, by Lemma \ref{lemma_monotone}, we know
	\[
	\lim\limits_{k\rightarrow\infty}\norm{\btheta_t^{k} - \btheta_t^{k+1}}_2^2 \leq 2\eta \lim\limits_{k\rightarrow\infty}\{
	f_{\lambda_t}( \btheta_t^{k}) - f_{\lambda_t}( \btheta_t^{k+1})
	\} = 0,
	\]
	and by Lemma \ref{lemma_nes_subop}, we have
	\[
	\lim\limits_{k\rightarrow\infty}\omega_{\lambda_t}(\btheta_t^k) \leq (\frac 1\eta + \rho_{+})\lim\limits_{k\rightarrow\infty}\norm{\btheta_t^{k} - \btheta_t^{k-1}}_2 = 0.
	\]
	Consequently, the limit point of the sequence $\{\btheta_t^k\}_{k = 0}^\infty$ generated by Algorithm \ref{pf_algo} will satisfy  $\lim\limits_{k\rightarrow\infty}\omega_{\lambda_t}(\btheta_t^k) = 0$, which implies that the limit point is a local solution $\widehat{\btheta}_t$ satisfying the first order optimality condition with sparsity $\norm{(\widehat{\btheta}_t)_{S^{*c}}}_0 \leq \tilde{s}.$
	
	To show the uniqueness of the local solution, suppose there exists another limit point $\widehat{\btheta}'_t$ of the sequence $\{\btheta_t^k\}_{k = 0}^\infty$ satisfying the first order optimality condition $\omega_{\lambda_t}(\widehat{\btheta}'_t) = 0$. Similarly, we can show that $\norm{(\widehat{\btheta}'_t)_{S^{*c}}}_0 \leq \tilde{s}.$ 
	By the restricted strong convexity in Assumption \ref{RSC}, we obtain
	\beq \label{thom1e2}
	R_\delta^n(\widehat{\btheta}_t) - R_\delta^n(\widehat{\btheta}'_t) \geq
	\nabla R_\delta^n(\widehat{\btheta}'_t) (\widehat{\btheta}_t - \widehat{\btheta}'_t) + \frac {\rho_{-}} {2} \norm{\widehat{\btheta}_t - \widehat{\btheta}'_t}_2^2.
	\eeq
	
	Now consider the difference of the function value $f_{\lambda_t}( \widehat{\btheta}_t) - f_{\lambda_t}(\widehat{\btheta}'_t) $, by choosing $\bxi' \in \partial \norm{\widehat{\btheta}'_t}_1$ that attains the minimum in $\omega_{\lambda_t}(\widehat{\btheta}'_t)$
	we can see that
	\beq \label{thom1e3}
	f_{\lambda_t}( \widehat{\btheta}_t) - f_{\lambda_t}(\widehat{\btheta}'_t) &= R_\delta^n(\widehat{\btheta}_t) + \lambda_t \norm{\widehat{\btheta}_t}_1 - (R_\delta^n(\widehat{\btheta}'_t) + \lambda_t \norm{\widehat{\btheta}'_t}_1)\\
	&\stackrel{(1)}{\geq} (\nabla R_\delta^n(\widehat{\btheta}'_t) + \lambda_t\bxi')^T(\widehat{\btheta}_t - \widehat{\btheta}'_t) + \frac {\rho_{-}} {2} \norm{\widehat{\btheta}_t - \widehat{\btheta}'_t}_2^2\\
	&\stackrel{(2)}{\geq} \frac {\rho_{-}} {2} \norm{\widehat{\btheta}_t - \widehat{\btheta}'_t}_2^2,
	\eeq
	where (1) follows from the convexity of the penalty function and (\ref{thom1e2}), and (2) follows from
	
	\beq
	\frac{(\widehat{\btheta}'_t - \widehat{\btheta}_t)^T}{\norm{\widehat{\btheta}'_t - \widehat{\btheta}_t}_1}(\nabla R_\delta^n(\widehat{\btheta}'_t) + \lambda_t\bxi')&'
	\leq \max_{\btheta} \bigg \{
	\frac{(\widehat{\btheta}'_t - \btheta)^T}{\norm{\widehat{\btheta}'_t - \btheta}_1}(\nabla R_\delta^n(\widehat{\btheta}'_t) + \lambda_t\bxi')
	\bigg \}\\
	&\leq 0.
	\eeq
By switching the role of $\widehat{\btheta}_t$ and $\widehat{\btheta}'_t$, we can similarly show that $f_{\lambda_t}( \widehat{\btheta}'_t) - f_{\lambda_t}(\widehat{\btheta}_t)\geq  \frac {\rho_{-}} {2} \norm{\widehat{\btheta}_t - \widehat{\btheta}'_t}_2^2$. Thus, combining with (\ref{thom1e3}), we have $\widehat{\btheta}_t = \widehat{\btheta}_t'.$ Therefore the local solution is unique.
	
	Now we start to prove the geometric convergence rate at stage t. Recall that by definition, $\btheta_t^{k+1} = \cS_{\lambda_{t} \eta}\big(\btheta_t^k\big)$ is the minimizer of the following local quadratic approximation of $f_{\lambda_t}(\btheta)$ at $\btheta_t^k$
	\beq
	\psi_\eta(\btheta,\btheta_t^k)
	:= R_\delta^n(\btheta_t^k) + \nabla R_\delta^n(\btheta_t^k)^T(\btheta - \btheta_t^k) + \frac{1}{2\eta}\norm{\btheta - \btheta_t^k}_2^2 + \lambda_t\norm{\btheta}_1.
	\eeq
	By Assumption \ref{RSC} when $\eta \leq \frac{1}{\rho_+}$, we will see that
	\beq\label{thom1e4}
	f_{\lambda_t}(\btheta_t^{k+1}) &\leq \psi_\eta(\btheta_t^{k+1},\btheta_t^k) = \min_{\btheta}\psi_\eta(\btheta,\btheta_t^k)\\
	&\leq \psi_\eta(\alpha \widehat{\btheta}_t + (1 - \alpha)  \btheta_t^{k},\btheta_t^k)
	\eeq
	for some $\alpha \in [0,1]$.
	Since we know $\norm{(\widehat{\btheta}_t )_{S^{*c}}}_0 \leq \tilde{s}$ and $\norm{(\btheta^k_t )_{S^{*c}}}_0 \leq \tilde{s}$, applying Assumption \ref{RSC} gives
	\beq\label{thom1e5}
	R_\delta^n(\alpha \widehat{\btheta}_t + (1 - \alpha)  \btheta_t^{k}) \geq R_\delta^n(\btheta_t^k) + \nabla R_\delta^n(\btheta_t^k)^T((\alpha \widehat{\btheta}_t + (1 - \alpha)  \btheta_t^{k}) - \btheta_t^k).
	\eeq
	Combining (\ref{thom1e4}) and (\ref{thom1e5}) gives
	\beq \label{thom1e6}
	&f_{\lambda_t}(\btheta_t^{k+1})\\ &\leq R_\delta^n(\alpha \widehat{\btheta}_t + (1 - \alpha)  \btheta_t^{k}) + \frac{1}{2\eta}\norm{(\alpha \widehat{\btheta}_t + (1 - \alpha)  \btheta_t^{k}) - \btheta_t^k}_2^2 + \lambda_t\norm{\alpha \widehat{\btheta}_t + (1 - \alpha)  \btheta_t^{k}}_1\\
	&\leq \alpha (R_\delta^n(\widehat{\btheta}_t )+\lambda_{t}\norm{\widehat{\btheta}_t}_1) + (1-\alpha)(R_\delta^n(\btheta_t^{k})+\lambda_{t}\norm{\btheta_t^{k}}_1) + \frac{1}{2\eta}\norm{(\alpha \widehat{\btheta}_t + (1 - \alpha)  \btheta_t^{k}) - \btheta_t^k}_2^2\\
	&= \alpha f_{\lambda_t}(\widehat{\btheta}_t) + (1-\alpha)f_{\lambda_t}(\btheta_t^{k}) + \frac{1}{2\eta}\norm{(\alpha \widehat{\btheta}_t + (1 - \alpha)  \btheta_t^{k}) - \btheta_t^k}_2^2\\
	&= f_{\lambda_t}(\btheta_t^{k}) - \alpha(f_{\lambda_t}(\btheta_t^{k}) - f_{\lambda_t}(\widehat{\btheta}_t) ) + \frac{\alpha^2}{2\eta}\norm{\widehat{\btheta}_t - \btheta_t^{k} }_2^2,
	\eeq
	where the second inequality follows from the restricted strong convexity by Assumption \ref{RSC} and the convexity of the penalty function.
	
	On the other hand, similar to (\ref{thom1e3}), we can derive that
	\beq
	f_{\lambda_t}(\btheta_t^{k} ) - f_{\lambda_t}(\widehat{\btheta}_t) \geq \frac {\rho_{-}} {2} \norm{\btheta_t^{k} - \widehat{\btheta}_t}_2^2.
	\eeq
	Plugging this into the RHS of (\ref{thom1e6}) gives
	\beq
	f_{\lambda_t}(\btheta_t^{k+1})&\leq f_{\lambda_t}(\btheta_t^{k}) - \alpha(f_{\lambda_t}(\btheta_t^{k}) - f_{\lambda_t}(\widehat{\btheta}_t) ) + \frac{\alpha^2}{\eta\rho_{-}}(f_{\lambda_t}(\btheta_t^{k} ) - f_{\lambda_t}(\widehat{\btheta}_t)).
	\eeq
	By selecting $\alpha = \frac{\eta\rho_{-}}{2} < 1$(since $\eta \leq \frac{1}{\rho_{+}}$), we get
	\beq
	f_{\lambda_t}(\btheta_t^{k+1})- f_{\lambda_t}(\widehat{\btheta}_t) \leq (1-\frac{\eta\rho_{-}}{4})(f_{\lambda_t}(\btheta_t^{k} ) - f_{\lambda_t}(\widehat{\btheta}_t)),
	\eeq
	which further implies that
	\beq
	f_{\lambda_t}(\btheta_t^{k})- f_{\lambda_t}(\widehat{\btheta}_t) \leq (1-\frac{\eta\rho_{-}}{4})^k(f_{\lambda_t}(\btheta_t^{0} ) - f_{\lambda_t}(\widehat{\btheta}_t)).
	\eeq
	This completes the proof.
	
\end{proof}

\subsection{Proof of Theorem \ref{theorem_main2}}\label{proof_theo_main2}

\begin{proof}
    First of all, by proposition \ref{prop_bias} we know with probability greater than $1 - 2d^{-1}$,  $\norm{\nabla R_\delta^n(\btheta^*) - \nabla R_\delta(\btheta^*)}_\infty \leq \lambda_{tgt}/8$ holds.
	Now we prove this theorem by induction. Notice that the initialization in Algorithm \ref{pf_algo} guarantees that
	$$\norm{(\btheta^0_0)_{S^{*c}}}_0 \leq \tilde{s}\; \text{and}\; \omega_{\lambda_0}(\btheta^0_0) \leq \frac12 \lambda_0.$$
	Suppose at stage $t = 1,\dotso,N-1$, we have
	$$\norm{(\btheta^0_t)_{S^{*c}}}_0 \leq \tilde{s}\; \text{and}\; \omega_{\lambda_t}(\btheta^0_t) \leq \frac12 \lambda_t.$$
	By Proposition \ref{prop_key}, we know $\norm{(\btheta_t^k)_{S^{*c}}}_0 \leq \tilde{s}$ for $k = 1,\dotso,$ which implies that $\norm{(\tilde{\btheta}_t)_{S^{*c}}}_0 \leq \tilde{s}$ if exists.
	Recall that at stage $t = 1,\dotso,N-1$, the stopping criteria requires $\omega_{\lambda_{t}}(\btheta) \leq \frac14\lambda_{t}$, therefore it suffices to find $k$ such that $\omega_{\lambda_{t}}(\btheta_t^{k}) \leq \frac14\lambda_{t}$ to finish stage $t$.
	By Lemma \ref{lemma_nes_subop}, we have
	
	\beq\label{thom2e1}
	\omega_{\lambda_t}(\btheta_t^{k}) \leq (\frac 1\eta + \rho_{+})\norm{\btheta_t^{k} - \btheta_t^{k-1}}_2.
	\eeq
	On the other hand, from Lemma \ref{lemma_monotone}
	\beq
	\frac{1}{2\eta}\norm{\btheta_t^{k} - \btheta_t^{k-1}}_2^2  &\leq f_{\lambda_t}(\btheta_t^{k-1}) - f_{\lambda_t}( \btheta_t^{k})\\
	&\leq 	f_{\lambda_t}(\btheta_t^{k})- f_{\lambda_t}(\widehat{\btheta}_t) \\
	&\leq (1-\frac{\eta\rho_{-}}{4})^{k-1}(f_{\lambda_t}(\btheta_t^{0} ) - f_{\lambda_t}(\widehat{\btheta}_t))\\
	&\leq (1-\frac{\eta\rho_{-}}{4})^{k-1} \frac{\bar{C}_2}{\rho_{-}}\bigg( \delta^{2\beta} \lor \sqrt{s}\delta^{\beta}\lambda_t \lor   s\lambda_t^2  \bigg),
	\eeq
	where the last inequality follows from Lemma \ref{smaller_lambda}. Now it suffices to  guarantee that
	\beq
	(\frac 1\eta + \rho_{+})\sqrt{2\eta(1-\frac{\eta\rho_{-}}{4})^{k-1} \frac{\bar{C}_2}{\rho_{-}}\bigg( \delta^{2\beta} \lor \sqrt{s}\delta^{\beta}\lambda_t \lor   s\lambda_t^2  \bigg)} \leq \frac14\lambda_{t}.
	\eeq
	Recall that we choose $\delta = c\big(\frac{s\log d }{n}\big)^{1/(2\beta + 1)}$ and $\lambda_{t} > \lambda_{tgt},$ with $c \leq C^{2/(2\beta +1)}$ and $\lambda_{tgt} = C\sqrt{\frac{\log d}{n\delta}}$. With some algebra we can show that it suffices to guarantee
	\beq \label{thom_algo_e2}
	k\geq \log\bigg (
	\frac{32(\frac 1\eta + \rho_{+})^2\eta \bar{C}_2 s}{\rho_{-}}
	  \bigg) \bigg/\log\bigg(\frac{4}{4 - \eta\rho_{-}}\bigg) + 1,
	\eeq
	where the RHS is independent of $\lambda$.
	At the same time, Lemma \ref{good_ini} guarantees that $\omega_{\lambda_{t+1}}(\tilde{\btheta}_t) \leq \frac12\lambda_{t+1}.$ By induction, we conclude that for the first $N-1$ stage of the path-following algorithm, the number of iterations is no more than
	\beq
	(N-1)\bigg[
	\log\bigg (
	\frac{32(\frac 1\eta + \rho_{+})^2\eta \bar{C}_2 s}{\rho_{-}}
	\bigg) \bigg/\log\bigg(\frac{4}{4 - \eta\rho_{-}}\bigg) + 1
	\bigg].
	\eeq
	At the last stage, the initialization guarantees that
	$$\norm{(\btheta^0_N)_{S^{*c}}}_0 \leq \tilde{s}~~~ \text{and}~~~ \omega_{\lambda_{N}}(\btheta^0_N) \leq \frac12 \lambda_N = \lambda_{tgt}.$$
	A similar derivation from above will yield that it takes no more than
	\beq
	\bigg[\log\bigg (\frac{2\eta(\frac 1\eta + \rho_{+})^2\bar{C}_2s}{\rho_{-}}\cdot \frac{\lambda_{tgt}^2}{\epsilon_{tgt}^2}\bigg) \bigg/\log\bigg(\frac{4}{4 - \eta\rho_{-}}\bigg) + 1\bigg] \lor 1
	\eeq
	iterations for the $N$th stage. Following the derivation above, it's easy to see that the final estimator $\tilde{\btheta}_{tgt}$ satisfies the conditions of Lemma \ref{lemma_norm} with $\lambda = \lambda_{tgt}$, and thus achieves the desired rate of convergence with $\delta \asymp \big(\frac{s\log d }{n}\big)^{1/(2\beta + 1)}$.
	This completes the proof.
\end{proof}

\subsection{Proof of Corollary \ref{corollary_main1}}\label{proof_coro}

\begin{proof}
By Proposition \ref{prop_key}, we know for the final path $N$, the initialization satisfies
	\beq
	\norm{(\btheta^0_N)_{S^{*c}}}_0 \leq \tilde{s}~~~ \text{and}~~~ \omega_{\lambda_{N}}(\btheta^0_N) \leq \frac12 \lambda_N \asymp \bigg(
	\frac{s\log d}{n}
	\bigg)^{\beta/(2\beta + 1)}.
	\eeq
	Hence Lemma \ref{lemma_initialization} implies the desired result.
\end{proof}

\subsection{Auxiliary results for minimax lower bound}

\subsubsection{Proof of Lemma \ref{lemma_holder_verify}}\label{proof_holder_verify}
\begin{proof}
	By the construction of $P(X,Y,\bZ)$ in the proof of Theorem \ref{theorem_minimax},
  we know $\tilde{g}_{j,\bz}(x)$ and $g(x)$ are $l = \floor{\beta}$ times differentiable. By directly checking the definition of $\cP(\beta,L)$ in Definition \ref{def_holder_new}, we obtain that for $f_j(x|Y =1,\bz),j = 1,\dotso,|\cH'|$,
	\beq
	&\bigg |\int \bv^T\bz \big [
	f_j^{(l)}(\triangle + \btheta^{*T}\bz|Y =1,\bz) - f_j^{(l)}( \btheta^{*T}\bz|Y =1,\bz)
	\big ]f(\bz|Y = 1)d\bz \bigg | \\
	\leq  & \bigg |\int \bv^T\bz \big [
	g^{(l)}(\triangle + \btheta^{*T}\bz - 2\sigma |Y =1,\bz) - g^{(l)}( \btheta^{*T}\bz - 2\sigma|Y =1,\bz)
	\big ]f(\bz|Y = 1)d\bz\bigg |\\
	&+ \bigg |\int \bv^T\bz \big [ \tilde{g}_{j,\bz}^{(l)}(\triangle + \btheta^{*T}\bz|Y =1,\bz) - \tilde{g}_{j,\bz}^{(l)}( \btheta^{*T}\bz|Y =1,\bz)
	\big ]f(\bz|Y = 1)d\bz\bigg |\\
	\stackrel{(1)}{\leq}& \frac {L}{2}|\triangle|^{\beta - l} \sqrt{\int \bv^T\bz\bz^T\bv f(\bz|Y = 1)d\bz }  \\
	&~~~+  \int \bigg | \bv^T\bz L\delta^{\beta - l} \frac{\bw_j^T\bz}{\sqrt{s}}\big [
	R^{(l)}((\triangle + \btheta^{*T}\bz - x_{j,\bz})/\delta) - R^{(l)}(( \btheta^{*T}\bz - x_{j,\bz})/\delta)
	\big ] \bigg | f(\bz|Y = 1) d\bz\\
	\stackrel{(2)}{\leq}& \frac {L}{2}|\triangle|^{\beta - l}\norm{\bv}_2  \bigg ( \sqrt{M_2}+ M_2\norm{\bw_j}_2/\sqrt{s} \bigg) \\
	\leq& L|\triangle|^{\beta - l}\norm{\bv}_2,
	\eeq
	where (1) follows from the definition of $g\in \Sigma(\beta,L/2)$, $\tilde{g}_{j,\bz}$ and Jensen's inequality and in step (2) we use the fact that $\EE[ZZ^T|Y = 1] = 1/3\II_d$ and thus $M_2=1/3$, and for the second term
	\begin{align*}
	 &\int \bigg | \bv^T\bz L\delta^{\beta - l} \frac{\bw_j^T\bz}{\sqrt{s}}\big [
	R^{(l)}((\triangle + \btheta^{*T}\bz - x_{j,\bz})/\delta) - R^{(l)}(( \btheta^{*T}\bz - x_{j,\bz})/\delta)
	\big ] \bigg | f(\bz|Y = 1) d\bz\\
	\leq &L\delta^{\beta - l}\Big[\int  (\bv^T\bz)^2f(\bz|Y = 1) d\bz\Big]^{1/2} \Big[\int\frac{(\bw_j^T\bz)^2}{{s}}(\frac{\triangle}{\delta})^{2(\beta-l)}f(\bz|Y = 1) d\bz\Big]^{1/2}
	\end{align*}
	followed by $R \in \Sigma(\beta,1/2)$.
	
	Meanwhile, for all $\bz$, $f_0(x|Y =1,\bz)$ and $f_j(x|Y= -1,\bz),j = 0,\dotso,|\cH'|$ are well defined density function in  $\Sigma(\beta,L/2)$. Therefore, following a similar derivation from above we can also show that these densities satisfy the smoothness condition defined in \ref{eq_holder_new}. Combining the  results above completes the proof.
\end{proof}

\subsubsection{Proof of Lemma \ref{lemma_sparse_mini}}\label{proof_lemma_sparse_mini}
\begin{proof}
	Let $R_j(\btheta)$ be the risk under $P_j$.
	Recall that by definition
	\[
	R_j(\btheta) = \EE_{j} \bigg[w(Y)\bigg(1 - \sign{Y(X - \btheta^T\bZ)}  \bigg)  \bigg],
	\]
	where the subscript $j$ corresponds to the probability measure $P_j$. With some algebra, under the distributional assumption the difference between $R_j(\btheta)$ and $R_j(\btheta_j)$ can be written as
	\beq
	R_j(\btheta) - R_j(\btheta_j) &= \EE_{j}\bigg[w(Y)Y(\sign{X - \btheta_j^T\bZ} - \sign{X - \btheta^T\bZ})\bigg]\\
	& = 4\int_{\cG}\sign{x - \btheta_j^T\bz}\EE_{j}[Y|x,\bz] dP_{j;X,\bZ},
	\eeq
	where
	\beq
	\cG = \{(x,\bz) | \sign{x - \btheta^T\bz} \neq \sign{x - \btheta_j^T\bz}  \}
	\eeq
	and $P_{j;X,\bZ}$ is the joint distribution of $(X,\bZ)$ under  $P_j$. Notice that
	\beq\label{equ_minimax_mini_2}
	\sign{\EE_{j}[Y|x,\bz]} &= \sign{\PP_{j}(Y = 1|x,\bz) - \PP_{j}(Y = -1|x,\bz)}\\
	& = \sign{f_j(x|Y=1,\bz) - f_j(x|Y=-1,\bz)}.\\
	\eeq
	We consider two cases:\\
	When $j = 0$, (\ref{equ_minimax_mini_2}) reduces to
	\[
	\sign{g(x - 2\sigma) -g(x + 2\sigma) } = \sign{x - \btheta_0^T\bz},
	\]
	where the equality follows from the symmetry of $g$ and $\btheta_0 = 0$. This implies that
	\beq
	R_0(\btheta) - R_0(\btheta_0) &= 4\int_{\cG}|\EE_{0}[Y|x,\bz]| dP_{0;X,\bZ} \geq 0,
	\eeq
	which further implies that $\btheta_0$ is a minimizer of $R_0(\btheta)$.
	
	When $j\neq0$ (\ref{equ_minimax_mini_2}) can be expressed as
	\[
	\sign{\EE_{j}[Y|x,\bz]} = \sign{g(x-2\sigma) - g(x + 2\sigma) + L\delta^\beta R(\frac{x - x_{i,\bz}}{\delta})\frac{\bomega_j^T\bz}{\sqrt{s}}}.
	\]
	In the following, we aim to show that \beq\label{eq_lemma_mini_1}
	\sign{\EE_{j}[Y|x,\bz]} = \sign{x - \btheta_j^T\bz}~\forall j = 1,\dotso,|\cH|.\eeq
	
	{\textbf{Step 1: }} The first step is to show that $\sign{\EE_{j}[Y|\btheta_j^T\bz,\bz]} = 0$.
	Recall that we choose $\delta \asymp \big(\frac{s\log(d/s)}{n}\big)^{1/(2\beta+1)}$ and $\btheta_j = \frac{\delta^\beta L\sigma^2 a e^{-1}}{2m\sqrt{s}} \bomega_j $. Under the conditions of Theorem \ref{theorem_minimax}, with a large enough constant $C$, $|\btheta_j^T\bz|$ will be bounded above by $\sigma$ , and therefore we obtain
	\[
	g(\btheta_j^T\bz-2\sigma) - g(\btheta_j^T\bz + 2\sigma) = g(\btheta_j^T\bz-2\sigma) - g(-\btheta_j^T\bz - 2\sigma) = \frac{2m}{\sigma^2}\btheta_j^T\bz,
	\]
	since $g$ is symmetric around 0 and linear on $[-3\sigma,-\sigma]$ by construction. In addition, we know $\forall \bz$, $\btheta_j^T\bz = x_{j,\bz} +\delta$, which implies that $R(\frac{\btheta_j^T\bz - x_{i,\bz}}{\delta}) \equiv R(1) = -a e^{-1}.$ Therefore we obtain $\forall \bz$,
	\beq\label{eq_lemma_mini_2}
	&g(\btheta_j^T\bz-2\sigma) - g(-\btheta_j^T\bz - 2\sigma) + L\delta^\beta R(\frac{\btheta_j^T\bz - x_{i,\bz}}{\delta})\frac{\bomega_j^T\bz}{\sqrt{s}}\\
	&=\frac{2m}{\sigma^2}\btheta_j^T\bz + \frac{-L\delta^\beta a e^{-1} \bomega_j^T\bz}{\sqrt{s}}\\
	&= 0,
	\eeq
	which implies the desired result.
	
	{\textbf{Step 2: }}
	Now in order to make (\ref{eq_lemma_mini_1}) holds, it's equivalent to show that $ g(x-2\sigma) - g(x + 2\sigma) + L\delta^\beta R(\frac{x - x_{i,\bz}}{\delta})\frac{\bomega_j^T\bz}{\sqrt{s}} > 0$  when $x > \btheta_j^T\bz$ and also true reversely when $x < \btheta_j^T\bz$.
	\begin{figure}[ht]
		\centering
		\includegraphics[width= 8cm,height=8cm]{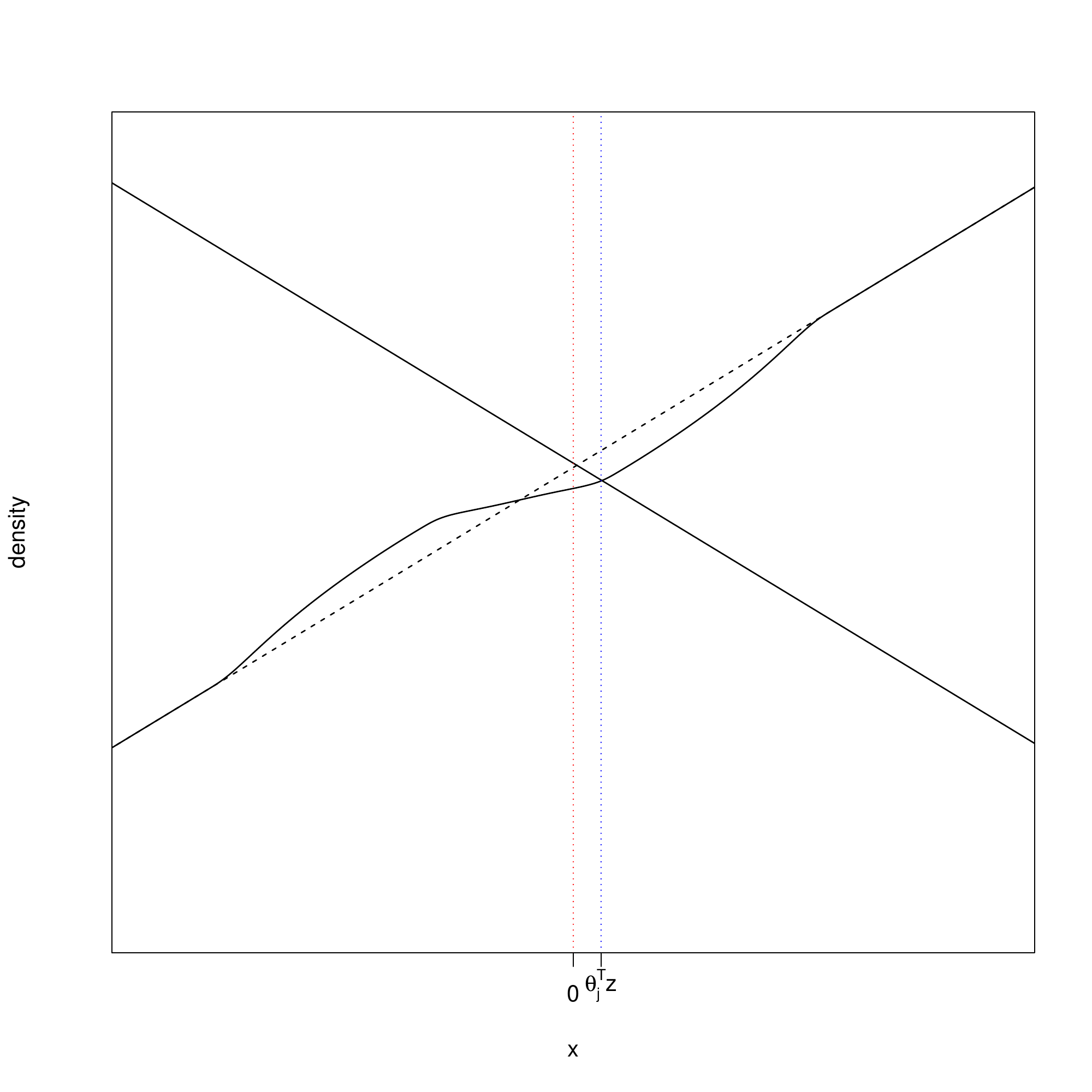}
		\caption{A demonstration of the shape of conditional density $f_j(x|Y=1,z)$ and $f_j(x|Y=-1,z)$ locally around $\btheta_j^T\bz$ with a unique intersection at $\btheta_j^T\bz$, where the strictly increasing curve corresponds to $f_j(x|Y=1,z)$.
		}\label{plot_minimax}
	\end{figure}
	
	Consider two cases:
	\begin{enumerate}
		\item When $x \in (-\infty,-\sigma] \cup [\sigma,\infty)$,  recall that by (\ref{eq_minimax_1}), $\tilde{g}_{j,\bz}(x) = 0~\forall x \in (-\infty,-\sigma] \cup [\sigma,\infty)$. Since we choose $C$ large enough to additionally ensure that $|\btheta^T_j\bz| < \sigma$, we obtain that $\forall x \in (-\infty,-\sigma] \cup [\sigma,\infty)$
		\beq
		&\sign{g(x-2\sigma) - g(x + 2\sigma) + L\delta^\beta R(\frac{x - x_{i,\bz}}{\delta})\frac{\bomega_j^T\bz}{\sqrt{s}} }\\
		=& \sign{g(x-2\sigma) - g(x + 2\sigma)}\\ =& \sign{x}\\ =&  \sign{x - \btheta^T_j\bz}.
		\eeq
		\item When $x \in (-\sigma,\sigma)$,
		it's easy to check that $\tilde{g}(x) = g(x - 2\sigma) - g(x + 2\sigma)$ is increasing linearly with gradient $\tilde{g}'(x) = \frac{2m}{\sigma^2}$. Meanwhile, with a large constant $C$, we can ensure that
		\beq\label{eq_unique_11}
		|\tilde{g}'_{j,\bz}(x)| = \bigg|L\delta^{\beta-1} R'(\frac{x - x_{i,\bz}}{\delta})\frac{\bomega_j^T\bz}{\sqrt{s}} \bigg | < \frac{m}{\sigma^2}\eeq
		uniformly on $(-\sigma,\sigma)$. This implies that $g(x-2\sigma) - g(x + 2\sigma) + L\delta^\beta R(\frac{x - x_{i,\bz}}{\delta})\frac{\bomega_j^T\bz}{\sqrt{s}} $ is a strictly increasing function on  $(-\sigma,\sigma)$, which  together with (\ref{eq_lemma_mini_2}) further implies that
		\beq
		\sign{g(x-2\sigma) - g(x + 2\sigma) + L\delta^\beta R(\frac{x - x_{i,\bz}}{\delta})\frac{\bomega_j^T\bz}{\sqrt{s}}} = \sign{x - \btheta^T_j\bz}, ~~x \in (-\sigma,\sigma).
		\eeq
		For better illustration, in Figure \ref{plot_minimax} we plot the shape of $f_j(x|Y=1,\bz)$ and $f_j(x|Y=-1,\bz)$ around 0 with a unique intersection at $\btheta_j^T\bz$, where $f_j(x|Y=1,\bz)$ is strictly increasing and $f_j(x|Y=-1,\bz)$ is strictly decreasing locally.
	\end{enumerate}
	
	Combining the above cases, we conclude that (\ref{eq_lemma_mini_1}) holds for all $j = 1,\dotso,|
	\cH|.$ Similar to the case when $j = 0$, we obtain
	\beq
	R_j(\btheta) - R_j(\btheta_j) &= 4\int_{\cG}|\EE_{j}[Y|x,\bz]| dP_{j;X,\bZ} \geq 0~~,j = 1,\dotso,|\cH|.
	\eeq


	To see the uniqueness, for any $\btheta \neq \btheta_j$, consider the set
	$$
	\cG_{\bz}=\{\bz: (\btheta-\btheta_j)^T\bz\neq 0, |\btheta^T\bz|\leq \sigma, |\btheta_j^T\bz|\leq \sigma\},
	$$
	which is the set of $\bz$ such that $\btheta^T\bz\neq \btheta_j^T\bz$. Since there exists an open neighborhood in $\cG_{\bz}$, the set $\cG_{\bz}$ has nonzero measure. Then define
	\beq
	\bar{\cG} = \big  \{
	(x,\bz): \btheta^T\bz < x < \btheta_j^T\bz~~\textrm{or}~~\btheta_j^T\bz < x < \btheta^T\bz, \bz\in\cG_{\bz}
	\big\}.
	\eeq
	By construction, we know $\bar{\cG}$ has nonzero measure and $\bar{\cG} \subset \cG$. Moreover, (\ref{equ_minimax_mini_2}) implies that on this set $\sign{\EE_{j}[Y|x,\bar{\bz}} \neq 0$, which further implies that $R(\btheta) - R(\btheta_j) > 0$ for any $\btheta \neq \btheta_j$.
	This completes the first part of the proof.

    Now we show that $\lambda_{\min}(\nabla^2R_P(\btheta_j)) \geq \rho>0$ for some $\rho$.
    Direct calculation gives that by construction
    \beq\label{eq_unique_hessian}
    \nabla^2 R(\btheta) &= \sum_{y = \pm 1}w(y)\int_{\bZ} \bz\bz^T y f'(\btheta^T\bz|\bz,y) f(\bz,y) d\bz\\
    &= 2
    \int_{\bZ} \bz\bz^T
    \big[
    f'(\btheta^T\bz|\bz,Y=1)  - f'(\btheta^T\bz|\bz,Y=-1)
    \big ]f(\bz) d\bz.
    \eeq
    When $j = 0$, we have
    \beq
    \nabla^2 R_0(\mathbf{0}) = 2
    \int_{\bZ} \bz\bz^T
    \big[
    g'(-2\sigma)  - g'(2\sigma)
    \big ]f(\bz) d\bz.
    \eeq
    Since $g'(-2\sigma)  - g'(2\sigma) = 2m/\sigma^2 >0$ by the definition of $g$, we have $$
    \lambda_{\min}(\nabla^2R_0(\mathbf{0}))=\frac{4m}{\sigma^2}\lambda_{\min}(\Cov(\bZ))=\frac{4m}{3\sigma^2}>0.
    $$
    When $j >0$, we have
    \beq
    \nabla^2 R_j(\btheta_j) = 2
    \int_{\bZ} \bz\bz^T
    \big[
    g'(\btheta_j^T\bz-2\sigma) + \tilde{g}'_{j,\bz}(\btheta_j^T\bz)  - g'(\btheta_j^T\bz + 2\sigma)
    \big ]f(\bz) d\bz.
    \eeq
    { Recall that the conditions in Theorem \ref{theorem_minimax} ensure that $|\btheta_j^T\bz|<\sigma$. Also recall that by construction, $g$ is linear on $[-3\sigma,-\sigma]$ and $[\sigma,3\sigma]$ with gradient $m/\sigma^2$ and $-m/\sigma^2$, respectively. This implies that $g'(\btheta_j^T\bz-2\sigma) - g'(\btheta_j^T\bz + 2\sigma) = 2m/\sigma^2$ for all $j>0$ and $\bz \in [-1,1]^d$, which together with (\ref{eq_unique_11}) implies that $g'(\btheta_j^T\bz-2\sigma) + \tilde{g}'_{j,\bz}(\btheta_j^T\bz)  - g'(\btheta_j^T\bz+2\sigma) >m/\sigma^2$.
    Similar to the case when $j = 0$, we can show that
    $$
    \lambda_{\min}(\nabla^2R_j(\btheta_j))>\frac{2m}{3\sigma^2}>0.
    $$
    This completes the proof.

    }

\end{proof}

\subsection{Proof for adaptive estimations}
\subsubsection{Proof of Theorem \ref{theo_adapt_1}}\label{proof_adaptive1}
\begin{proof}
	For simplicity we write $\tilde{\btheta}_{\delta,\lambda_{\delta}} = \tilde{\btheta}_{\delta}$ unless otherwise explained.
	Firstly, we define the oracle $\delta_*$ as
	\[
	\delta_* = \max \bigg \{
	\delta \in \cD, \delta^\beta \leq c'\sqrt{\frac{s\log d}{n\delta}}
	\bigg \},
	\]
	where $c'$ is to be chosen later. By definition, we have
	\beq\label{eq_ada_1}
	c'' \sqrt{\frac{s\log d}{n\delta_*}}\leq \delta_*^\beta \leq c'\sqrt{\frac{s\log d}{n\delta_*}},
	\eeq
	where $c''$ is a constant that depends on $\beta$ and $c'$.
	In the sequel, we want to show that $\hat{\delta} \geq \delta_*$ holds w.h.p.
	
	Notice that Proposition \ref{prop_var} implies that
	\beq
	\PP \bigg ( \norm{\nabla R_\delta^n(\btheta^*) - \nabla R_\delta(\btheta^*)}_\infty  > C_1\sqrt{\frac{\log d}{n\delta}}\bigg )\leq 2d^{-1},
	\eeq
	where $C_1$ is a constant.
	Therefore, with probability greater than $1 - 2\log_{2}(n)/d$, by the choice of $\lambda_{\delta} = C\sqrt{\frac{\log d}{n\delta}}$ with $C \geq 8C_1$, for each $\delta \in \cD$ in (\ref{leiski_1}), the event $\norm{\nabla R_\delta^n(\btheta^*) - \nabla R_\delta(\btheta^*)}_\infty \leq \lambda_{\delta}/8$ holds.
	Now on this event, by choosing $c'$ no greater than $C$, we can ensure that
	\beq
	\delta_*^\beta \leq c'\sqrt{\frac{s\log d}{n\delta_*}}\leq  C\sqrt{\frac{s \log d}{n\delta_*}} = \sqrt{s}\lambda_{\delta_*}.
	\eeq
	Then for each $\delta \leq \delta_*, \delta \in \cD$, this further implies that
	\beq\label{lepski_1_eq1}
	\delta^{\beta} \leq \delta_*^{\beta} \leq \sqrt{s}\lambda_{\delta_*} \leq \sqrt{s}\lambda_{\delta}.
	\eeq
	In particular, $\lambda_\delta\geq \lambda_{\delta_*}=C\sqrt{\frac{\log d}{n\delta_*}}$.
	By choosing the constant $C$ large enough and by (\ref{eq_ada_1}), this ensures $\lambda_\delta\geq \lambda_{tgt}$. Thus, on the event defined above, Lemma \ref{lemma_initialization} implies $\norm{\tilde{\btheta}_{\delta} - \btheta^*}_2 \leq \frac{\bar{C}_1}{\rho_{-}}\big( \delta^{\beta} \lor \sqrt{s}\lambda \big)$ for each $\delta \leq \delta_*, \delta \in \cD$.
	Therefore, we conclude that with probability greater than $1 - 2\log_{2}(n)/d$, for each $\delta \leq \delta_*, \delta \in \cD$ with corresponding choice of $\lambda_{\delta}$, it holds that
	\beq
	\norm{\tilde{\btheta}_{\delta} - \tilde{\btheta}_{\delta_*}}_2 &\leq
	\norm{\tilde{\btheta}_{\delta} - \btheta^*}_2 + \norm{\tilde{\btheta}_{\delta_*} - \btheta^*}_2\\
	& \leq \frac{\bar{C}_1}{\rho_-}\big[ (\delta^{\beta} \lor \sqrt{s}\lambda_{\delta}) + (\delta_*^{\beta} \lor \sqrt{s}\lambda_{\delta_*})\big]\\
	&\leq \frac{2\bar{C}_1}{\rho_-}\sqrt{s}\lambda_{\delta}\\
	&\leq c\sqrt{\frac{s \log d}{n\delta}},
	\eeq
	where the third inequality follows from (\ref{lepski_1_eq1}), and the last follows by choosing the constant $c$ in (\ref{leiski_1}) large enough.
	This implies that $\hat{\delta} \geq \delta_*$ by the definition of $\hat{\delta}$. Thus we obtain
	\beq
	\norm{\tilde{\btheta}_{\hat{\delta}} - \btheta^*}_2 &\leq \norm{\tilde{\btheta}_{\hat{\delta}} - \tilde{\btheta}_{\delta_*}}_2 + \norm{\tilde{\btheta}_{\delta_*} - \btheta^*}_2\\
	&\lesssim \sqrt{\frac{s\log d}{n\delta_*}} + \norm{\tilde{\btheta}_{\delta_*} - \btheta^*}_2\\
	&\lesssim \bigg(\frac{s\log d  }{n}\bigg)^{\beta/(2\beta+1)},
	\eeq
	where the second inequality follows from the definition of $\hat{\delta}$  and the last inequality follows from (\ref{eq_ada_1}) and Theorem \ref{theorem_main2}. This completes the proof.
\end{proof}

\subsubsection{Proof of Theorem \ref{theo_adapt_2}}\label{proof_adaptive2}
\begin{proof}
	This proof is very similar to the proof for Theorem \ref{theo_adapt_1}. We first define the oracle choice of $s_*$ as the largest number in $\cD'$ no greater than $s$,
	\beq
	s_* = \max\{s' \in \cD': s' \leq s\} .
	\eeq
	By definition, we have
	\beq\label{lepski_oracleeq_2}
	\frac12 s < s_* \leq s.
	\eeq
	In the sequel, we want to show that $s_* \geq \hat{s}$ holds w.h.p.
	Notice that Proposition \ref{prop_var} implies that
	\beq
	\PP \bigg ( \norm{\nabla R_\delta^n(\btheta^*) - \nabla R_\delta(\btheta^*)}_\infty  > C_1\sqrt{\frac{\log d}{n\delta}}\bigg )\leq 2d^{-1},
	\eeq
	where $C_1$ is a constant.
	Therefore, with probability greater than $1 - 2\log_{2}(d)/d$, by the choice of $\lambda_{s'} = C\sqrt{\frac{\log d}{n\delta_{s'}}}$ with $C \geq 8C_1$,  for each $s' \in \cD'$ in (\ref{leiski_2}),  the condition $\norm{\nabla R_{\delta_{s'}}^n(\btheta^*) - \nabla R_{\delta_{s'}}(\btheta^*)}_\infty \leq \lambda_{s'}/8$ holds. On this event, the choice of $\delta_{s'}$ in (\ref{leiski_2}) together with Lemma \ref{lemma_initialization} further implies that for each $s' \geq s_*$,
	\beq
	\norm{\tilde{\btheta}_{\delta_{s_*},\lambda_{s_*}} - \tilde{\btheta}_{\delta_{s'},\lambda_{s'}}}_2 &\leq \norm{\tilde{\btheta}_{\delta_{s_*},\lambda_{s_*}}  - \btheta^*}_2 + \norm{\tilde{\btheta}_{\delta_{s'},\lambda_{s'}} - \btheta^*}_2\\
	&\leq \frac{\bar{C}_1}{\rho_{-}}\bigg[
	(\delta_{s_*}^{\beta} \lor \sqrt{s_*}\lambda_{s_*}) + (\delta_{s'}^{\beta} \lor \sqrt{s'}\lambda_{s'})
	\bigg]\\
	&\leq \frac{\bar{2C}_1}{\rho_{-}}(\delta_{s'}^{\beta} \lor \sqrt{s'}\lambda_{s'})\\
	&\leq \bar{c} \bigg(
	\frac{s' \log d}{n}
	\bigg)^{\beta/(2\beta+1)},
	\eeq
	where the last step is achieved by choosing the constant $\bar{c}$ in (\ref{leiski_2}) large enough. By the definition of $\hat{s}$, we conclude that $s_* \geq \hat{s}$  with probability greater than $1 - 2\log_{2}(d)/d$.
	Now we have
	\beq
	\norm{\tilde{\btheta}_{\hat{s},\lambda_{\hat{s}}} - \btheta^*}_2 & \leq
	\norm{\tilde{\btheta}_{\hat{s},\lambda_{\hat{s}}} - \tilde{\btheta}_{\delta_{s_*},s_*}}_2 + \norm{\tilde{\btheta}_{\delta_{s_*},s_*} - \btheta^*}_2\\
	&\lesssim \bigg(\frac{s_*\log d  }{n}\bigg)^{\beta/(2\beta+1)}\\
	&\lesssim \bigg(\frac{s\log d  }{n}\bigg)^{\beta/(2\beta+1)},
	\eeq
	where the second inequality follows from the definition of $\hat{s}$ for the first term and Lemma \ref{lemma_initialization} for the second term, and the last inequality follows from (\ref{lepski_oracleeq_2}). The proof is complete.
	
\end{proof}

\section{Verification of RSC/RSM conditions in specific models}\label{app_RSC}
In this section, we verify that the restricted strong convexity and restricted smoothness condition hold w.h.p for the class of conditional mean model introduced in Section \ref{simulation}. For simplicity, we assume $Y\sim \text{Uniform}(\{-1,1\})$, $w(y) = 1/{\PP(Y = y)} = 2$ is known. Recall the model is defined as
\beq\label{def_condi_mean}
X = \btheta^{*T}\bZ + \mu Y + u.
\eeq
Without loss of generality, we can rescale the above model such that $var(u)=1$. In addition, assume $\bZ\in \RR^d$ is  a zero-mean sub-Gaussian vector with parameter $\sigma^2$, $\bZ \independent u \sim N(0,1)$ and $\mu >0$. Notice that in this case, the conditional density $f(x|y,\bz) \in C^{\infty}$ and thus the smoothness parameter can be arbitrarily large.
We set $\Omega = \{\btheta:\norm{\btheta}_2 \leq R\}$ where $\|\btheta^*\|_2\leq R$ and apply a ``proper'' kernel function $K$ defined in Definition \ref{def_kernel1} satisfying (i) $K$ has bounded support on $[-1,1]$, (ii)  $\norm{K}_\infty,\norm{K'}_\infty$ and $-\int K'(t)tdt > 0$ are bounded above by universal constants.  Common kernels constructed by orthogonal polynomial basis will satisfy this condition \citep{tsybakovintroduction}.
To verify that Assumption \ref{RSC} holds w.h.p, it suffices to show that the following sparse eigenvalue condition holds w.h.p:

\beq
&\rho_{\max} = \sup\bigg \{
\bv^T\nabla^2R_\delta^n(\btheta)\bv: \norm{\bv}_2 = 1,\norm{\bv}_0 \leq Cs, \btheta \in \Omega,\norm{\btheta}_0\leq Cs \bigg\} < +\infty,\\
&\rho_{\min} = \inf\bigg \{
\bv^T\nabla^2R_\delta^n(\btheta)\bv: \norm{\bv}_2 = 1,\norm{\bv}_0 \leq Cs, \btheta \in \Omega,\norm{\btheta}_0\leq Cs
\bigg \} > 0
\eeq
are two constants. We let $\bSigma_{\bZ}=\Cov(\bZ)$ and $\phi(\cdot)$ be the p.d.f of the error $u$.

\begin{proposition}\label{prop_verify_main}
	Under the setup above, suppose \beq\label{eq_prop_verify_0}
	c\sigma^2\exp\bigg(
	-\frac{c'T^2}{\sigma^2R^2}
	\bigg) \leq \frac{\tilde{K}\tilde{\phi'}}{2\tilde{K}(\tilde{\phi'} + \norm{\phi'}_\infty)}\lambda_{\min}(\bSigma_{\bZ})~~~\text{and}~~~ |\delta + T|\leq \mu/2,
	\eeq
	where $\tilde{K} = -\int K'(t)tdt>0$,  $\tilde{\phi'} = \min_{u \in [\frac{-3\mu}{2},\frac{-\mu}{2}]}\phi'(u)>0$, $T,c,c'$ are positive constants. If
	\[
	n \geq C\bigg(
	\frac{s\log((ds\log d)^2/\delta)}{\delta^3} \lor \frac{s^2\log((ds\log d)^2/\delta)\log(d\lor n)}{\delta^2}
	\bigg),
	\]
	for some constant $C$ large enough, then with probability greater than $1 - c(\log n)^{-1}$, it holds that
	\beq
	\rho_{\min} \geq \frac12\tilde{K}\tilde{\phi'}\lambda_{\min}(\bSigma_{\bZ}) > 0,~~~
	\rho_{\max} \leq  \frac52\tilde{K} \norm{\phi'}_\infty \lambda_{\max}(\bSigma_{\bZ}) < +\infty.
	\eeq
\end{proposition}

\begin{proof}(Proof of Proposition \ref{prop_verify_main})
	Firstly,
the following lemma which will be proved later shows that the population risk $R_\delta = \EE[R_\delta^n]$ is strongly convex and smooth for all $\btheta \in \Omega$:
\begin{lemma}\label{prop_verify}
	Suppose
	\beq
	c\sigma^2\exp\bigg(
	-\frac{c'T^2}{\sigma^2R^2}
	\bigg) \leq \frac{\tilde{K}\tilde{\phi'}}{2\tilde{K}(\tilde{\phi'} + \norm{\phi'}_\infty)}\lambda_{\min}(\bSigma_{\bZ})~~~\text{and}~~~ |\delta + T|\leq \mu/2,
	\eeq
	where $\tilde{K} = -\int K'(t)tdt$,  $\tilde{\phi'} = \min_{u \in [\frac{-3\mu}{2},\frac{-\mu}{2}]}\phi'(u)$, $T,c,c'$ are positive constants. Then for all unit vector $\bv \in \RR^d$ the population risk $R_\delta$ satisfies
	\beq
	\bv^T\nabla^2R_\delta (\btheta)\bv \geq \tilde{K}\tilde{\phi'}\lambda_{\min}(\bSigma_{\bZ}) > 0,~~~
	\bv^T\nabla^2R_\delta (\btheta)\bv \leq  2\tilde{K} \norm{\phi'}_\infty \lambda_{\max}(\bSigma_{\bZ}) < +\infty.
	\eeq
\end{lemma}
Taking Lemma \ref{prop_verify} as given for now, suppose $\{(x_i,y_i,\bz_i)\}_{i = 1}^n$ is drawn i.i.d from the conditional mean model (\ref{def_condi_mean}).
Now consider the class of functions
\[
\cF = \bigg \{
f_{\bv,\btheta}(x,y,\bz) = -y(\bv^T\bz)^2 K'\bigg ( \frac{x - \btheta^T\bz}{\delta}\bigg):  \bv,\btheta \in \RR^d, \norm{\bv}_2 = 1,\btheta \in \Omega,\norm{\bv}_0,\norm{\btheta}_0 \leq Cs
\bigg
 \}.
\]
Our goal reduces to control the empirical process
\beq
\norm{\bG_n}_{\cF} &= \sup_{f\in \cF} \bigg|n^{-1/2}\sum_{i = 1}^n\bigg(
f(x_i,y_i,\bz_i) - \EE(f(X,Y,{\bZ}))
\bigg)\bigg |\\
&=n^{1/2}\delta^{2} \sup_{ \norm{\bv}_2 = 1, \norm{\bv}_0 \leq Cs}
\sup_{\btheta \in \Omega, \norm{\btheta}_0 \leq Cs}
\bigg |
\bv^T(\nabla^2R_\delta^n(\btheta) - \nabla R_\delta(\btheta))\bv
\bigg |.
\eeq
To achieve this, we will apply the following maximal inequality due to \citet{chernozhukov2014gaussian}.
\begin{lemma}\label{lemma_emp} Based on the setup above,
	suppose that $F \geq \sup_{f\in \cF}|f|$ is a measurable envelope with $\norm{F}_{P,2} < \infty$. Let $M = \max_{i\leq n}F(x_i,y_i,\bz_i)$ and
 $\rho^2$ be  any positive number such that $\sup_{f\in \cF} \norm{f}^2_{P,2} \leq \rho^2 \leq \norm{F}_{P,2}^2.$
	Suppose there exist $a\geq e$ and $\nu \geq 1$ such that
	\[
	\log\sup_Q N(\epsilon\norm{F}_{Q,2},\cF,\norm{.}_{Q,2}) \leq \nu(\log (a/\epsilon)),~~~\epsilon \in (0,1],
	\]
	then
	\[
	\EE_P[\norm{\bG_n}_{\cF}] \leq K\bigg(
	\rho\sqrt{\nu\log(a\norm{F}_{P,2}/\rho)} + \frac{\nu\norm{M}_{P,2}}{\sqrt{n}}\log(a\norm{F}_{P,2}/\rho)
	\bigg).
	\]
	Moreover, when $a \geq n$, with probability greater than $1 - c(\log n)^{-1}$,
	\[
	\norm{\bG_n}_{\cF} \leq K(q,c)\bigg(
	\rho\sqrt{\nu\log(a\norm{F}_{P,2}/\rho)} + \frac{\nu\norm{M}_{P,q}}{\sqrt{n}}\log(a\norm{F}_{P,2}/\rho)
	\bigg),
	\]
	where $K,K(q,c)$ are absolute constants, $Q$ is any finitely discrete probability measure, $N(\epsilon,\cF,\norm{.})$ is the covering number w.r.t the class $\cF$ and radius $\epsilon$, and $\norm{f}_{P,p} = (\int |f|^pdP)^{1/p}.$
\end{lemma}
Write
\beq
&\cF_1(q) = \{(x,y,\bz)  \rightarrow (\bv^T\bz)^q:\norm{\bv}_2 = 1,\norm{\bv}_0 \leq Cs\},\\&\cF_2 = \{
(x,y,\bz) \rightarrow K'(x - \btheta^T\bz/\delta):\btheta \in \Omega,\norm{\btheta}_0\leq Cs
\},\\
&\cF_3 = \{(x,y,\bz)  \rightarrow y\}.
\eeq
For $\cF_3$ it holds trivially that
\[
\log\sup_Q N(\epsilon,\cF_3,\norm{.}_{Q,2}) \lesssim \log(e/\epsilon).
\]
For $\cF_1(1)$, we know it is contained in the union of at most $\binom{d}{Cs}$ VC-subgraph classes of functions with VC indices bounded by $C's$, and thus for the envelope $F_1 = \sup_{\norm{\bv}_2 = 1,\norm{\bv}_0 \leq Cs} \norm{\bv^T\bz}_2 $, we have
\[
\log\sup_Q N(\epsilon\norm{F_1}_{Q,2},\cF_1(1),\norm{.}_{Q,2}) \lesssim
(s\log(d) + \log(e/\epsilon)).
\]
For $\cF_2$, notice that $\cF'_{2}:= \{(x,y,\bz)  \rightarrow x - \btheta^T\bz\}$ is contained in the union of at most $\binom{d}{Cs}$ VC-subgraph classes of functions with VC indices bounded by $C's$. By assumption $K'$ is of bounded variation, i.e., it can be written as the difference of two bounded nondecreasing functions. By Lemma 2.6.18 of \citet{vanweak} the VC-subgraph property is preserved and we obtain
\[
\log\sup_Q N(\epsilon\norm{K'}_{\infty},\cF_2,\norm{.}_{Q,2}) \lesssim
(s\log(d) + \log(e/\epsilon)).
\]
Since $\cF$ is contained in the class of pointwise product $\cF_1(2)\cdot \cF_2\cdot\cF_3$, applying Lemma A.6 of \citet{chernozhukov2014gaussian}, we obtain that for the choice of envelope
\[
F = F_1^2\norm{K'}_\infty = \sup_{\norm{\bv}_2 = 1,\norm{\bv}_0 \leq Cs} |\bv^T\bz|^2\norm{K'}_{\infty}
\]
it holds that
\beq
\log\sup_Q N(\epsilon\norm{F}_{Q,2},\cF,\norm{.}_{Q,2}) \lesssim
(s\log(d) + \log(e/\epsilon)).
\eeq
Direct calculation gives that
\beq
\sup_{f\in \cF}\norm{f}^2_{P,2} \leq 16 \delta\norm{K'}^2_{\infty}\norm{\phi}_{\infty}\sigma^4 =: \rho^2.
\eeq

Now for any sparse index set $S \subseteq \{1,\dotso,d \} ,|S|\leq Cs$, consider the sphere $\mathcal{S}_S =\{ \bv : \norm{\bv}_2 = 1, supp(\bv) \subseteq S  \}$.
Let $\cN^{(\mathcal{S}_S)}_{1/6} \subseteq \mathcal{S}_S$ denote a $1/6$-net of $\cS_s$ such that $\forall~\bv \in \mathcal{S}_S$, $\exists \bu \in \cN^{(\mathcal{S}_S)}_{1/6}$ with $\norm{\bu - \bv} \leq 1/6$. We know the cardinality of $\cN^{(\mathcal{S}_S)}_{1/6}$ is no larger than $18^{Cs}$ and 	
\beq
\sup_{\bv \in \mathcal{S}_S }
|\bv^T{\bZ}|^4
\leq c \max_{\bv \in \cN^{(\mathcal{S}_S)}_{1/6} }
|\bv^T{\bZ}|^4
\eeq
for some absolute constant $c$.
Since there are no more than $\binom{d}{Cs}$ number of such sparse index sets, we know the cardinality of $$\bar{\mathcal{S}} = \bigcup_{\substack{S \subseteq \{1,\dotso,d \} ,|S|\leq Cs}} \cN^{(\mathcal{S}_S)}_{1/6}$$
is no larger than $(18d)^{Cs}$. Therefore we obtain
\beq
\norm{F}_{P,2}^2 &= \norm{K'}_\infty^2\EE_P\bigg[\sup_{\norm{\bv}_2 = 1,\norm{\bv}_0 \leq Cs} |\bv^T{\bZ}|^4\bigg]\\
&\lesssim \norm{K'}_\infty^2\EE_P\bigg[\max_{\norm{\bv}\in \bar{\mathcal{S}}} |\bv^T{\bZ}|^4\bigg]\\
&\lesssim  \norm{K'}_\infty^2\sigma^4(s\log(d))^2,
\eeq
where the last inequality follows from the proof of Lemma 14.12 in \citet{buhlmann2011statistics}. With a similar derivation, we can also show that
\beq
\norm{M}_{P,2}^2 = \norm{\max_{i\leq n}F(x_i,y_i,\bz_i)}_{P,2}^2 \lesssim \norm{K'}_\infty^2\sigma^4(s\log(d\lor n))^2.
\eeq
Plugging the above display into Lemma \ref{lemma_emp}, we obtain that with probability greater than $1 - c(\log n)^{-1}$
\beq
&\sup_{ \norm{\bv}_2 = 1, \norm{\bv}_0 \leq Cs}
\sup_{\btheta \in \Omega, \norm{\btheta}_0 \leq Cs}
\bigg |
\bv^T(\nabla^2R_\delta^n(\btheta) - \nabla R_\delta(\btheta))\bv
\bigg | \\
\lesssim &
\sqrt{\frac{s\log((ds\log d)^2/\delta)}{n\delta^3} } + \frac{s^2\log((ds\log d)^2/\delta)\log(d\lor n)}{n\delta^2}.
\eeq
Therefore, under the conditions of Proposition \ref{prop_verify_main}, we may guarantee that w.h.p.
\[
\sup_{ \norm{\bv}_2 = 1, \norm{\bv}_0 \leq Cs}
\sup_{\btheta \in \Omega, \norm{\btheta}_0 \leq Cs}
\bigg |
\bv^T(\nabla^2R_\delta^n(\btheta) - \nabla R_\delta(\btheta))\bv
\bigg | \leq \frac12 \tilde{K}\tilde{\phi}'\lambda_{\min}(\bSigma_{\bZ}),
\]
which implies that w.h.p.
\beq
\rho_{\min} \geq \frac12\tilde{K}\tilde{\phi'}\lambda_{\min}(\bSigma_{\bZ}) > 0,~~~
\rho_{\max} \leq  \frac52\tilde{K} \norm{\phi'}_\infty \lambda_{\max}(\bSigma_{\bZ}) < +\infty.
\eeq
This completes the proof.

\end{proof}

\subsection{Proof of Lemma \ref{prop_verify}}
\begin{proof}
	For any $\btheta \in \Omega$ and $\bv$ such that $\norm{\bv}_2 =1$ and $\norm{\bv}_0 \leq Cs$, with some algebra we can show that
	\beq
	\bv^T\nabla^2R_\delta (\btheta)\bv = -\frac{1}{\delta^2} \int (\bv ^T\bz)^2K'(\frac{x - \btheta^T\bz}{\delta})\bigg(
	\phi(x - \btheta^{*T}\bz - \mu) - \phi(x - \btheta^{*T}\bz + \mu)
	\bigg)f(\bz) dxd\bz,
	\eeq
	where $f$ is the p.d.f for ${\bZ}$.  With some change of variable, this can be further expressed as
	\beq
	-\frac{1}{\delta} \int (\bv ^T\bz)^2K'(t)\bigg(
	\phi(\delta t + \btheta^T\bz - \btheta^{*T}\bz - \mu) - \phi(\delta t + \btheta^T\bz - \btheta^{*T}\bz + \mu)
	\bigg)f(\bz) dtd\bz.
	\eeq
	By applying the mean value theorem on each of $\phi$ and the fact that $K'$ is an odd function, we finally obtain that
	\beq
	&\bv^T\nabla^2R_\delta (\btheta)\bv \\ =& -\int (\bv^T\bz)^2K'(t)t\bigg(
	\phi'(\tau \delta t + \btheta^T\bz - \btheta^{*T}\bz - \mu) + \phi'(-\tau'\delta t - \btheta^T\bz + \btheta^{*T}\bz - \mu)
	\bigg)f(\bz) dtd\bz,
	\eeq
	where $0\leq \tau,\tau' \leq 1$. Recall that $\tilde{K} = -\int K'(t)tdt > 0$.
	For the smoothness, since $\phi'$ and $\tilde{K}$ are bounded, we obtain that
	\beq
	\bv^T\nabla^2R_\delta (\btheta)\bv \leq 2\tilde{K} \norm{\phi'}_\infty \lambda_{\max}(\bSigma_{\bZ}).
	\eeq
	For the strong convexity, we apply a standard truncation argument. Consider the event
	$$A = \{ |\bz^T(\btheta - \btheta^*)| \leq T/(2R)\norm{\btheta - \btheta^*}_2  \}$$
	where $T$ is some constant. On this event, it follows by definition that $|\bz^T(\btheta - \btheta^*)|  \leq T.$
	This together with the condition (\ref{eq_prop_verify_0}) and $t\in[-1,1]$ by the choice of the kernel function can ensure that
	\[ \tau \delta t +\btheta^T\bz - \btheta^{*T}\bz - \mu,-\tau'\delta t -\btheta^T\bz + \btheta^{*T}\bz - \mu \in  [\frac{-3\mu}{2},\frac{-\mu}{2}].	
	\]
	
	Therefore we obtain
	\beq\label{eq_prop_verify_2}
	\bv^T\nabla^2R_\delta (\btheta)\bv
	&\geq 2\tilde{K}\tilde{\phi'}\EE\bigg[
	(\bv^T {\bz})^2\mathds{1}(A)
	\bigg] -  2\tilde{K} \norm{\phi'}_\infty \EE\bigg[
	(\bv^T {\bz})^2 \mathds{1}(A^c)
	\bigg]\\
	&\geq 2\tilde{K}(\tilde{\phi'}+ \norm{\phi'}_\infty)   \EE\bigg[
	(\bv^T {\bz})^2\mathds{1}(A)
	\bigg]-  2\tilde{K} \norm{\phi'}_\infty \EE\bigg[
	(\bv^T {\bz})^2
	\bigg]\\
	& = 2\tilde{K}\tilde{\phi'}\EE\bigg[
	(\bv^T {\bz})^2
	\bigg] - 2\tilde{K}(\tilde{\phi'}+ \norm{\phi'}_\infty) \bigg(
	\EE\bigg[
	(\bv^T {\bz})^2
	\bigg]  -\EE\bigg[
	(\bv^T {\bz})^2\mathds{1}(A)
	\bigg]
	\bigg).
	\eeq
	For the second term on the RHS of (\ref{eq_prop_verify_2}), we have
	\beq
	&\EE\bigg[
	\bv^T{\bz}{\bz}^T\bv
	\bigg]  -\EE\bigg[
	\bv^T{\bz}{\bz} ^T\bv\mathds{1}(A)
	\bigg]\\ \leq&  \EE\bigg[
	(\bv^T {\bz})^2\mathds{1}(|\bz^T(\btheta - \btheta^*)| > \frac{T}{2R}\norm{\btheta - \btheta^*}_2)
	\bigg]\\
	\leq & c\sigma^2\PP\Big(|\bz^T(\btheta - \btheta^*)| > \frac{T}{2R}\norm{\btheta - \btheta^*}_2\Big)\\
	\leq& c\sigma^2\exp\bigg(
	-\frac{c'T^2}{\sigma^2R^2}
	\bigg),
	\eeq
	where the second inequality follows from Cauchy-Schwarz inequality and that $\bZ$ is sub-Gaussian with parameter $\sigma^2$, and the last inequality follows from the sub-Gaussianity of $\bZ$ and $\norm{\btheta - \btheta^*}_2 \leq 2R$.
	Therefore, under the condition (\ref{eq_prop_verify_0}), we can ensure that
	\[
	2\tilde{K}(\tilde{\phi'}+ \norm{\phi'}_\infty)  \Big(
	\EE\big[
	(\bv ^T\bz)^2
	\big]  - \EE\big[
	(\bv ^T\bz)^2 \mathds{1}(A)
	\big]
	\Big) \leq \tilde{K}\tilde{\phi'}\lambda_{\min}(\bSigma_{\bZ}).
	\]
	This implies that
	\beq
	\bv^T\nabla^2R_\delta (\btheta)\bv\geq \tilde{K}\tilde{\phi'}\lambda_{\min}(\bSigma_{\bZ}).
	\eeq
	The proof is complete.
\end{proof}

\section{Example of different convergence rate from \texorpdfstring{\cite{ban2019}}{Lg}}\label{sup_ban}
In this section we provide a concrete example in which the soft margin condition is satisfied with $\alpha = 1$ while the smoothness parameter $\beta$ defined in Definition \ref{def_holder_new} can be arbitrarily large. Note that in this example the values of parameters are chosen for better exposition without loss of generality.

Consider $W = (X,Z) \in \RR^{d+1}$ where $X\sim N(0,\sigma^2)$, $X\independent Z\sim N_d(0,\sigma^2\II_d)$ and $W \independent \epsilon \sim N(0,(2\sigma)^2)$. 
Recall the binary response model takes the form
$
Y = \sign{\btheta^TW + \epsilon},
$
where we assume $\theta_1 = 1$ and $\norm{\btheta}_2 = 2$.
Direct calculation gives that
\beq
\eta(w):= \PP(Y =1|W = w) = \PP(\btheta^T W + \epsilon \geq 0 | W = w) = F_\epsilon (\btheta^Tw),
\eeq
where $F_{\epsilon}(\cdot)$ is the c.d.f of $\epsilon$. Thus
\beq
\PP_W(|\eta(W) - 0.5|\leq t) &= \PP_W(F_\epsilon^{-1}(0.5 - t) \leq \btheta^TW \leq F_\epsilon^{-1}(0.5 + t) )\\
&= F_{\btheta^T W}(F_\epsilon^{-1}(0.5 + t) ) - F_{\btheta^T W}(F_\epsilon^{-1}(0.5 - t) )\\
&=2t,
\eeq
where the last step follows from $\btheta^T\bW \stackrel{d}{=}\epsilon\sim N(0,4\sigma^2)$.
This implies that the soft margin condition is satisfied with $\alpha = 1$ and  not satisfied for any $\alpha > 1$. Therefore the fastest convergence rate for the estimator in \cite{ban2019} is $\big(
\frac{s\log(d/s)\log n}{n}
\big)^{1/3}$. However in this case, given a fixed smoothness $\beta > 1$ (without loss of generality suppose $\beta$ is an integer), we have
\beq
|f^{(\beta)}(x|y = 1,\bz)| &= \bigg | \frac{\sum^{\beta}_{k = 0} \binom{\beta}{k}F_\epsilon^{(k)}(x + \btheta_{-1}^T\bz)f^{(\beta - k )}_{X}(x) }{\PP(Y=1|\bz)}
\bigg |\\
&\leq
\frac{\sum^{\beta}_{k = 0} \binom{\beta}{k}\norm{H^{(k)}}_{\infty}\norm{H^{(\beta+1-k)}}_{\infty}
}{\sigma^{\beta+1}\PP(Y=1|\bz)},
\eeq
where $H^{(k)}$ is the $k$th derivative of the c.d.f. of standard normal distribution. By the boundedness of standard normal c.d.f, p.d.f and derivatives, we know that for any $L > 0$, with a large enough $\sigma$ we can always ensure that $|f^{(\beta)}(x|y = 1,\bz)|\leq \frac{L}{2\sigma\PP(Y=1|\bz)}$. With some algebra we can show that in this case

\beq
\bigg|
\int\bv^T\bz [f^{(\beta-1)}(\triangle + \btheta^T\bz|y = 1,\bz) - f^{(\beta-1)}(\btheta^T\bz|y = 1,\bz)] f(\bz|y = 1)d\bz
\bigg| \leq L\norm{\bv}_2|\triangle|.
\eeq
This holds similar when $Y = -1$, and thus Assumption \ref{assum_density} holds with smoothness parameter $\beta$. Since $\beta$ can be chosen arbitrarily large, Theorem \ref{theorem_main2} suggests that for this example, the $\ell_2$ rate of our estimator is  $\big(\frac{s \log d}{n}\big)^{\beta/(2\beta+1)}$ which can be much faster than \cite{ban2019}.

\section{Additional results on real data analysis}
\begin{figure}
    \centering
    \includegraphics[width = 15cm,height = 15cm]{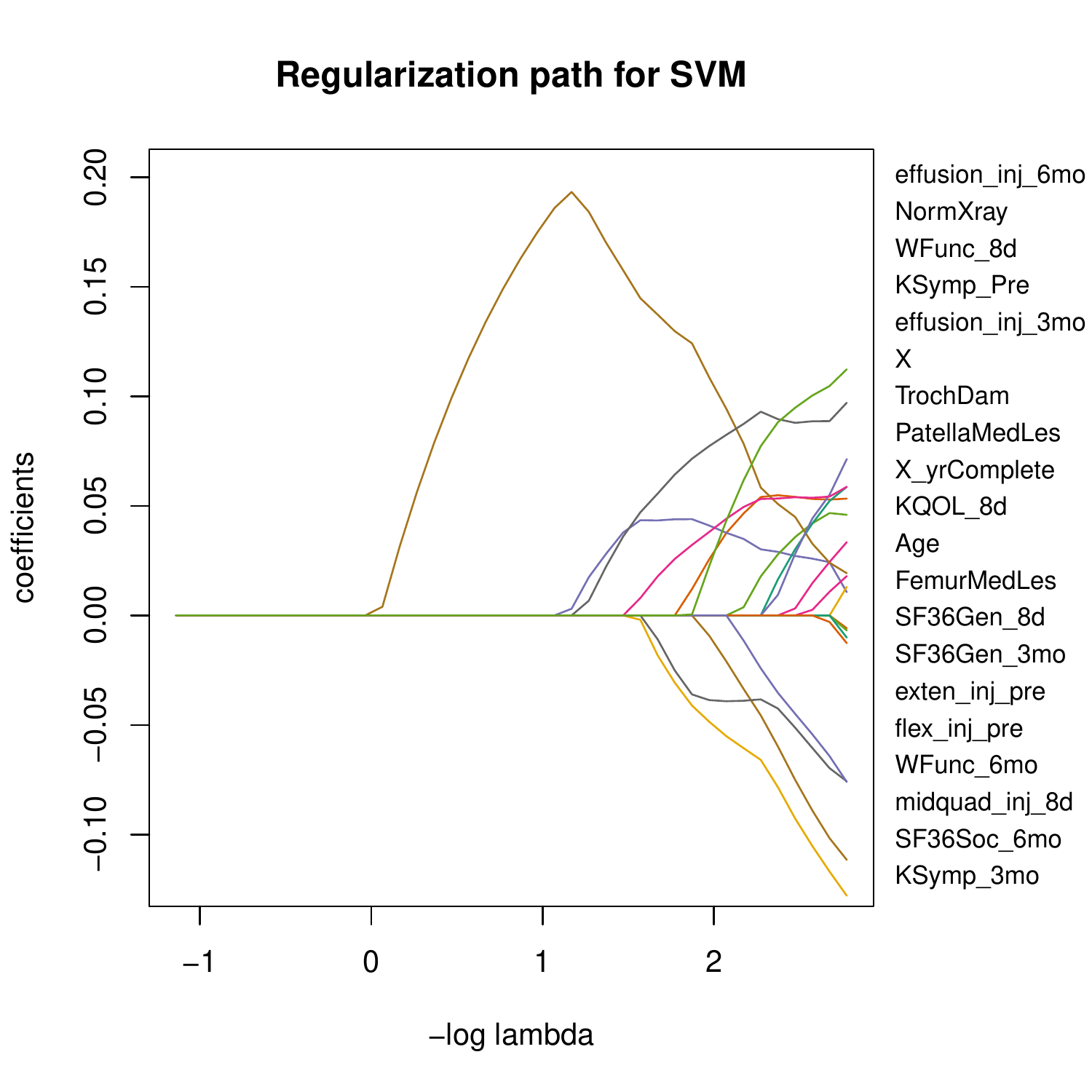}
    \caption{The regularization path of penalized SVM. }\label{fig_SVM}
\end{figure}
In this section we discuss additional results from the real data application. 
Figure \ref{fig_SVM} depicts the regularization path of $\ell_1$ penalized SVM. For better exposition we list the active variable names based on the order of values of coefficients at the last stage of the path. Compared to the proposed method and logistic regression, SVM yields very different output. This may be due to the fact that the estimand of SVM is intrinsically different. Nevertheless, similar to the proposed method and logistic regression, variables such as \textit{KSymp\_3mo} and \textit{ SF36Soc\_6mo} are all active with negative coefficients and the major measurement variable $X$ becomes active early with a positive sign.
\clearpage
\setlength{\bibsep}{0.85pt}
{
\bibliographystyle{ims}
\bibliography{ref}
}

\end{document}